\pdfoutput=1
\RequirePackage{ifpdf}
\ifpdf % We~are running pdfTeX in pdf mode
\documentclass[pdftex]{sigma}
\else
\documentclass{sigma}
\fi

\numberwithin{equation}{section}

\newtheorem{Theorem}{Theorem}[section]
\newtheorem*{Theorem*}{Theorem}
\newtheorem*{Question*}{Question}
\newtheorem{Corollary}[Theorem]{Corollary}
\newtheorem{Lemma}[Theorem]{Lemma}
\newtheorem{Proposition}[Theorem]{Proposition}

 { \theoremstyle{definition}
\newtheorem{Definition}[Theorem]{Definition}

\newtheorem{Example}[Theorem]{Example}
\newtheorem{Remark}[Theorem]{Remark} }

\def\H{\mathbb H}

\def\GL{\mathrm{GL}}

\def\SL{\mathrm{SL}}

\def\SU{\mathrm{SU}}
\def\sM{\mathfrak{M}}

\def\C{\mathbb C}

\def\R{\mathbb R}

\def\Z{\mathbb Z}
\def\mod{\ \mathrm{mod}\ }
\def\ord{\operatorname{ord}}

\def\LT{\operatorname{LT}}

\def\Im{\operatorname{Im}}
\def\gen#1{\langle #1\rangle}

\def\gauss#1{\left\lfloor #1\right\rfloor}
\def\SL{\mathrm{SL}}
\def\PSL{\mathrm{PSL}}
\def\ol#1{\overline{#1}}
\def\wt#1{\widetilde{#1}}
\def\M#1#2#3#4{\begin{pmatrix}#1&#2\\#3&#4\end{pmatrix}}
\def\SM#1#2#3#4{\left(\begin{smallmatrix}#1&#2\\#3&#4\end{smallmatrix}
 \right)}

\def\X #1,#2{X^{#1}_0(#2)}

\begin{document}
\allowdisplaybreaks

\newcommand{\arXivNumber}{2106.12438}

\renewcommand{\PaperNumber}{013}

\FirstPageHeading

\ShortArticleName{Modular Ordinary Differential Equations on $\mathrm{SL}(2,\mathbb{Z})$ of Third Order}

\ArticleName{Modular Ordinary Differential Equations on $\boldsymbol{\mathrm{SL}(2,\mathbb{Z})}$\\ of Third Order and Applications}

\Author{Zhijie CHEN~$^{\rm a}$, Chang-Shou LIN~$^{\rm b}$ and Yifan YANG~$^{\rm c}$}

\AuthorNameForHeading{Z.~Chen, C.-S.~Lin and Y.~Yang}

\Address{$^{\rm a)}$~Department of Mathematical Sciences, Yau Mathematical Sciences Center,\\
\hphantom{$^{\rm a)}$}~Tsinghua University, Beijing, 100084, China}
\EmailD{\href{mailto:zjchen2016@tsinghua.edu.cn}{zjchen2016@tsinghua.edu.cn}}

\Address{$^{\rm b)}$~Center for Advanced Study in Theoretical Sciences, National Taiwan University,\\
\hphantom{$^{\rm b)}$}~Taipei 10617, Taiwan}
\EmailD{\href{mailto:cslin@math.ntu.edu.tw}{cslin@math.ntu.edu.tw}}

\Address{$^{\rm c)}$~Department of Mathematics, National Taiwan University\\
\hphantom{$^{\rm c)}$}~and National Center for Theoretical Sciences, Taipei 10617, Taiwan}
\EmailD{\href{mailto:yangyifan@ntu.edu.tw}{yangyifan@ntu.edu.tw}}

\ArticleDates{Received June 24, 2021, in final form February 13, 2022; Published online February 22, 2022}

\Abstract{In this paper, we study third-order modular ordinary differential equations (MODE for short) of the following form $y'''+Q_2(z)y'+Q_3(z)y=0$, $z\in\mathbb{H}=\{z\in\mathbb{C} \,|\,\operatorname{Im}z>0 \}$, where $Q_2(z)$ and $Q_3(z)-\frac12 Q_2'(z)$ are meromorphic modular forms on ${\rm SL}(2,\mathbb{Z})$ of weight $4$ and $6$, respectively. We show that any quasimodular form of depth $2$ on ${\rm SL}(2,\mathbb{Z})$ leads to such a MODE. Conversely, we introduce the so-called Bol representation $\hat{\rho}\colon {\rm SL}(2,\mathbb{Z})\to{\rm SL}(3,\mathbb{C})$ for this MODE and give the necessary and sufficient condition for the irreducibility (resp.\ reducibility) of the representation. We show that the irreducibility yields the quasimodularity of some solution of this MODE, while the reducibility yields the modularity of all solutions and leads to solutions of certain ${\rm SU}(3)$ Toda systems. Note that the ${\rm SU}(N+1)$ Toda systems are the classical Pl\"ucker infinitesimal formulas for holomorphic maps from a~Riemann surface to $\mathbb{CP}^N$.}


\Classification{11F11; 34M03}

\section{Introduction}

Let $Ly=0$ be a Fuchsian ordinary differential equation of third order defined on the upper half plane $\mathbb{H}=\{z\in \mathbb{C} \,|\, \operatorname{Im}z>0\}$:
\begin{equation}\label{eq-01}
Ly:=y'''+Q_2(z)y'+Q_3(z)y=0,\qquad z\in\mathbb{H},
\end{equation}
where $':=\frac{{\rm d}}{{\rm d}z}$. Near a regular point $z_0$ of $Ly=0$, a local
solution $y(z)$ can be obtained by giving the initial values
$y^{(k)}(z_0)$, $k=0,1,2$, and then $y(z)$ could be globally defined
through analytic continuation. However, globally $y(z)$ might be
multi-valued. If $Ly=0$ is defined on~$\mathbb{C}$, then the monodromy
representation from $\pi_1(\mathbb{C}\setminus\{\text{singular
 points}\})$ to $\SL(3,\mathbb{C})$ is introduced to
characterize the multi-valueness of solutions. If the potentials
$Q_2(z)$ and $Q_3(z)$ are elliptic functions with periods $1$ and
$\tau$ ($\Im \tau>0$) and any solution of $Ly=0$ is single-valued and
meromorphic, then the monodromy representation reduces to a
homomorphism from $\pi_1(E_{\tau})$ to $\SL(3,\mathbb{C})$, where
$E_{\tau}:=\mathbb{C}/(\mathbb{Z}+\mathbb{Z}\tau)$ is the elliptic
curve. The well-known examples are the integral Lam\'{e} equations and
its generalizations; see, e.g., \cite{CKL1,CKL2,CKL3} for some recent
developments of this subject. Note that $\pi_1(E_{\tau})$ is
abelian. In this paper, we consider the case that $Ly=0$ is defined on
$\mathbb{H}$ and instead of the monodromy representation defined on
$\pi_1(\mathbb{H}\setminus\{\text{singular points}\})$, we study a new
representation defined on a discrete non-abelian group $\Gamma$ that
is related to the modular property of $\Gamma$ acting on
$\mathbb{H}$. It is called the Bol representation in this paper as in
\cite{LY} where the Bol representation was first introduced for second
order differential equations.

Let $\Gamma$ be a discrete subgroup of $\SL(2,\mathbb{R})$ that is commensurable with $\SL(2,\mathbb{Z})$.
Equation~\eqref{eq-01} is called a \emph{modular ordinary
 differential equation} (MODE for short) on $\Gamma$ if $Q_2(z)$ and
$Q_3(z)-\frac12 Q_2'(z)$ are meromorphic modular forms on $\Gamma$ of
weight $4$ and $6$, respectively.
Modular ordinary differential equations (of general order) appear
prominently in the study of rational conformal field theories
(see, e.g., \cite{Arike-Kaneko-Nagatomo-Sakai,
 Franc-Mason2,Gaberdiel-Keller,Hampapura-Mukhi,KNS,KS,
 Mathur-Mukhi-Sen,Zhu}). They provide a practical tool for
classifying rational conformal field theories. As an object in the
theory of modular forms, modular differential equations have also been
studied by mathematicians. See, for example,
\cite{Franc-Mason0,Franc-Mason1,Grabner,Kaneko-Koike-2003,KK-2006,Sebbar-Saber}.

Given a MODE \eqref{eq-01}, it is natural to ask whether there are
solutions satisfying some modular property. The main goal of this
paper is to study when MODE \eqref{eq-01} has solutions that lead to
modular forms or quasimodular forms. The approach is to calculate the
Bol representation, which will be explained below.

First we recall some basic notions from the ODE aspect.
Equation \eqref{eq-01} is called \emph{Fuchsian} if the order of any pole
of $Q_j(z)$ is at most $j$, $j=2,3$. At the
cusp $\infty$, we let $q_N={\rm e}^{2\pi {\rm i} z/N}$, where~$N$ is the width of
$\infty$ in $\Gamma$. Then $\frac{{\rm d}}{{\rm d}z}=\frac{2\pi {\rm i}
}{N}q_N\frac{{\rm d}}{{\rm d}q_N}$ and so \eqref{eq-01} becomes
\begin{equation}\label{eq-03}
\left(q_N\frac{{\rm d}}{{\rm d}q_N}\right)^3y+\left(\frac{N}{2\pi {\rm i}}\right)^2 Q_2(z)q_N\frac{{\rm d}}{{\rm d}q_N}y+\left(\frac{N}{2\pi {\rm i}}\right)^3Q_3(z)y=0.
\end{equation}
From here we see that
\[
\text{\it \eqref{eq-01} is Fuchsian at $\infty$ if and only if $Q_2(z)$ and $Q_3(z)$ are holomorphic at $\infty$},
\]
and similar conclusions hold for other cusps of $\Gamma$. By (\ref{eq-03}), the indicial equation at the cusp $\infty$ is given by
\[
\kappa^3+\Big(\frac{N}{2\pi {\rm i}}\Big)^2 Q_2(\infty)\kappa+\left(\frac{N}{2\pi {\rm i}}\right)^3Q_3 (\infty)=0,
\]
the roots of which are called the local exponents of (\ref{eq-01}) at $\infty$, denoted by $\kappa_{\infty}^{(1)}$, $\kappa_{\infty}^{(2)}$ and $\kappa_{\infty}^{(3)}$, satisfying $\sum_{j}\kappa_{\infty}^{(j)}=0$.
In this paper, we always assume that the exponent differences $\kappa_{\infty}^{(j)}-\kappa_{\infty}^{(1)}$ are \emph{integers} for $j=2,3$, Then $\sum_{j}\kappa_{\infty}^{(j)}=0$ implies
$\kappa_{\infty}^{(j)}\in\frac13\mathbb{Z}$ for all~$j$
and so we may assume $\kappa_{\infty}^{(1)}\leq \kappa_{\infty}^{(2)}\leq \kappa_{\infty}^{(3)}$. Similar assumptions are made for other cusps.

On the other hand, let $z_0\in\mathbb{H}$ be a singular point of (\ref{eq-01}) and write
\[Q_j(z)=A_j(z-z_0)^{-j}+O\big((z-z_0)^{-j+1}\big) \qquad\text{\it at }z_0,\]
then the indicial equation at $z_0$ is given by
\[
\kappa (\kappa-1)(\kappa-2)+A_2 \kappa+A_3=0,
\]
the roots of which are the local exponents of $(\ref{eq-01})$ at $z_0$, denoted by $\kappa_{z_0}^{(1)}$, $\kappa_{z_0}^{(2)}$ and $\kappa_{z_0}^{(3)}$, satisfying $\sum_{j}\kappa_{z_0}^{(j)}=3$. In this paper, we always assume that the exponent differences $\kappa_{z_0}^{(j)}-\kappa_{z_0}^{(1)}$ are \emph{integers} for $j=2,3$. Then $\sum_{j}\kappa_{z_0}^{(j)}=3$ implies
$\kappa_{z_0}^{(j)}\in\frac13\mathbb{Z}$ for all~$j$
and so we may assume $\kappa_{z_0}^{(1)}\leq\kappa_{z_0}^{(2)}\leq\kappa_{z_0}^{(3)}$.
Since the exponent differences are integers, (\ref{eq-01}) might have solutions with logarithmic singularities at $z_0$. See Appendix~\ref{section-3} for all possibilities of the solution structure of (\ref{eq-01}) at~$z_0$. The singularity~$z_0$ is called \emph{apparent} if~\eqref{eq-01} has no solutions with logarithmic singularities at~$z_0$. In this case, the three local exponents must be distinct, i.e., $\kappa_{z_0}^{(1)}<\kappa_{z_0}^{(2)}<\kappa_{z_0}^{(3)}$; see, e.g., Appendix~\ref{section-3}. In this paper, we always assume that~$L$ is \emph{apparent at any singularity $z_0\in \mathbb{H}$}. More precisely, we assume that the MODE~(\ref{eq-01}) satisfies
\begin{itemize}\itemsep=0pt
\item[(H1)] The MODE (\ref{eq-01}) is Fuchsian on $\mathbb{H}\cup\{\text{cusps}\}$;
\item[(H2)] At any singular point $z_0\in\mathbb{H}$,
 $\kappa_{z_0}^{(1)}<\kappa_{z_0}^{(2)}<\kappa_{z_0}^{(3)}$ satisfy
 $\kappa_{z_0}^{(1)}\in\frac13\mathbb{Z}_{\leq 0}$ and
 $\kappa_{z_0}^{(j)}-\kappa_{z_0}^{(1)}\in\mathbb{Z}$ for
 $j=2,3$. Furthermore, $z_0$ is apparent.
\item[(H3)] At any cusp $s$ of $\Gamma$,
 $\kappa_{s}^{(1)}\leq\kappa_{s}^{(2)}\leq \kappa_{s}^{(3)}$
 satisfies $\kappa_{s}^{(1)}\in\frac13\mathbb{Z}_{\leq 0}$ and
 $\kappa_{s}^{(j)}-\kappa_{s}^{(1)}\in\mathbb{Z}$ for $j=2,3$.
\end{itemize}
The motivation of all these assumptions will be clear from Theorem~\ref{thm-02} below.

Note $\ln q_N=\frac{2\pi {\rm i}}{N}z$. Under our assumption (H3),
(\ref{eq-01}) might have solutions containing $(\ln
q_N)^2=-\frac{4\pi^2}{N^2}z^2$ terms; see Remark
\ref{rmk-apparent}. In this case, we call the cusp $\infty$ to be
\emph{completely not apparent} or \emph{maximally unipotent},
because under the Bol representation $\hat{\rho}$ that will be
introduced below, the corresponding matrix $\hat{\rho}(T)\in \SL(3,\mathbb{C})$ of
$T=\SM{1}{N}{0}{1}\in\Gamma$ has eigenvalues $\{1,1,1\}$ but
$\operatorname{rank}(\hat{\rho}(T)-I_3)=2$, i.e., $\hat{\rho}(T)$ is maximally unipotent. Here $I_3$ denotes the $3\times 3$
identity matrix.

One class of the MODEs can be derived from quasimodular forms of depth
$2$. The notion of quasimodular forms was first introduced by Kaneko
and Zagier~\cite{KZ}.
See Section~\ref{section-2} for a~brief overview of basic properties of quasimodular forms. In particular, given a holomorphic function~$\phi(z)$ satisfying
\begin{equation*}%\label{eqe-05}
(\phi\big|_{2}\gamma)(z):=(cz+d)^{-2}\phi(\gamma z)=\phi(z)+\frac{\alpha c}{cz+d}
\end{equation*}
for all $\gamma=\SM{a}{b}{c}{d}$ for some nonzero complex number $\alpha$, any quasimodular form $f(z)$ of weight $k$ and depth $2$ with character $\chi$ can be expressed as
\[f(z)=f_0(z)+f_1(z)\phi(z)+f_2(z)\phi(z)^2,\]
where $f_j(z)$ is a modular form on $\Gamma$ of weight $k-2j$ with character $\chi$ and $f_2\neq 0$. Consider
\begin{align}\label{0eq-22}
\begin{pmatrix}
h_1(z)\\
h_2(z)\\
h_3(z)
\end{pmatrix}:=\begin{pmatrix}
z^2 f(z)+\alpha z(f_1(z)+2f_2(z)\phi(z))+\alpha^2f_2(z)\\
2 z f(z)+\alpha(f_1(z)+2f_2(z)\phi(z))\\
f(z)
\end{pmatrix}
\end{align}
and define
\begin{equation}\label{wfz}W_f(z):=\det\begin{pmatrix}
h_1& h_1'& h_1''\\
h_2& h_2'& h_2''\\
h_3& h_3'& h_3''
\end{pmatrix}\end{equation}
to be the Wronskian associated to $f$.
Then
$W_f(z)$ is a modular form on $\Gamma$ of weight $3k$ with character $\chi^3$; see Lemma \ref{lemma-2-1} for a proof. This $W_f(z)$ was first introduced by Pellarin \cite{PN} for $\Gamma=\SL(2,\mathbb{Z})$ and $\phi(z)=E_2(z)$.

Now we define $g_j(z):=\frac{h_j(z)}{\sqrt[3]{W_f(z)}}$, then
\[\det\begin{pmatrix}
g_1& g_1'& g_1''\\
g_2& g_2'& g_2''\\
g_3& g_3'& g_3''
\end{pmatrix}=1,\]
and a further differentiation leads to
\begin{equation}\label{eq-171}\det\begin{pmatrix}
g_1& g_1'& g_1'''\\
g_2& g_2'& g_2'''\\
g_3& g_3'& g_3'''
\end{pmatrix}=0,\end{equation}
so $g_3(z)$ is a solution of (\ref{eq-01}) with
\begin{equation}\label{0eq-7}Q_2(z):=\frac{g_1'''g_2-g_1g_2'''}{g_1g_2'-g_1'g_2},\qquad Q_3(z):=\frac{g_1'g_2'''-g_2'g_1'''}{g_1g_2'-g_1'g_2}.\end{equation}
It is easy to see that $Q_2(z)$ and $Q_3(z)$ are single-valued, and $g_1$, $g_2$ are also solutions of~(\ref{eq-01}). Our first result reads as follows.

\begin{Theorem}\label{thm-02}
Let $Q_2(z)$ and $Q_3(z)$ be given by \eqref{0eq-7}. Then
\begin{itemize}\itemsep=0pt
\item[$(1)$] \eqref{eq-01} is a MODE, i.e., $Q_2(z)$ and $Q_3(z)-\frac12
 Q_2'(z)$ are meromorphic modular forms on $\Gamma$ $($with trivial
 character$)$ of weight $4$ and $6$, respectively.
\item[$(2)$] $(H1)$--$(H3)$ hold for \eqref{eq-01}.
\end{itemize}
Furthermore, for $\Gamma=\SL(2,\mathbb{Z})$, we have that
\begin{itemize}\itemsep=0pt
\item[$(3)$] At the elliptic point ${\rm i}$, $\big\{3\kappa_{\rm i}^{(1)},3\kappa_{\rm i}^{(2)},
 3\kappa_{\rm i}^{(3)}\big\}\equiv\{0,0,1\} \mod 2$.
\item[$(4)$] At the elliptic point $\rho=\frac{-1+\sqrt{3}i}{2}$, $\kappa_{\rho}^{(j)}\in\mathbb{Z}$ for all $j$ and $\big\{\kappa_{\rho}^{(1)},\kappa_{\rho}^{(2)},
 \kappa_{\rho}^{(3)}\big\}\equiv\{0,1,2\} \mod 3$.
\end{itemize}
\end{Theorem}

We emphasize that (1) and (2) in Theorem~\ref{thm-02} hold for any~$\Gamma$, not only for $\Gamma=\SL(2,\mathbb{Z})$. They will lay the
ground for our future study of general MODEs on other congruence
subgroups.

As an example, in Section~\ref{section-extremal}, we will work out the MODE in the case
$f(z)$ is an extremal quasimodular form on $\SL(2,\Z)$, introduced
first in~\cite{KK-2006}; see Theorem~\ref{thm-2}.

Now we introduce the notion of the Bol representation of $\Gamma$
associated to the MODE (\ref{eq-01}), which was first introduced in~\cite{LY} for second order MODEs. It is well known that any (local)
solution $y(z)$ of (\ref{eq-01}) can be extended to a multi-valued function in
$\mathbb{H}$ through analytic continuation. Fix a point
$z_0\in\mathbb{H}$ that is not a singular point of~(\ref{eq-01}) and
let $U$ be a simply-connected neighborhood of $z_0$ that contains no
singularities of (\ref{eq-01}). For $\gamma=\SM{a}{b}{c}{d}\in\Gamma$,
choose a~path~$\sigma$ from~$z_0$ to $\gamma z_0$ and consider the
analytic continuation of $y(z)$, $z\in U$, along the path. Then
$y(\gamma z)$ is well-defined in $U$. Define
\[(y|_{-2}\gamma)(z):=(cz+d)^2y(\gamma z),\qquad z\in U,\]
then by a direct computation or by using Bol's identity \cite{Bol}, we see that
$(y|_{-2}\gamma)(z)$ is also a~solution of (\ref{eq-01}). Thus, given
a~fundamental system of solutions $Y(z)=(y_1(z),y_2(z),y_3(z))^t$,
there is $\hat{\gamma}\in \SL(3,\mathbb{C})$ such that
\[(Y\big|_{-2}\gamma)(z)=\hat{\gamma}Y(z),\]
where the fact $\det \hat{\gamma}=1$ follows from that the Wronskians
of $Y$ and $(Y\big|_{-2}\gamma)$ are the same. Obviously, this matrix
$\hat{\gamma}$ depends on the choice of the path $\sigma$. However,
under the above assumptions, all local monodromy matrices are
$\varepsilon I_3$ with $\varepsilon^3=1$, so different choices of
$\sigma$ will only possibly change $\hat{\gamma}$ to ${\rm e}^{\pm\frac{2\pi {\rm i}}{3}}\hat{\gamma}$. From here, we see that there is a
well-defined homomorphism $\rho\colon \Gamma\to
\PSL(3,\mathbb{C})$ such that
\[(Y\big|_{-2}\gamma)(z)={\rm e}^{\frac{2\pi {\rm i} k}{3}}\rho(\gamma)Y(z),\qquad k\in\{0,\pm 1\},\]
where $y_j(\gamma z)$ are always understood to take analytic
continuation along the same path for $j=1,2,3$. This homomorphism
$\rho$ will be called the \emph{Bol representation} as in
\cite{LY}. For the convenience of computations, it is better to lift
$\rho$ to a homomorphism $\hat{\rho}\colon \Gamma\to {\rm GL}(3,\mathbb{C})$ as
follows. Suppose that we can find a multi-valued meromorphic function
$F(z)$ such that: (i)~The analytic continuation of
$\hat{y}(z):=F(z)y(z)$, where $y(z)$ is any solution of (\ref{eq-01}),
gives rise to a~single-valued holomorphic function on $\mathbb{H}$,
and (ii)~$F(z)^3$ is a modular form on~$\Gamma$ of weight~$3k$ with
some character, where $k\in\mathbb{N}$. Such~$F(z)$ can be constructed
explicitly when $\Gamma$ is a triangle group. Then by letting
$\hat{Y}(z):=F(z)Y(z)$, there is $\hat{\rho}(\gamma)\in
{\rm GL}(3,\mathbb{C})$ such that
\[\big(\hat{Y}\big|_{\ell}\gamma\big)(z)=\hat{\rho}(\gamma)\hat{Y}(z),\qquad
 \text{where }\ell=k-2.\]
This homomorphism $\hat{\rho}\colon \Gamma\to {\rm GL}(3,\mathbb{C})$, as a
lift of $\rho$, will also be called the \emph{Bol representation}
since there is no confusion arising.
Naturally we consider the following problem:

\begin{Question*}
Can we characterize, in terms of local exponents, the MODEs
 \eqref{eq-01} whose Bol representations are irreducible?
\end{Question*}

One purpose of this paper is to answer this question for the case
$\Gamma=\SL(2,\mathbb{Z})$. For $\Gamma=\SL(2,\mathbb{Z})$, the above
$F(z)$ can be taken to be
\begin{equation}\label{0eq-5-5}
F(z):=\Delta(z)^{-\kappa_{\infty}^{(1)}}E_4(z)^{-\kappa_{\rho}^{(1)}}E_6(z)^{-\kappa_{\rm i}^{(1)}}
\prod_{j=1}^m F_j(z)^{-\kappa_{z_j}^{(1)}},
\end{equation}
where
\begin{equation} \label{equation: E4, E6}
E_4(z)=1+240\sum_{n=1}^\infty\frac{n^3q^n}{1-q^n}, \qquad
E_6(z)=1-504\sum_{n=1}^\infty\frac{n^5q^n}{1-q^n}, \qquad q={\rm e}^{2\pi {\rm i}z},
\end{equation}
are the Eisenstein series of weight $4$ and $6$, respectively,
\[
\Delta(z)=\frac{E_4(z)^3-E_6(z)^2}{1728}=q-24q^2+252q^3-1472q^4+\cdots,
\]
${\rm i}=\sqrt{-1}$ and $\rho=\big({-}1+\sqrt{3}{\rm i}\big)/2$ are the elliptic points of
$\SL(2,\mathbb{Z})$, $\{z_1,\dots, z_m\}$ $\sqcup\{{\rm i},\rho,\infty\}$
denotes the set of singular points of the MODE (\ref{eq-01}) mod
$\SL(2,\mathbb{Z})$, $t_j:=E_4(z_j)^3/E_6(z_j)^2$ and
$F_j(z):=E_4(z)^3-t_j E_6(z)^2$. Then $F(z)^3$ is a modular form of
weight $3(\ell+2)$, where the integer $\ell=k-2$ is given by
\begin{equation*}
\ell:=-2-12\kappa_{\infty}^{(1)}-4\kappa_{\rho}^{(1)}-6\kappa_{\rm i}^{(1)}-12\sum_{j=1}^m \kappa_{z_j}^{(1)}.
\end{equation*}
 In other words, besides the assumptions (H1)--(H3), we need to assume
 further that $\kappa_{\rho}^{(1)}\in\mathbb{Z}$ such that
 $\ell\in\mathbb{Z}$. Consequently, we will see from Lemma~\ref{lemm-3-1} that the Bol representation $\hat{\rho}$ is indeed a group homomorphism from $\SL(2,\mathbb{Z})$ to $\SL(3,\mathbb{C})$.

\begin{Remark} The choice of $F(z)$ is not unique since we can
 multiply $F(z)$ by a holomorphic modular form to obtain a new
 one. Different choices of $F(z)$'s may give different weights $k$
 (and so $\ell$) but keeping $\hat{\rho}(\gamma)$ invariant. For
 example, when the MODE~(\ref{eq-01}) comes from a quasimodular form
 $f(z)$ of depth $2$ on~$\Gamma$ as shown in Theorem~\ref{thm-02},
 then one choice is to take $\sqrt[3]{W_f(z)}$ as~$F(z)$,
 i.e., $\hat{Y}=(h_1,h_2,h_3)^t$ defined in~(\ref{0eq-22}). Note that
 for $\Gamma=\SL(2,\mathbb{Z})$, $\sqrt[3]{W_f(z)}$ might be different
 from the~$F(z)$ given by~(\ref{0eq-5-5}). To obtain that
 $\sqrt[3]{W_f(z)}$ equals to the $F(z)$ given by~(\ref{0eq-5-5}), we
 need to assume that $f_0$, $f_1$, $f_2$ have no common zeros.
\end{Remark}

Note from $\sum_{j}\kappa_{\rm i}^{(j)}=3$ that we have either
 $\big\{3\kappa_{\rm i}^{(1)},3\kappa_{\rm i}^{(2)},
 3\kappa_{\rm i}^{(3)}\big\}\equiv\{0,0,1\} \mod 2$ or
 $\big\{3\kappa_{\rm i}^{(1)},\allowbreak 3\kappa_{\rm i}^{(2)},
 3\kappa_{\rm i}^{(3)}\big\}\equiv\{1,1,1\} \mod 2$. Our second result of this
 paper is

\begin{Theorem}\label{thm-01} Let $\Gamma=\SL(2,\mathbb{Z})$ and
 suppose that the MODE \eqref{eq-01} satisfies $(H1)$--$(H3)$ and
 $\kappa_{\rho}^{(1)}\in\mathbb{Z}$. Then the Bol representation
 $\hat\rho$ is irreducible if and only if
 $\big\{3\kappa_{\rm i}^{(1)},3\kappa_{\rm i}^{(2)}$,
 $3\kappa_{\rm i}^{(3)}\big\}\equiv\{0,0,1\} \mod 2$.
\end{Theorem}

Note that all irreducible representations of $\SL(2,\Z)$ of rank up to
$5$ have been classified by Tuba and Wenzl~\cite{Tuba-Wenzl}.\footnote{We thank the referee for providing the
 reference.} One may use their results, the work of Westbury~\cite{Westbury}, and Lemma~\ref{lem-3-9} below to give another
proof of Theorem~\ref{thm-01} different from that given in Section~\ref{section-5}. See Remark~\ref{remark: Tuba}.

As an application of Theorem \ref{thm-01}, we can show that the converse statement of Theorem \ref{thm-02} holds. More precisely, we have

\begin{Theorem}\label{thm-01+} Let $\Gamma=\SL(2,\mathbb{Z})$ and
 suppose that the MODE \eqref{eq-01} satisfies $(H1)$--$(H3)$ and
 $\kappa_{\rho}^{(1)}\in\mathbb{Z}$. Let $y_+(z)$ be the solution of~\eqref{eq-01} of the form
 $y_+(z)=q^{\kappa_\infty^{(3)}}\sum_{j=0}^\infty c_jq^j$, $c_0=1$,
 and~$F(z)$ be defined by~\eqref{0eq-5-5}.
\begin{itemize}\itemsep=0pt
\item[$(1)$] If $\big\{3\kappa_{\rm i}^{(1)},3\kappa_{\rm i}^{(2)},
 3\kappa_{\rm i}^{(3)}\big\}\equiv\{0,0,1\} \mod 2$, then $\hat{y}_+(z):=F(z)y_+(z)$ is a quasimodular form of weight $\ell+2$ and depth~$2$.
\item[$(2)$] If $\big\{3\kappa_{\rm i}^{(1)},3\kappa_{\rm i}^{(2)},
 3\kappa_{\rm i}^{(3)}\big\}\equiv\{1,1,1\} \mod 2$, then the Bol representation $\hat\rho$ is trivial, i.e., $\hat{\rho}(\gamma)=I_3$ for all $\gamma\in \SL(2,\mathbb{Z})$. In particular, $12 |\ell$ and $\hat{y}(z):=F(z)y(z)$ is a modular form of weight $\ell$ for any solution~$y(z)$ of~\eqref{eq-01}.
\end{itemize}
\end{Theorem}

Together with Theorems \ref{thm-01+} and \ref{thm-02}, we can obtain

\begin{Corollary}\label{coro-1-3}
Let $\Gamma=\SL(2,\mathbb{Z})$ and suppose $(H1)$--$(H3)$ and
$\kappa_{\rho}^{(1)}\!\in\mathbb{Z}$ hold for the MODE~\eqref{eq-01}. Then~$(3)$ and~$(4)$ in Theorem~{\rm \ref{thm-02}} are
equivalent.
\end{Corollary}

In the reducible case, Theorem \ref{thm-01+}(2) can be applied to
construct solutions of the $\SU(3)$ Toda system. See Section~\ref{section-6} for the precise statement. The Toda system is an important integrable system in mathematical physics. In algebraic geometry, the~$\SU(N+1)$ Toda system is exactly the classical infinitesimal Pl\"{u}cker
formula associated with holomorphic maps from Riemann surfaces to
$\mathbb{CP}^N$; see, e.g., \cite{CL-JDG,LNW,LWY} and references therein for
the recent development of the Toda system.

The rest of this paper is organized as follows. In Section~\ref{section-2}, we give the proof of Theorem~\ref{thm-02}, namely we
will prove that every quasimodular form of depth~$2$ leads to a MODE
(\ref{eq-01}) satisfying the conditions (H1)--(H3). We focus on the
case $\Gamma=\SL(2,\mathbb{Z})$ from Section~\ref{section-5}. Theorems
\ref{thm-01}--\ref{thm-01+} and Corollary~\ref{coro-1-3} will be proved
in Section \ref{section-5}. In Section~\ref{section-6}, we discuss the reducible case and prove the converse
statement of Theorem~\ref{thm-01+}(2). We also give an application to
the $\SU(3)$ Toda system. In Section~\ref{section-4}, we discuss the
criterion on the existence of the MODE~(\ref{eq-01}) which is Fuchsian and
apparent throughout~$\H$ with prescribed local exponents at
singularities and at cusps. In Section~\ref{section-extremal}, as
examples of MODEs, we will work out the explicit expressions of
$Q_j(z)$'s for an extremal quasimodular form~$f(z)$. Finally in Appendix~\ref{section-3}, we recall the theory
of the solution structure of third order ODEs at a regular singular
point.

\section[Quasimodular forms of depth 2 and its associated 3rd order MODE]{Quasimodular forms of depth 2\\ and its associated 3rd order MODE}\label{section-2}

The main purpose of this section is to prove Theorem~\ref{thm-02}.
Let $\Gamma$ be a discrete subgroup of $\SL(2,\mathbb{R})$ that is
commensurable with $\SL(2,\mathbb{Z})$ and $\chi\colon \Gamma\to
\mathbb{C}^{\times}$ be a character of~$\Gamma$ of finite order. A~holomorphic function $f(z)$ defined on the upper half plane
$\mathbb{H}$ is a modular form of weight~$k$ with character~$\chi$ if
the following conditions hold:
\begin{itemize}\itemsep=0pt
\item[(1)] $(f\big|_{k}\gamma)(z):=(cz+d)^{-k}f(\gamma z)=\chi(\gamma)f(z)$ for any $\gamma\in\SM{a}{b}{c}{d}\in\Gamma$;
\item[(2)] $f$ is holomorphic at any cusp $s$ of~$\Gamma$.
\end{itemize}
For example, the Eisenstein series $E_4(z)$ and $E_6(z)$ in
\eqref{equation: E4, E6} are modular forms of weight~$4$ and~$6$ on
$\SL(2,\Z)$, respectively. We let $\sM_k(\Gamma, \chi)$ denote the
space of modular forms of weight $k$ with character~$\chi$ on~$\Gamma$. For example, for $\Gamma=\SL(2,\mathbb{Z})$, $\infty$ is the
only cusp and we assume that the character is trivial,
i.e., $\chi\equiv 1$. Then condition~(1) implies $f(z+1)=f(z)$, which
implies that~$f(z)$ can be viewed as a function of $q={\rm e}^{2\pi {\rm i} z}$,
and condition~(2) just means that~$f$ is holomorphic at $q=0$.

The notion of quasimodular forms was introduced by Kaneko and Zagier~\cite{KZ}. Originally, they are defined as the holomorphic parts
of nearly holomorphic modular forms.
For our purpose, it suffices to know that a holomorphic function~$f(z)$ is a quasimodular form of weight~$k$ and depth~$r$
with character $\chi$ on $\Gamma$ if and only if~$f(z)$ can be expressed as
\[f(z)=\sum_{j=0}^r f_j(z)\phi(z)^j,\]
where $f_j(z)\in\sM_{k-2j}(\Gamma, \chi)$ with $f_r\not\equiv0$ and $\phi(z)$ is a holomorphic function satisfying that
\begin{equation}\label{eqe-5}(\phi\big|_{2}\gamma)(z):=(cz+d)^{-2}\phi(\gamma z)=\phi(z)+\frac{\alpha c}{cz+d}\end{equation}
for all $\gamma=\SM{a}{b}{c}{d}\in\Gamma$ for some nonzero complex number
$\alpha$ and $\phi(z)$ is holomorphic at cusps of~$\Gamma$. This
$\phi(z)$ is called a quasimodular form of weight~$2$ and depth~$1$ on~$\Gamma$. For example, if $\Gamma$ is a subgroup of
$\SL(2,\mathbb{Z})$, we can always let
\[
\phi(z)=E_2(z):=\frac1{2\pi {\rm i}}\frac{\Delta'(z)}{\Delta(z)}
=1-24\sum_{n=1}^\infty\frac{nq^n}{1-q^n},
\]
and so
$\alpha=\frac{6}{\pi {\rm i}}$. We let $\widetilde{\sM}_k^{\leq
 r}(\Gamma,\chi)$ denote the space of quasimodular forms of weight
$k$ and depth $\leq r$ with character $\chi$.
 One basic property is that the quasi-modularity is invariant under
 the differentiation, namely if $f(z)\in \widetilde{\sM}_k^{\leq
 r}(\Gamma,\chi)$, then $f'(z)\in \widetilde{\sM}_{k+2}^{\leq
 r+1}(\Gamma,\chi)$; see \cite[Proposition 20]{Z}. We refer the
 reader to \cite{Choie-Lee,KZ,Z} for the general theory of quasimodular forms.

Now we consider the Wronskian $W_f(z)$ (see \cite{PN}) associated to
\[
f(z)=f_0(z)+f_1(z)\phi(z)+f_2(z)\phi(z)^2\in\widetilde{\sM}_{k}^{\leq
 2}(\Gamma,\chi),
 \]
 where $f_j(z)\in\sM_{k-2j}(\Gamma, \chi)$ with
$f_2\neq 0$. Then
\begin{align}\label{eqe-1}\big(f\big|_{k}\gamma\big)(z):={} &(cz+d)^{-k}f(\gamma z)\\
={} &\chi(\gamma)\sum_{j=0}^2f_j(z)\left(\phi(z)+\frac{\alpha c}{cz+d}\right)^j,\qquad \gamma=\begin{pmatrix}
a & b\\
c & d
\end{pmatrix}
\in \Gamma.\nonumber
\end{align}
As in \cite{PN}, we set
\begin{gather}
P_{f}(t):=\sum_{j=0}^2f_j(z)\left(\phi(z)+\alpha t\right)^j,\nonumber\\
Q_f(t):=t^2P_f(1/t)=f(z)t^2+\alpha(f_1(z)+2f_2(z)\phi(z))t+\alpha^2f_2(z),\nonumber\\
F_f(z):=\begin{pmatrix}
Q_f(t)\\
\dfrac{\partial}{\partial t}Q_f(t)\vspace{1mm}\\
\dfrac12\dfrac{\partial^2}{\partial t^2}Q_f(t)
\end{pmatrix}_{|{t=z}}\nonumber\\
\hphantom{F_f(z)~}{} =\begin{pmatrix}
z^2 f(z)+\alpha z(f_1(z)+2f_2(z)\phi(z))+\alpha^2f_2(z)\\
2 z f(z)+\alpha(f_1(z)+2f_2(z)\phi(z))\\
f(z)
\end{pmatrix}=:\begin{pmatrix}
h_1(z)\\
h_2(z)\\
h_3(z)
\end{pmatrix}.\label{eq-22}
\end{gather}
and define $W_f(z)$ as in \eqref{wfz}
to be the Wronskian associated to $f$.

\begin{Lemma}\label{lemma-2-1}
$W_f(z)$ is a modular form on $\Gamma$ of weight $3k$ with character $\chi^3$.
\end{Lemma}

\begin{proof} For $\Gamma=\SL(2,\mathbb{Z})$, this result is proved in \cite{PN} and can be also derived from Mason \cite[Lemma~3.1]{Mason}, and the approaches in \cite{Mason,PN} can be easily applied to general discrete subgroups~$\Gamma$. Here we provide an elementary proof for general $\Gamma$ for completeness.
By~(\ref{eqe-1}) and $ad-bc=1$ we have
\begin{gather*}%\label{eqe-3}
f(\gamma z) =\chi(\gamma)(cz+d)^{k}\left[f+\frac{\alpha c}{cz+d}(f_1+2 f_2\phi)+\frac{\alpha^2 c^2}{(cz+d)^2}f_2\right]\nonumber\\
\hphantom{f(\gamma z)}{}
=\chi(\gamma)(cz+d)^{k-2}\big(c^2h_1+cdh_2+d^2h_3\big),
\\
h_2(\gamma z)= 2 \frac{az+b}{cz+d} f(\gamma z)+\alpha(f_1(\gamma z)+2f_2(\gamma z)\phi(\gamma z))\\
\hphantom{h_2(\gamma z)}{}
= \chi(\gamma)(cz+d)^{k-2}\bigg[2\frac{az+b}{cz+d}((cz+d)^2f+\alpha c(cz+d)(f_1+2 f_2\phi)\\
\hphantom{h_2(\gamma z)=}{} +\alpha^2 c^2f_2)+\alpha(f_1+2f_2\phi)+\frac{2\alpha^2 c}{cz+d}f_2\bigg]\\
\hphantom{h_2(\gamma z)}{}
=\chi(\gamma)(cz+d)^{k-2}[2ac h_1+(ad+bc)h_2+2bd h_3],
\\
h_1(\gamma z)=
\frac{(az+b)^2}{(cz+d)^2} f(\gamma z)+\alpha \frac{az+b}{cz+d}(f_1(\gamma z)+2f_2(\gamma z)\phi(\gamma z))+\alpha^2f_2(\gamma z)\\
\hphantom{h_1(\gamma z)}{}
=\chi(\gamma)(cz+d)^{k-2}\bigg[\frac{(az+b)^2}{(cz+d)^2}\big((cz+d)^2f+\alpha c(cz+d)(f_1+2 f_2\phi)\\
\hphantom{h_1(\gamma z)=}{}+\alpha^2 c^2f_2\big)+\frac{az+b}{cz+d}\Big[\alpha(f_1+2f_2\phi)+\frac{2\alpha^2 c}{cz+d}f_2\Big]+\frac{\alpha^2}{(cz+d)^2}f_2\bigg]\\
\hphantom{h_1(\gamma z)}{}
=\chi(\gamma)(cz+d)^{k-2}\big[a^2h_1+abh_2+b^2h_3\big].
\end{gather*}
Thus
\[F_f(\gamma z)=\chi(\gamma)(cz+d)^{k-2}AF_f(z),\qquad \text{where} \quad A:=\begin{pmatrix}a^2&ab&b^2\\ 2ac&ad+bc&2bd\\c^2&cd&d^2
\end{pmatrix}.\]
Together with the fact $\det A=(ad-bc)^3=1$, it is easy to see that
\begin{gather*}W_f(\gamma z) =\det(F_f(\gamma z), F_f'(\gamma z), F_f''(\gamma z))=\chi(\gamma)^3 (cz+d)^{3k}W_f(z).\end{gather*}
In particular,
for the case $\gamma=\SM1N01$, where $N$ is the width of the cusp $\infty$, the
transformation law shows that~$W_f(z)$ is a polynomial only in $f_0$, $f_1$, $f_2$, $\phi$,
and their derivatives. Thus, by the computation above and the fact
that the ring of quasimodular forms is invariant under
differentiation, $W_f(z)$ is a quasimodular form that is actually
modular. In other words, $W_f(z)\in\sM_{3k}\big(\Gamma,\chi^3\big)$. This completes the proof.
\end{proof}

Now we define $g_j(z):=h_j(z)/\sqrt[3]{W_f(z)}$. Then~(\ref{eq-171}) holds,
so $g_3(z)$ is a solution of
\begin{equation}\label{eq-1}
Ly:=y'''+Q_2(z)y'+Q_3(z)y=0,
\end{equation}
where
\begin{equation}\label{eq-7}
Q_2(z):=\frac{g_1'''g_2-g_1g_2'''}{g_1g_2'-g_1'g_2},\qquad Q_3(z):=\frac{g_1'g_2'''-g_2'g_1'''}{g_1g_2'-g_1'g_2}.
\end{equation}
Note that
\begin{itemize}\itemsep=0pt
\item[(i)] $g_1(z)$ and $g_2(z)$ are also solutions of~(\ref{eq-1})
 and $g_1$, $g_2$, $g_3$ are linearly independent.
\item[(ii)] Although $g_j$ might be multi-valued, $g_j'/g_j$ is
 single-valued. Thus $Q_j(z)$ is single-valued and meromorphic on
 $\mathbb{H}$ for $j=1,2$. We will see from Theorem~\ref{thm-1} below
 that any pole of~$Q_j(z)$ comes from zeros of~$W_f(z)$.
\end{itemize}

For $z_0\in\H$, we denote
$\kappa_{z_0}^{(1)}\leq\kappa_{z_0}^{(2)}\leq\kappa_{z_0}^{(3)}$ to be the local
exponents of \eqref{eq-1} at $z_0$.

\begin{Theorem}\label{thm-1} Under the above notations,
$Q_2(z)$ and $Q_3(z)-\frac12 Q_2'(z)$ are meromorphic modular forms on
$\Gamma$ $($with trivial character$)$ of weight $4$ and $6$,
respectively. Furthermore, all singular points of~\eqref{eq-1} on
$\mathbb{H}$ comes from the zeros of $W_f(z)$, $(H1)$--$(H3)$ hold, and
every cusp of~$\Gamma$ is completely not apparent for~\eqref{eq-1}.
\end{Theorem}

\begin{proof} To prove the modularity of $Q_j$, we consider
\[\tilde{y}(z):=\big(y\big|_{-2}\gamma\big)(z)=(cz+d)^{2}y(\gamma z),\qquad\gamma\in \Gamma.\]
Then
\[\tilde{y}'(z)=y'(\gamma z)+2c(cz+d)y(\gamma z),\qquad \tilde{y}'''(z)=(cz+d)^{-4}y'''(\gamma z),\]
so
\begin{gather}
L\tilde{y}= (cz+d)^{-4}\big\{y'''(\gamma z)+(cz+d)^4Q_2(z)y'(\gamma z)\nonumber\\
\hphantom{L\tilde{y}=}{}
+\big[(cz+d)^6Q_3(z)+2c(cz+d)^5 Q_2(z)\big]y(\gamma z)\big\}.\label{eq-2}
\end{gather}
Recalling $g_3(z)=\frac{f(z)}{\sqrt[3]{W_f(z)}}$, we have
\begin{align*}
\big(g_3\big|_{-2}\gamma\big)(z)&=\bigg(\frac{f}{\sqrt[3]{W_f}}\Big|_{-2}\gamma\bigg)(z)
=\frac{(cz+d)^2 f(\gamma z)}{\sqrt[3]{W_f(\gamma z)}}
=\frac{\varepsilon}{\sqrt[3]{W_f(z)}}\big(c^2h_1+cdh_2+d^2h_3\big)\\
&=\varepsilon\big(c^2 g_1(z)+cd g_2(z)+d^2g_3(z)\big),
\end{align*}
where $\varepsilon^3=1$. Thus $\big(g_3\big|_{-2}\gamma\big)(z)$ is also a solution of~(\ref{eq-1}). From this and~(\ref{eq-2}), we have
\begin{gather*}
Q_2(\gamma z)=(cz+d)^4Q_2(z),\\
 Q_3(\gamma z)=(cz+d)^6Q_3(z)+2c(cz+d)^5 Q_2(z),
 \end{gather*}
so $\big(Q_2\big|_{4}\gamma\big)=Q_2$ and $\big((Q_3-\frac12 Q_2')\big|_{6}\gamma\big)=Q_3-\frac12 Q_2'$. This proves the modularity of~$Q_2$ and~$Q_3$.

To prove (H1)--(H3),
 we let $z_0$ be any pole of $Q_j(z)$ for some $j=1,2$. Clearly $g_j(z)=(z-z_0)^{\alpha_j}(c_j+O(z-z_0)^j)$ near $z_0$ for some $\alpha_j\in \frac13\mathbb{Z}$ and $c_j\neq 0$. By replacing $g_2$ by $g_2-\frac{c_2}{c_1}g_1$ if necessary, we may assume $\alpha_1\neq \alpha_2$. Then we easily deduce from (\ref{eq-7}) that{\samepage
\begin{gather*}
Q_2(z)=\frac{\alpha_1(\alpha_1-1)(\alpha_1-2)-\alpha_2(\alpha_2-1)(\alpha_2-2)}
{(\alpha_2-\alpha_1)(z-z_0)^2}+O\big((z-z_0)^{-1}\big),\\
 Q_3(z)=\frac{\alpha_1\alpha_2[(\alpha_2-1)(\alpha_2-2)-(\alpha_1-1)(\alpha_1-2)]}
{(\alpha_2-\alpha_1)(z-z_0)^3}+O\big((z-z_0)^{-2}\big),
\end{gather*}
so $z_0$ is a regular singular point of~(\ref{eq-1}). Thus, (\ref{eq-1}) is Fuchsian on~$\mathbb{H}$.}

Let $z_0$ be any singular point. It follows from $g_j(z)=\frac{h_j(z)}{\sqrt[3]{W_f(z)}}$ that
\begin{equation}\label{eqe-4}g_j(z)=(z-z_0)^{\frac{-\text{ord}_{z_0}W_f}{3}}\sum_{l\geq 0}d_j (z-z_0)^j,\end{equation}
where $\text{ord}_{z_0}W_f$ denotes the zero order of $W_f(z)$ at $z_0$. Since $(g_1, g_2, g_3)$ is a fundamental system of solutions of (\ref{eq-1}) and $g_j$'s have no logarithmic singularities at $z_0$, we conclude from (\ref{eqe-4}) and Remark \ref{rmkk} that (1) the local exponents $\kappa_{z_0}^{(j)}\in\frac13\mathbb{Z}$ and are all distinct; (2) the exponent differences are all nonzero integers, namely $m_{z_0}^{(j)}:=\kappa_{z_{0}}^{(j+1)}-\kappa_{z_{0}}^{(j)}-1$ are nonnegative integers for $j=1,2$; (3) $z_0$ is an apparent singularity of (\ref{eq-1}).

Since
\begin{equation}\label{eqe-6}Q_j(z)=A_j(z-z_0)^{-j}+O\big((z-z_0)^{-j+1}\big),\qquad j=2,3,\end{equation}
then the indicial equation of (\ref{eq-1}) at $z_0$ is
\begin{equation}\label{eqe-7}\kappa(\kappa-1)(\kappa-2)+A_2\kappa+A_3=0,\end{equation}
which implies $\sum_{j=1}^3 \kappa_{z_0}^{(j)}=3$ and so $\kappa_{z_0}^{(1)}=-\frac{2m_{z_0}^{(1)}+m_{z_0}^{(2)}}{3}\in \frac{1}{3}\mathbb{Z}_{\leq 0}$. This proves (H2). Remark that if $z_0$ is not a zero of $W_f(z)$, i.e., $\text{ord}_{z_0}W_f=0$, then it follows from (\ref{eqe-4}) that $\kappa_{z_0}^{(1)}\in\mathbb{Z}_{\geq 0}$ and so $\kappa_{z_0}^{(1)}=m_{z_0}^{(1)}=m_{z_0}^{(2)}=0$, i.e., the local exponents at $z_0$ are $\{0,1,2\}$. This already implies $A_2=A_3=0$. Together with the fact that $z_0$ is apparent, we easily deduce from the Frobenius method that both~$Q_2(z)$ and~$Q_3(z)$ are holomorphic at~$z_0$, a contradiction with that $z_0$ is a~singular point. Thus $z_0$ is a zero of $W_f(z)$. This proves that all singular points of~(\ref{eq-1}) on $\mathbb{H}$ come from the zeros of~$W_f(z)$.

Let $N$ be the width of the cusp $\infty$ and $q_N={\rm e}^{2\pi {\rm i}z/N}$. Since modular forms $f_j(z)$, $W_f(z)$ are holomorphic in terms of $q_N$ and $z=\frac{N}{2\pi {\rm i}}\ln q_N$, we see from (\ref{eq-22}) that \[g_j(z)=\frac{h_j(z)}{\sqrt[3]{W_f(z)}}=\sum_{k=0}^{3-j}(\ln q_N)^{k}q_N^{-\frac{\ord_{\infty}W_f}{3}}\sum_{t=0}^{\infty}c_{j,k,t}q_N^t,\qquad j=1,2,3,\] so $\infty$ is also a regular singular point of~(\ref{eq-1}), i.e.,~(\ref{eq-1}) is Fuchsian at $\infty$ and so $Q_j(z)$ is holomorphic at $\infty$ for $j=2,3$. Since $(\ln q_N)^2$ appears in the expression of~$g_1(z)$, we see from Remark~\ref{rmk-apparent} that $\infty$ is completely not apparent and the local exponents $\kappa_\infty^{(1)}\le\kappa_\infty^{(2)}\le\kappa_\infty^{(3)}$ satisfy $\kappa_\infty^{(j)}\in\frac13\mathbb{Z}$ and
\[m_{\infty}^{(1)}:=\kappa_\infty^{(2)}-\kappa_\infty^{(1)}\in\mathbb{Z}_{\geq 0},\qquad m_{\infty}^{(2)}:=\kappa_\infty^{(3)}-\kappa_\infty^{(2)}\in\mathbb{Z}_{\geq 0}.\]
Note that the indicial equation at $\infty$ is
\[\kappa^3+\left(\frac{N}{2\pi {\rm i}}\right)^2Q_2(\infty)\kappa+\left(\frac{N}{2\pi {\rm i}}\right)^3Q_3(\infty)=0,\]
which implies $\sum \kappa_{\infty}^{(j)}=0$ and so $\kappa_{\infty}^{(1)}=-\frac{2m_{\infty}^{(1)}+m_{\infty}^{(2)}}{3}$. We now consider other cusps.

Assume that $s$ is another cusp of $\Gamma$ different from
$\infty$. Let $\sigma=\SM abcd\in\SL(2,\Z)$ be a matrix such that
$\sigma\infty=s$. Regarding $f(z)$ as a~quasimodular form on
$\Gamma'=\ker\chi\cap\SL(2,\Z)$, we can express $f(z)$ as $f(z)=\wt
f_0(z)+\wt f_1(z)E_2(z)+\wt f_2(z)E_2(z)^2$ for some $\wt
f_j(z)\in\sM_{k-2j}(\Gamma')$. We check that
\begin{align*}
 \left(g_3|_{-2}\sigma\right)(z)
 &=\frac{(cz+d)^2f(\sigma z)}{\sqrt[3]{W_f(\sigma z)}}
 =\epsilon\frac{(cz+d)^2(f|_k\sigma)(z)}{\sqrt[3]{(W_f|_{3k}\sigma)(z)}}\\
 &=\epsilon\frac{(cz+d)^2p_1(z)+\alpha
 c(cz+d)p_2(z)+\alpha^2c^2p_3(z)}{\sqrt[3]{(W_f|_{3k}\sigma)(z)}},
 \end{align*}
where
\begin{gather*}
 p_1(z)=\big(\wt f_0|_k\sigma\big)(z)+\big(\wt f_1|_{k-2}\sigma\big)(z)E_2(z)
 +\big(\wt f_2|_{k-4}\sigma\big)(z), \\
 p_2(z)=\big(\wt f_1|_{k-2}\sigma\big)(z)+2\big(\wt
 f_2|_{k-4}\sigma\big)(z)E_2(z), \\
 p_3(z)=\big(\wt f_2|_{k-4}\sigma\big)(z),
 \end{gather*}
$\alpha=6/\pi {\rm i}$ and $\epsilon$ is a third root of unity. Except
for $cz+d$, every term in the expression has a~$q_M$-expansion, where
$M$ is the width of the cusp $\sigma$ and $q_M={\rm e}^{2\pi {\rm i} z/M}$. Since
$s\neq\infty$, we have $c\neq0$. This shows that there is a local
solution at the cusp $s$ having a factor $z^2$. According to the
solution structure discussed in the appendix, the point $s$ must be
completely not apparent and we have
$\kappa_s^{(2)}-\kappa_s^{(1)},\kappa_s^{(3)}-\kappa_s^{(1)}\in\Z$. By
the same reasoning as in the case of the cusp $\infty$, the sum
$\kappa_s^{(1)}+\kappa_s^{(2)}+\kappa_s^{(3)}$ is equal to~$0$ and
hence $\kappa_s^{(j)}\in\frac13\Z$ for all~$j$.
This proves (H1), (H3), and that every cusp is completely not apparent.
\end{proof}

For a MODE, the local exponents are invariant under $z_0\to \gamma z_0$ for any $\gamma\in\Gamma$.

\begin{Proposition}\label{prop-3} Let $z_0$ be a singular point of~\eqref{eq-1}. Then the local exponents of~\eqref{eq-1} at $\gamma
 z_0$ are the same for all $\gamma=\SM{a}{b}{c}{d}\in\Gamma$.
\end{Proposition}

\begin{proof}
Let
\[Q_j(z)=\tilde{A}_j(z-\gamma z_0)^{-j}+O\big((z-\gamma z_0)^{-j+1}\big),\qquad j=2,3. \]
Recalling (\ref{eqe-6}) and (\ref{eqe-7}). we only need to prove $\big(\tilde{A}_2,\tilde{A}_3\big)=(A_2, A_3)$.

Since
\[
\gamma z-\gamma z_0=\frac{z-z_0}{(cz+d)(cz_0+d)}=\frac{z-z_0}{(cz_0+d)^2}(1+O(z-z_0)) \qquad \text{as $z\to z_0$},
\]
 we have for $z\to z_0$ that
\begin{align*}
Q_2(z)&=(cz+d)^{-4}Q_2(\gamma z)\\
&=(cz_0+d)^{-4}(1+O(z-z_0))\big[\tilde{A}_2(\gamma z-\gamma z_0)^{-2}+O\big((\gamma z-\gamma z_0)^{-1}\big)\big]\\
&=\tilde{A}_2(z-z_0)^{-2}+O\big((z-z_0)^{-1}\big),
\end{align*}
so $\tilde{A}_2=A_2$. Since $Q_3-\frac12 Q_2'$ is a modular form of weight $6$, a similar argument implies $\tilde{A}_3+\tilde{A}_2=A_3+A_2$ and so $\tilde{A}_3=A_3$.
\end{proof}

Now we consider $\Gamma=\SL(2,\mathbb{Z})$ and discuss the local exponents of~(\ref{eq-1}) at the elliptic points ${\rm i}$ and $\rho$, namely to complete the proof of Theorem~\ref{thm-02}. For this purpose, we note that the remark below could be used to simplify some computations.

\begin{Remark}
If $f_0$, $f_1$, $f_2$ have a common zero $z_0\in\mathbb{H}\cup\{\infty\}$, we take $M(z)$ to be a holomorphic modular form such that it has only one simple zero at $z_0$. Then $f_j(z)/M(z)$ are holomorphic modular forms and $f(z)/M(z)$ is a quasimodular form. Since $W_f(z)=M(z)^3W_{f/M}(z)$ and so
\[g_j(z)=\frac{h_j(z)}{\sqrt[3]{W_f(z)}}=\frac{h_j(z)/M(z)}{\sqrt[3]{W_{f/M}(z)}},\]
namely all $g_j(z)$'s are invariant by replacing $f(z)$ by $f(z)/M(z)$, so the differential equation~(\ref{eq-1}) derived from $f(z)$ and that from $f(z)/M(z)$ are the same (Note that $g_j(z)$ are solutions of~(\ref{eq-1}) but neither $f(z)$ nor $f(z)/M(z)$ are solutions of~(\ref{eq-1}), so we do not mean that $f(z)/M(z)$ satisfies the same differential equation as $f(z)$). Therefore, without loss of generality, we assume throughout the section that $f_0$, $f_1$, $f_2$ have no common zeros on $\mathbb{H}\cup\{\infty\}$.
\end{Remark}

By applying this remark, we have

\begin{Lemma} \label{lemma: nonvanishing}
 Let $\Gamma_\infty=\pm\gen T$ be the stabilizer subgroup of $\infty$
 in $\Gamma$. Then there are at most two right cosets
 $\Gamma_\infty\gamma$ in $\Gamma_\infty\backslash\Gamma$ $($the set of
 right cosets of $\Gamma_\infty$ in $\Gamma)$ such that $f(\gamma
 z_0)=0$.
\end{Lemma}

\begin{proof} Suppose that there are three distinct right cosets
 $\Gamma_\infty\gamma$ in $\Gamma_\infty\backslash\Gamma$ such
 that $f(\gamma z_0)=0$. Without loss of generality, we assume that
 one of them is the coset of $I$, i.e., $f(z_0)=0$. Now we have
 \[
 \left(f|_k\gamma\right)(z)=\chi(\gamma)\left(f(z)
 +\frac{\alpha c}{cz+d}(f_1(z)+2f_2(z)\phi(z))
 +\left(\frac{\alpha c}{cz+d}\right)^2f_2(z)\right)
 \]
 for $\gamma\in\Gamma$. If $f(\gamma_1z_0)=f(\gamma_2z_0)=0$ for
 $\gamma_1$ and $\gamma_2$ in two different cosets in
 $\Gamma_\infty\backslash\Gamma$, then we have
 $f_1(z_0)+2f_2(z_0)\phi(z_0)=f_2(z_0)=0$. However, this
 implies that $f_0(z_0)\neq 0$ since $f_j(z)$ are assumed to have no
 common zeros on $\H$, and hence $f(z_0)\neq 0$, a contradiction.
\end{proof}

Now we let $\Gamma={\rm SL}(2,\mathbb{Z})$. Given a character $\chi$, we have $\chi(T)={\rm e}^{2\pi {\rm i} m/24}$ for some integer $m\in [0,23]$. Recall the Dedekind eta function
\[\eta(z)={\rm e}^{2\pi {\rm i} z/24}\prod_{n=1}^\infty\big(1-{\rm e}^{2n\pi {\rm i} z}\big).\]
Since the MODE associated to $f$ is the same as that associated to $f/\eta^m$, by considering $f/\eta^m$ if necessary, we can always assume $\chi(T)=1$ and so $\chi(S)=\chi(R)=1$, i.e., we can always assume that the character $\chi$ is trivial for $\Gamma={\rm SL}(2,\mathbb{Z})$.

\begin{proof}[Proof of Theorem \ref{thm-02}]
The conclusions (1) and~(2) are proved in Theorem~\ref{thm-1}. It suffices to consider $\Gamma=\SL(2,\mathbb{Z})$ and prove (3) and~(4).

Let $z_0\in \{{\rm i}, \rho\}$.
By Lemma~\ref{lemma: nonvanishing}, we have $f(\gamma z_0)\neq 0$ for
 some $\gamma\in\SL(2,\Z)$. Then it follows from Proposition~\ref{prop-3} that
 \begin{equation*}%\label{eeqq}
 \kappa_{z_0}^{(1)}=\kappa_{\gamma z_0}^{(1)}
 =\ord_{\gamma z_0}\frac{f(z)}{\sqrt[3]{W_f(z)}}
 =-\frac13\ord_{z_0} W_f(z).
 \end{equation*}
 Since $W_f(z)$ is a modular form of weight $3k$ on $\SL(2,\Z)$, it
 follows from the valence formula for modular forms (see, e.g., \cite{Serre}) that
 \[\frac{\ord_{\rm i} W_f}{2}+\frac{\ord_{\rho} W_f}{3}\equiv\frac{k}{4}\quad\mod 1.\]
 This implies
 $\kappa_\rho^{(1)}=-\frac{\ord_{\rho} W_f}{3}\in\Z_{\le 0}$ and
 $3\kappa_{\rm i}^{(1)}=-\ord_{\rm i}W_f\equiv k/2 \mod 2$.

Recall that we may assume that $f_0$, $f_1$, and $f_2$
 have no common zeros. When $k\equiv 0\mod 4$, we have
 \[
 m_{\rm i}^{(2)}\equiv 3\kappa_{\rm i}^{(1)}\equiv\frac k2\equiv 0\quad \mod 2,
 \]
 and we are done. When $k\equiv 2\mod 4$, the weights of $f_0$ and
 $f_2$ are congruent to $2$ modulo $4$ and their expansions in
 $w=(z-{\rm i})/(z+{\rm i})$ are of the form
 \[
 f_0(z)=(1-w)^k\sum_{n=0}^\infty a_{2n+1}w^{2n+1}, \qquad
 f_2(z)=(1-w)^{k-4}\sum_{n=0}^\infty c_{2n+1}w^{2n+1},
 \]
 while the expansion of $f_1(z)$ is of the form
\[
 f_1(z)=(1-w)^{k-2}\sum_{n=0}^\infty b_{2n}w^{2n}.
 \]
 (See Proposition~\ref{prop-2} and Remark~\ref{remark: about
 coefficients} below.)

 Let $h_j(z)$, $j=1,2,3$, be given by (\ref{eq-22}). Then the local exponent of $ah_1(z)+bh_2(z)+ch_3(z)$ at $z={\rm i}$ must be one of
 \begin{equation}\label{eeqqqq}
 \big\{0, \kappa_{\rm i}^{(2)}-\kappa_{\rm i}^{(1)}, \kappa_{\rm i}^{(3)}-\kappa_{\rm i}^{(1)}\big\}
 \end{equation} for any $(a,b,c)\in\mathbb{C}^3$.
 Consider the function
 \begin{gather*}
 h_1+{\rm i}h_2-h_3 =z^2\big(f_0+f_1E_2+f_2E_2^2\big)+\alpha z(f_1+2f_2E_2)+\alpha^2f_2 \\
\hphantom{h_1+{\rm i}h_2-h_3 =}{}
+2{\rm i}z\big(f_0+f_1E_2+f_2E_2^2\big)+{\rm i}\alpha(f_1+2f_2E_2)
 -\big(f_0+f_1E_2+f_2E_2^2\big) \\
\hphantom{h_1+{\rm i}h_2-h_3}{}=(z+{\rm i})^2f_0+(z+{\rm i})((z+{\rm i})E_2+\alpha)f_1+((z+{\rm i})E_2+\alpha)^2f_2.
 \end{gather*}
 We compute that $z={\rm i}(1+w)/(1-w)$ and hence
 \begin{equation} \label{equation: z+i}
 z+{\rm i}=\frac{2{\rm i}}{1-w}, \qquad
 \frac{{\rm d}w}{{\rm d}z}=\frac{2{\rm i}}{(z+{\rm i})^2}=\frac{(1-w)^2}{2{\rm i}}.
 \end{equation}
 Also, since $E_2=\frac1{2\pi {\rm i}}{\rm d}\log\Delta(z)/{\rm d}z$ and the expansion
 of $\Delta(z)$ is of the form
 \[
 \Delta(z)=(1-w)^{12}\sum_{n=0}^\infty d_{2n}w^{2n},
 \]
 we find that
 \begin{align}
 E_2(z)&=-\frac{(1-w)^2}{4\pi}\left(-\frac{12}{1-w}
 +\frac{\sum 2nd_{2n}w^{2n-1}}{\sum d_{2n}w^{2n}}\right) \nonumber\\
 &=\frac3\pi(1-w)-\frac{(1-w)^2}{4\pi}
 \sum_{n=0}^\infty d_{2n+1}'w^{2n+1} \label{equation: E2 expansion}
 \end{align}
 for some power series $\sum d_{2n+1}'w^{2n+1}$. It follows that, by~\eqref{equation: z+i},
\[
 (z+{\rm i})E_2(z)+\frac{6}{\pi {\rm i}}=\frac{1-w}{2\pi {\rm i}}
 \sum_{n=0}^\infty d_{2n+1}'w^{2n+1}.
\]
 From this, we see that
 \[
 h_1+{\rm i}h_2-h_3=(1-w)^{k-2}\sum_{n=0}^\infty e_{2n+1}w^{2n+1}
\]
 for some $e_j$. This, together with (\ref{eeqqqq}), implies that either
 $\kappa_{\rm i}^{(2)}-\kappa_{\rm i}^{(1)}$ or $\kappa_{\rm i}^{(3)}-\kappa_{\rm i}^{(1)}$
 is odd. This proves the assertion (3) that $\big\{3\kappa_{\rm i}^{(1)},3\kappa_{\rm i}^{(2)},
 3\kappa_{\rm i}^{(3)}\big\}\equiv\{0,0,1\}$ $\mod 2$.

 The proof of (4) is similar. We consider the function
\[
 h_1-\ol\rho h_2+\ol\rho^2 h_3
 =(z-\ol\rho)^2f_0+(z-\ol\rho)((z-\ol\rho)E_2+\alpha)f_1
 +((z-\ol\rho)E_2+\alpha)^2f_2.
\]
 Setting $w=(z-\rho)/(z-\ol\rho)$, we have $z=(\rho-\ol\rho w)/(1-w)$,
\[
 z-\ol\rho=\frac{\sqrt 3{\rm i}}{1-w}, \qquad
 \frac{{\rm d}w}{{\rm d}z}=\frac{(1-w)^2}{\sqrt3{\rm i}},
\]
 and
 \begin{equation} \label{equation: (z-rho)E2}
 (z-\ol\rho)E_2(z)+\frac6{\pi {\rm i}}=\frac{1-w}{2\pi {\rm i}}\sum_{n=0}^\infty
 d_{3n+2}'w^{3n+2}
 \end{equation}
 (since the expansion of $\Delta(z)$ is $(1-w)^{12}\sum d_{3n}w^{3n}$
 for some $d_{3n}$). When $k\equiv 1\mod 3$, the expansions of $f_j$
 are of the form $f_0=(1-w)^k\sum a_{3n+1}w^{3n+1}$,
 $f_1=(1-w)^{k-2}\sum b_{3n+2}w^{3n+2}$, and $f_2=(1-w)^{k-4}\sum
 c_{3n}w^{3n}$. Therefore, the expansion of $h_1-\ol\rho
 h_2+\ol\rho^2h_3$ is of the form
\[
 h_1-\ol\rho h_2+\ol\rho^2 h_3=(1-w)^{k-2}\sum_{n=0}^\infty
 e_{3n+1}w^{3n+1},
\]
 which implies that either $\kappa_\rho^{(2)}-\kappa_\rho^{(1)}$ or
 $\kappa_\rho^{(3)}-\kappa_\rho^{(1)}$ is congruent to $1$ modulo
 $3$. Since the sum of~$\kappa_\rho^{(j)}$ is~$3$, we deduce that the
 set $\big\{\kappa_\rho^{(1)},\kappa_\rho^{(2)},\kappa_\rho^{(3)}\big\}$ is
 congruent to $\{0,1,2\}$ modulo~$3$. Likewise, when $k\equiv 2\mod
 3$, we can show that $h_1-\ol\rho h_2+\ol\rho^2h_3=(1-w)^{k-2}
 \sum e_{3n+2}w^{3n+2}$ and obtain the same conclusion.

 For the case $k\equiv 0\mod 3$, we need to make the computation more
 precise. Let
 \begin{gather*}
 f_0(z)=(1-w)^k\sum_{n=0}^\infty a_{3n}w^{3n}, \\
 f_1(z)=(1-w)^{k-2}\sum_{n=0}^\infty b_{3n+1}w^{3n+1}, \\
 f_2(z)=(1-w)^{k-4}\sum_{n=0}^\infty c_{3n+2}w^{3n+2}
 \end{gather*}
 be the expansions of $f_j$. A computation similar to \eqref{equation: E2
 expansion} yields
\[
 E_2(z)=\frac{2\sqrt3}\pi(1-w)-\frac{(1-w)^2}{2\pi\sqrt3}
 \sum_{n=0}^\infty d_{3n+2}'w^{3n+2}
\]
 and hence
 \begin{gather*}
 f(z)=(1-w)^k\sum_{n=0}^\infty a_{3n}w^{3n}\\
 \hphantom{f(z)=}{}
 +(1-w)^{k-1}\left(\frac{2\sqrt3}\pi-\frac{1-w}{2\pi\sqrt3}
 \sum_{n=0}^\infty d_{3n+2}'w^{3n+2}\right)
 \left(\sum_{n=0}^\infty b_{3n+1}w^{3n+1}\right)\\
 \hphantom{f(z)=}{}+(1-w)^{k-2}\left(\frac{2\sqrt3}\pi-\frac{1-w}{2\pi\sqrt3}
 \sum_{n=0}^\infty d_{3n+2}'w^{3n+2}\right)^2
 \left(\sum_{n=0}^\infty c_{3n+2}w^{3n+2}\right).
 \end{gather*}
 On the other hand, by \eqref{equation: (z-rho)E2}, we have
 \begin{gather*}
 h_1-\ol\rho h_2+\ol\rho^2h_3 \\
 \qquad =(1-w)^{k-2}\left(-3\sum_{n=0}^\infty a_{3n}w^{3n}
 +\frac{\sqrt3}{2\pi}\left(\sum_{n=0}^\infty
 d_{3n+2}'w^{3n+2}\right)
 \left(\sum_{n=0}^\infty b_{3n+1}w^{3n+1}\right) \right.\\
 \left.\qquad\quad{}-\frac1{4\pi^2}\left(\sum_{n=0}^\infty
 d_{3n+2}'w^{3n+2}\right)^2
 \left(\sum_{n=0}^\infty c_{3n+2}w^{3n+2}\right)\right).
 \end{gather*}
 We then check that the expansion of $h_1-\ol\rho
 h_2+\ol\rho^2h_3+3f$ is of the form
\[
 (1-w)^{k-2}\sum_{n=0}^\infty\big(e_{3n+1}w^{3n+1}+e_{3n+2}w^{3n+2}\big),
\]
 which again implies that either
 $\kappa_\rho^{(2)}-\kappa_\rho^{(1)}$ or
 $\kappa_\rho^{(3)}-\kappa_\rho^{(1)}$ is not congruent to $0$ modulo
 $3$ and hence
 $\big\{\kappa_\rho^{(1)},\kappa_\rho^{(2)},\kappa_\rho^{(3)}\big\}\equiv\{0,1,2\}\mod
 3$. This completes the proof.
\end{proof}

\section[The MODE on SL(2,Z)]{The MODE on $\boldsymbol{\SL(2,\mathbb{Z})}$}\label{section-5}

The purpose of this section is to prove Theorems~\ref{thm-01} and~\ref{thm-01+}, and Corollary~\ref{coro-1-3}.
Let $Q_2(z)$ and $Q_3(z)-\frac12 Q_2(z)$ be meromorphic modular forms on $\SL(2,\mathbb{Z})$ of weight $4$ and $6$ respectively, i.e., \begin{equation}\label{eq-10}
Ly:=y'''(z)+Q_2(z)y'(z)+Q_3(z)y(z)=0,\qquad z\in\mathbb{H}
\end{equation} is a MODE. In this section, we use the notations $S=\bigl(\begin{smallmatrix}0 & -1\\
1 & 0\end{smallmatrix}\bigr)$, $T=\bigl(\begin{smallmatrix}1 & 1\\
0 & 1\end{smallmatrix}\bigr)$ and $R=ST=\bigl(\begin{smallmatrix}0 & -1\\
1 & 1\end{smallmatrix}\bigr)$ such that $S^2=R^3=-I_2$.
Suppose the MODE (\ref{eq-10}) has regular singularities at $\{z_1,\dots, z_m\}\sqcup\{{\rm i}, \rho,\infty\}$ mod $\SL(2,\mathbb{Z})$, where $\rho=\big({-}1+\sqrt{3}{\rm i}\big)/2$ is the fixed point of $R$.

\subsection{Proof of Theorem \ref{thm-01+}(1)}
First we want to give the proof of Theorem \ref{thm-01+}(1), which is long and will be separated into several lemmas. For this purpose, throughout this section we always assume that
 \begin{itemize}\itemsep=0pt
 \item[(S1)] $\kappa_{\infty}^{(1)}, \kappa_{\rm i}^{(1)}, \kappa_{z_j}^{(1)}\in\frac13\mathbb{Z}_{\leq 0}$ and $\kappa_\rho^{(1)}\in\mathbb{Z}_{\leq 0}$ such that
 \begin{equation*}
\ell:=-2-12\kappa_{\infty}^{(1)}-4\kappa_{\rho}^{(1)}-6\kappa_{\rm i}^{(1)}-12\sum_{j=1}^m \kappa_{z_j}^{(1)}\in\mathbb{Z}.
\end{equation*} (Note $\ell\in\mathbb{Z}$ implies $\ell\in 2\mathbb{Z}$.) Furthermore, $m_{z}^{(j)}:=\kappa_{z}^{(j+1)}-\kappa_{z}^{(j)}-1\in\mathbb{Z}_{\geq 0}$ for $z\in \{z_1,\dots, z_m\}\sqcup\{{\rm i}, \rho\}$ and $j=1,2$, and $m_{\infty}^{(j)}:=\kappa_{\infty}^{(j+1)}-\kappa_{\infty}^{(j)}\in\mathbb{Z}_{\geq 0}$ for $j=1,2$.
 \item[(S2)]ODE (\ref{eq-10}) is apparent at any singular point $z\in \{z_1,\dots, z_m\}\sqcup\{{\rm i}, \rho\}$.
 \item[(S3)]$\big\{3\kappa_{\rm i}^{(1)},3\kappa_{\rm i}^{(2)},
 3\kappa_{\rm i}^{(3)}\big\}\equiv\{0,0,1\} \mod 2$.
 \end{itemize}
\begin{Remark}Note that (S1) and~(S2) are equivalent to the assumptions (H1)--(H3) and $\kappa_{\rho}^{(1)}\in\mathbb{Z}_{\leq 0}$ in Theorem~\ref{thm-01}, while (S3) is needed to obtain quasimodular forms as stated in Theorem~\ref{thm-01}(1). In view of Theorem~\ref{thm-02}, the above assumptions (S1)--(S3) are necessary for the validity of Theorem~\ref{thm-01+}(1). We will prove below that they are also sufficient.\end{Remark}

Set $t_j:=E_4(z_j)^3/E_6(z_j)^2\not\in\{0,1,\infty\}$ and
$F_j(z):=E_4(z)^3-t_j E_6(z)^2$.
Then this modular form $F_j(z)$ has only one simple zero at $z_j$, up
to $\SL(2,\Z)$-equivalence. Define
\begin{equation*}%\label{eq-5-5}
F(z):=\Delta(z)^{-\kappa_{\infty}^{(1)}}E_4(z)^{-\kappa_{\rho}^{(1)}}E_6(z)^{-\kappa_{\rm i}^{(1)}}
\prod_{j=1}^m F_j(z)^{-\kappa_{z_j}^{(1)}}.
\end{equation*}
Then $F(z)^3$ is a modular form of weight $3(\ell+2)$.

Clearly for any solution $y(z)$ of (\ref{eq-10}),
\[\hat{y}(z):=F(z)y(z)\]
is single-valued and holomorphic on $\mathbb{H}$. Furthermore, its order at $z\in\{z_1,\dots, z_m\}\sqcup\{{\rm i}, \rho\}$ is one of $\big\{0,m_{z}^{(1)}+1, m_{z}^{(1)}+m_{z}^{(2)}+2\big\}$, and at $\infty$ is one of $\big\{0,m_{\infty}^{(1)}, m_{\infty}^{(1)}+m_{\infty}^{(2)}\big\}$.

Fix a fundamental system of solutions $Y(z)=(y_1(z)$, $ y_2(z), y_3(z))^t$ of (\ref{eq-10}) and let $\hat{Y}(z):=F(z)Y(z)$. Then for any $\gamma\in \SL(2,\Z)$, there is a matrix $\hat{\rho}(\gamma)\in {\rm GL}(3,\mathbb{C})$ such that
\begin{equation}\label{eq-82}\big(\hat{Y}\big |_{\ell}\gamma\big)(z)=\hat{\rho}(\gamma)\hat{Y}(z).\end{equation}
This is a lifting of the Bol representation. We will use freely the notation $\hat{\gamma}=\hat{\rho}(\gamma)$ just for convenience.

\begin{Lemma}\label{lemm-3-1}
There holds $\det\hat{\rho}(\gamma)=1$ for any $\gamma\in \SL(2,\Z)$. That is, $\hat{\rho}$ is a group homomorphism from $\SL(2,\mathbb{Z})$ to $\SL(3,\mathbb{C})$.
\end{Lemma}

\begin{proof} The proof is similar to that of \cite[Lemma 4.2]{LY},
 where the second-order MODE was studied. Let
\[W(z)=\det\begin{pmatrix}y_1&y_1'&y_1''\\y_2&y_2'&y_2''\\y_3&y_3'&y_3''\end{pmatrix},\qquad
\hat{W}(z)=\det\begin{pmatrix}\hat{y}_1&\hat{y}_1'&\hat{y}_1''\\
\hat{y}_2&\hat{y}_2'&y_2''\\ \hat{y}_3&\hat{y}_3'&\hat{y}_3''\end{pmatrix}.\]
Then $W(z)\equiv C$ is a nonzero constant.
By (\ref{eq-82}) we have
\[\big(\hat{W}\big|_{3(\ell+2)}\gamma\big)(z)=\det\hat{\rho}(\gamma)\hat{W}(z).\]
Since $\hat{W}(z)=F(z)^3W(z)=CF(z)^3$, we also have
\begin{align*}
\hat{W}\big|_{3(\ell+2)}\gamma=C\big(F^3\big|_{3(\ell+2)}\gamma\big)
=\frac{\big(F^3\big|_{3(\ell+2)}\gamma\big)(z)}{F(z)^3}\hat{W}(z).
\end{align*}
Thus
\begin{equation*}%\label{eq-83}
\det\hat{\rho}(\gamma)=\frac{\big(F^3\big|_{3(\ell+2)}\gamma\big)(z)}{F(z)^3}.
\end{equation*}
This proves $\det\hat{\rho}(\gamma)=1$ because $F(z)^3$ is a modular form of weight $3(\ell+2)$ on $\SL(2,\mathbb{Z})$.
\end{proof}

Remark that under our assumption $\kappa_{\infty}^{(1)}, \kappa_{\rm i}^{(1)}, \kappa_{z_j}^{(1)}\in\frac13\mathbb{Z}_{\leq 0}$ and $\kappa_\rho^{(1)}\in\mathbb{Z}_{\leq 0}$, $y(z)^3$ is a single-valued and meromorphic function on $\mathbb{H}$ for any solution $y(z)$ of (\ref{eq-10}).

\begin{Lemma}\label{lemma-3-1}
Under the assumptions $(S1)$--$(S3)$, there is at least one solution $y(z)$
of~\eqref{eq-10} such that $y(z)^3$ is not a meromorphic modular form
of weight~$-6$.
\end{Lemma}

\begin{proof}
Suppose the conclusion is not true, namely $y(z)^3$ is a meromorphic
modular form of weight $-6$ for any solution $y(z)$. Then by the
well-known valence formula for modular forms (see, e.g., \cite{Serre}),
we obtain
\[\frac{\ord_{\rm i}\big(y^3\big)}{2}+\frac{\ord_{\rho}\big(y^3\big)}{3}\equiv \frac12\quad \mod \mathbb{Z},\]
so $\ord_{\rm i}\big(y^3\big)$ is odd. Since $\ord_{\rm i}\big(y^3\big)$ can be chosen as any one of $3\kappa_{\rm i}^{(1)}$, $3\kappa_{\rm i}^{(2)}$, $3\kappa_{\rm i}^{(3)}$,
these three numbers are all odd, clearly a contradiction with our assumption~(S3).
\end{proof}

\begin{Lemma}\label{lemma-3}
Under the assumptions $(S1)$--$(S3)$, \eqref{eq-10} is completely not apparent at~$\infty$.
\end{Lemma}

To prove Lemma \ref{lemma-3}, we need the following well-known lemma due to Beukers and Heckman~\cite{BH}.

\begin{Lemma}[\cite{BH}] \label{lemma-4} Let $n\geq 2$ and $H\subset {\rm GL}(n,\mathbb{C})$ be a subgroup generated by two matrices $A$, $B$ such that $\operatorname{rank}(A-B)\leq 1$. Then $H$ acts irreducibly on $\mathbb{C}^n$ if and only if $A$ and $B$ have distinct eigenvalues.
\end{Lemma}

Let $\{a_1,\dots,a_n\}$ and $\{b_1,\dots,b_n\}$ be the eigenvalues of $A$ and $B$ respectively. The following lemma, due to Levelt, is to recover $A$ and $B$ by their eigenvalues. See~\cite{BH} for a proof.

\begin{Lemma}[cf.~\cite{BH}] \label{lemma-5}
Suppose that $\operatorname{rank}(A-B)= 1$ and $a_1,\dots,a_n$,
$b_1,\dots,b_n$ are all nonzero complex numbers with $a_i\neq b_j$
for any $i$, $j$. Then up to a common conjugation in ${\rm GL}(n,\mathbb{C})$,
$A$~and~$B$ can be uniquely determined by
\[A=\begin{pmatrix}0&0&\cdots&0&-A_n\\
1&0&\cdots&0&-A_{n-1}\\
0&1&\cdots&0&-A_{n-2}\\
\vdots&\vdots&\ddots&\vdots&\vdots\\
0&0&\cdots&1&-A_1\end{pmatrix},\qquad B=\begin{pmatrix}0&0&\cdots&0&-B_n\\
1&0&\cdots&0&-B_{n-1}\\
0&1&\cdots&0&-B_{n-2}\\
\vdots&\vdots&\ddots&\vdots&\vdots\\
0&0&\cdots&1&-B_1\end{pmatrix},\]
where $A_j$'s and $B_j$'s are given by
\begin{gather*}
\prod_{j=1}^{n}(t-a_j)=t^n+A_1 t^{n-1}+\cdots+A_n,\\
\prod_{j=1}^{n}(t-b_j)=t^n+B_1 t^{n-1}+\cdots+B_n.
\end{gather*}
\end{Lemma}

\begin{Corollary}\label{coro-5-6}
Let $A$, $B$ be $3\times 3$ matrices such that
$\operatorname{rank}(A-B)\leq 1$. Suppose that $A^2=I_3$ and the
eigenvalues of~$A$ are $\{1,-1,-1\}$. Then $A$, $B$ have common
eigenvalues and so the group~$H$ generated by $A$ and $B$ acts on~$\mathbb{C}^3$ reducibly.
\end{Corollary}

\begin{proof}
This corollary is trivial if $\operatorname{rank}(A-B)=0$,
i.e., $A=B$. So we may assume $\operatorname{rank}(A-B)\allowbreak =1$. Assume by contradiction
that $A$ and $B$ have no common eigenvalues, then by
$(t-1)(t+1)^2=t^3+ t^2-t-1$, we see from Lemma~\ref{lemma-5} that~$A$
is conjugate to
\[\begin{pmatrix}
0&0&1\\
1&0&1\\
0&1&-1\end{pmatrix}\]
and so $A^2\neq I_3$, a contradiction with the assumption $A^2=I_3$.

Thus $A$ and $B$ have common eigenvalues, and it follows from Lemma~\ref{lemma-4} that $H$ acts on~$\mathbb{C}^3$ reducibly.
\end{proof}

\begin{proof}[Proof of Lemma \ref{lemma-3}]
Suppose that the statement is not true. Then
$\operatorname{rank}\big(\hat{T}-I_3\big)\leq 1$ (note $\hat{T}=I_3$ if and
only if $\infty$ is apparent). Then $\hat{R}=\hat{S}\hat{T}$ implies
 $\operatorname{rank}\big(\hat{S}-\hat{R}\big)\leq 1$.

By $\hat{R}^3=\hat{\rho}\big(R^3\big)=\hat{\rho}(-I_2)=(-1)^{\ell}I_3=I_3$ and $\det \hat{R}=1$, we have either $\hat{R}=\lambda I_3$ for some $\lambda^3=1$ or $\hat{R}$ is conjugate to $\operatorname{diag}\big(1, \varepsilon, \varepsilon^2\big)$ where $\varepsilon={\rm e}^{2\pi {\rm i}/3}$.
Similarly, by $\hat{S}^2=\hat{\rho}\big(S^2\big)=\hat{\rho}(-I_2)\allowbreak =I_3$ and $\det \hat{S}=1$, we have either $\hat{S}=I_3$ or $\hat{S}$ is a conjugate of $\operatorname{diag}(1,-1,-1)$.

If $\hat{S}=I_3$, then by $\operatorname{rank}\big(\hat{S}-\hat{R}\big)\leq 1$ we obtain $\hat{R}=I_3$. This implies that for any solution~$y(z)$ of~(\ref{eq-10}), $\hat{y}(z)^3=F(z)^3y(z)^3$ is a modular form of weight $3\ell$ and so $y(z)^3$ is a meromorphic modular form of weight $-6$, a~contradiction with Lemma~\ref{lemma-3-1}.

Thus $\hat{S}$ is a conjugate of $\operatorname{diag}(1,-1,-1)$. If $\hat{R}=\lambda I_3$ for some $\lambda^3=1$, then by $\lambda\neq -1$ we obtain{\samepage
\[1\geq\operatorname{rank}\big(\hat{S}-\hat{R}\big)=\operatorname{rank}(\operatorname{diag}(1-\lambda,-1-\lambda,-1-\lambda))\geq 2,\]
a contradiction.}

So $\hat{R}$ is conjugate to $\operatorname{diag}\big(1, \varepsilon, \varepsilon^2\big)$. By Corollary~\ref{coro-5-6}, there is a subspace $V\subsetneqq \mathbb{C}^3$ which is invariant under the actions $\hat{S}$ and $\hat{R}$. If $\dim V=2$, then there is an invertible matrix $P$ such that
\[P\hat{S}P^{-1}=\begin{pmatrix}A_1&0\\
*&a_1\end{pmatrix},\qquad P\hat{R}P^{-1}=\begin{pmatrix}B_1&0\\
*&b_1\end{pmatrix},\]
where $A_1$ and $B_1$ are $2\times 2$ matrices. This implies
\begin{equation}\label{eeq}\operatorname{rank}\begin{pmatrix}A_1-B_1&0\\
 *&a_1-b_1\end{pmatrix}=\operatorname{rank}\big(\hat{R}-\hat{S}\big)\leq 1.
\end{equation}
Note $a_1\in \{1,-1\}$ and $b_1\in \big\{1,\varepsilon,\varepsilon^2\big\}$. If $a_1\neq b_1$, then (\ref{eeq}) implies $A_1=B_1$, namely $\hat{S}$ and $\hat{R}$ have two common eigenvalues, a contradiction. So $a_1=b_1=1$. Then the eigenvalues of $A_1$ are $\{-1,-1\}$, so $A_1=-I_2$. Similarly, $B_1$ is conjugate to $\operatorname{diag}\big(\varepsilon,\varepsilon^2\big)$. Thus
\[1\geq\operatorname{rank}(A_1-B_1)=\operatorname{rank}\big(\operatorname{diag}\big({-}1-\varepsilon,-1-\varepsilon^2\big)\big)= 2,\]
a contradiction.
So $\dim V=1$, which implies the existence of an invertible matrix $P$ such that
\[P\hat{S}P^{-1}=\begin{pmatrix}a_1&0\\
*&A_1\end{pmatrix},\qquad P\hat{R}P^{-1}=\begin{pmatrix}b_1&0\\
*&B_1\end{pmatrix},\]
where $A_1$ and $B_1$ are $2\times 2$ matrices. Clearly the same argument as~(\ref{eeq}) also yields a contradiction.
This completes the proof.
\end{proof}

Let $\hat{y}_3(z):=\hat{y}_+(z)=F(z)y_+(z)$, $\hat{y}_1(z):=(\hat{y}_3\big|_{\ell}S)(z)$ and $\hat{y}_2(z):=(\hat{y}_3\big|_{\ell}R)(z)$, and $y_j(z):=\hat{y}_j(z)/F(z)$ for $j=1,2,3$.

\begin{Lemma}\label{lemma-5-8}
Under the assumptions $(S1)$--$(S3)$, $\hat{y}_1(z)$, $\hat{y}_2(z)$ and $\hat{y}_3(z)$ are linearly independent and $\hat{y}_1(z)$ can be written as
\begin{equation}\label{eq-51}
\hat{y}_1(z)=\beta z^2\hat{y}_3(z)+z\hat{m}_1^*(z)+\hat{m}_2(z),
\end{equation}
where $\beta\neq 0$ is a constant and $\hat{m}_1^*(z)$, $\hat{m}_2(z)$ are of the form
\begin{equation}\label{eq-124}\frac{\hat{m}_1^*(z)}{F(z)}=m_1^{*}(z)=q^{\kappa_{\infty}^{(2)}}\sum_{j\geq 0}c_{j,1}q^j,\qquad \frac{\hat{m}_2(z)}{F(z)}=m_2(z)=q^{\kappa_{\infty}^{(1)}}\sum_{j\geq 0}c_{j,2}q^j.\end{equation}
\end{Lemma}

\begin{proof} Under our assumption, Lemma~\ref{lemma-3} says that~$\infty$ is completely not apparent, so it follows from Remark~\ref{rmk-apparent} that (\ref{eq-10}) has a basis of solutions of
 the form $(y_{-}, y_{\perp}, y_+)$, where $y_+$, $y_\perp$, and~$y_-$ are given by~\eqref{eq-60-1} and~\eqref{eq-60-2}.

 {\bf Step 1.} We show that $\hat{y}_1(z)$ is linearly independent with $\hat{y}_3(z)$.

 Suppose not, i.e., there is some constant $\alpha\neq 0$ such that
 \[(\hat{y}_3\big|_{\ell}S)(z)=\hat{y}_1(z)=\alpha \hat{y}_3(z).\]
Then with respect to $(F y_{-}, F y_{\perp}, \hat{y}_3(z))^t$, we have{\samepage
\[\hat{\rho}(S)=\begin{pmatrix}S_1&*\\0&\alpha\end{pmatrix},\]
where $S_1$ is a $2\times 2$ matrix. Then $\hat{\rho}(S)^2=I_3$ implies $S_1^2=I_2$ and $\alpha^2=1$, i.e., $\alpha=\pm 1$.}

On the other hand, it follows from the expressions of $(y_{-}, y_{\perp}, y_+)$ in Remark~\ref{rmk-apparent} that with respect to $(F y_{-}, F y_{\perp}, \hat{y}_3(z))^t$,
\[\hat{\rho}(T)=\begin{pmatrix}1&2&*\\0&1&1\\0&0&1\end{pmatrix},\]
so
\[\hat{\rho}(R)=\hat{\rho}(S)\hat{\rho}(T)=\begin{pmatrix}S_1&*\\0&\alpha\end{pmatrix}
\begin{pmatrix}1&2&*\\0&1&1\\0&0&1\end{pmatrix}.\]
From this and $\hat{\rho}(R)^3=I_3$, we obtain
\begin{equation}\label{eq-78}
\left(S_1R_1\right)^3=I_2, \qquad\text{where}\quad R_1:=\begin{pmatrix}1&2\\0&1\end{pmatrix},
\end{equation} and $\alpha^3=1$, so $\alpha=1$. Then $\det S_1=\det\hat{\rho}(S)=1$, which together with $S_1^2=I_2$ easily implies $S_1=\pm I_2$, and then (\ref{eq-78}) yields $\pm\bigl(\begin{smallmatrix}1 & 6\\
0 & 1\end{smallmatrix}\bigr)=I_2$,
a contradiction.

{\bf Step 2.} By Step 1, there exists $(\beta, \delta)\neq (0,0)$ and $\epsilon$ such that
\begin{equation}\label{eq-79}\hat{y}_1=\beta Fy_-+\delta Fy_{\perp}+\epsilon \hat{y}_3.\end{equation}
We claim that $\beta\neq 0$ and (\ref{eq-51}) holds.

Assume by contradiction that $\beta=0$, i.e., $\hat{y}_1=\delta Fy_{\perp}+\epsilon \hat{y}_3$ with $\delta\neq 0$. Then with respect to $(Fy_-, \hat{y}_1, \hat{y}_3)$ we have
\[\hat{\rho}(T)=\begin{pmatrix}1&\frac{2}{\delta}&a\\0&1&\delta\\0&0&1\end{pmatrix}\qquad\text{for some }a\in\mathbb{C},\]
and
\[\hat{\rho}(S)=\begin{pmatrix}-1&b&b\\0&0&1\\0&1&0\end{pmatrix}\qquad\text{for some }b\in\mathbb{C},\]
 where we used $\hat{\rho}(S)^2=I_3$, $\det\hat{\rho}(S)=1$, $\hat{y}_1(z)=(\hat{y}_3\big|_{\ell}S)(z)$ and
 \begin{equation}\label{eq-80}(\hat{y}_1\big|_\ell S)(z)=(\hat{y}_3\big|_\ell (-I_2))(z)=(-1)^{-\ell}\hat{y}_3(z)=\hat{y}_3(z).\end{equation}
Thus
\[\hat{\rho}(R)=\hat{\rho}(S)\hat{\rho}(T)
=\begin{pmatrix}-1&b-\frac{2}{\delta}&b+b\delta-a\\0&0&1\\0&1&\delta\end{pmatrix}.\]
But this implies that $-1$ is an eigenvalue of $\hat{\rho}(R)$, a
contradiction with $\hat{\rho}(R)^3=I_3$. This proves $\beta\neq 0$
and so it follows from the expression of $(y_{-}, y_{\perp}, y_+)$ in
Remark~\ref{rmk-apparent} that (\ref{eq-51}) and~(\ref{eq-124}) hold.

{\bf Step 3.} We show that $\hat{y}_1(z)$, $\hat{y}_2(z)$ and $\hat{y}_3(z)$ are linearly independent.

In fact, we see from $R=ST$ that $\hat{y}_2(z)=(\hat{y}_+\big|_{\ell}R)(z)=(\hat{y}_1\big|_{\ell}T)(z)=\hat{y}_1(z+1)$, from which and~(\ref{eq-51}) we obtain
\begin{equation}\label{eq-52}
\hat{y}_2=\hat{y}_1+2\beta z\hat{y}_3+\hat{m}_1^*+\beta \hat{y}_3.
\end{equation}
By (\ref{eq-79}) and (\ref{eq-52}) we have
\[\begin{pmatrix}\hat{y}_1\\ \hat{y}_2\\ \hat{y}_3\end{pmatrix}=A\begin{pmatrix}Fy_-\\ Fy_{\perp}\\ \hat{y}_3\end{pmatrix}\qquad\text{with}\quad A=\begin{pmatrix}\beta&\delta&\epsilon\\ \beta&2\beta+\delta&*\\0&0&1\end{pmatrix}.\]
Since $\det A=2\beta^2\neq 0$, we obtain that $\hat{y}_1(z)$, $\hat{y}_2(z)$ and $\hat{y}_3(z)$ are also linearly independent. This completes the proof.
\end{proof}

Recalling \eqref{eq-51}, we
set $\hat{m}_1(z)$ and $\hat{m}_0(z)$ to be
\begin{gather}\label{eq-53}
\hat{m}_1^*(z)=\hat{m}_1(z)+\frac{\pi {\rm i}}{3}\hat{m}_2(z) E_2(z),
\\ \label{eq-54}
\hat{y}_3(z)=\left(\frac{\pi {\rm i}}{6}\right)^2\hat{m}_2(z) E_2(z)^2+\frac{\pi {\rm i}}{6}\hat{m}_1(z)E_2(z)+\hat{m}_0(z).
\end{gather}
The following result can be seen as the converse statement of Theorem \ref{thm-1}.

\begin{Theorem}\label{thm-sl-2}
Under the assumptions $(S1)$--$(S3)$, the following hold.
\begin{itemize}\itemsep=0pt
\item[$(a)$] $\beta =1$.
\item[$(b)$] $\hat{m}_j(z)$ are meromorphic modular forms of weight $\ell+2-2j$ for $j=0,1,2$, that is, $\hat{y}_3(z)$ is a quasimodular form of weight $\ell+2$ with depth~$2$.
\end{itemize}
\end{Theorem}

\begin{proof}
(a) By Lemma \ref{lemma-5-8}, we can take $\hat{Y}(z)=(\hat{y}_1(z), \hat{y}_2(z), \hat{y}_3(z))^t$ to be a basis and let $\hat{T}$,~$\hat{S}$,~$\hat{R}$ denote the associated matrices $\hat{\rho}(T)$, $\hat{\rho}(S)$ and $\hat{\rho}(R)$ of the Bol representation. Recalling~(\ref{eq-80}) that $(\hat{y}_3\big|_{\ell}S)=\hat{y}_1(z)$ and $(\hat{y}_1\big|_{\ell}S)=\hat{y}_3(z)$, we have
\begin{equation}\label{eq-55}
\hat{S}=\begin{pmatrix}
0&0&1\\
\lambda &-1 &\lambda\\
1&0&0
\end{pmatrix},\qquad \text{for some }\lambda\in\mathbb{C},
\end{equation}
where $\hat{S}^2=I_3$ is used. Note from $SR=S^2T=-T$ that
\[\hat{y}_1\big|_\ell R=\hat{y}_3\big|_\ell SR=\hat{y}_3\big|_\ell (-T)=\hat{y}_3,\]
and from $R^2=T^{-1}S$ that
\[\hat{y}_2\big|_\ell R=\hat{y}_3\big|_\ell R^2=\hat{y}_3\big|_\ell \big(T^{-1}S\big)=\hat{y}_3\big|_\ell S=\hat{y}_1,\]
so
\begin{equation}\label{eq-56}
\hat{R}=\begin{pmatrix}
0&0&1\\
1 &0 &0\\
0&1&0
\end{pmatrix}.
\end{equation}
Therefore,
\begin{align}\label{eq-57}
\hat{T}=\hat{S}^{-1}\hat{R}=\hat{S}\hat{R}=\begin{pmatrix}
0&0&1\\
\lambda &-1 &\lambda\\
1&0&0
\end{pmatrix}\begin{pmatrix}
0&0&1\\
1 &0 &0\\
0&1&0
\end{pmatrix}
=\begin{pmatrix}
0&1&0\\
-1 &\lambda &\lambda\\
0&0&1
\end{pmatrix}.
\end{align}

On the other hand, by (\ref{eq-52}) we have
\begin{align*}
\hat{y}_2\big|_{\ell}T=\hat{y}_1\big|_{\ell}T+2\beta z\hat{y}_3+\hat{m}_1^*+3\beta \hat{y}_3
=-\hat{y}_1+2\hat{y}_2+2\beta\hat{y}_3.
\end{align*}
Hence (\ref{eq-57}) yields $\lambda=2$ and $\beta=1$. This proves
(a). In particular, it is easy to see from (\ref{eq-55})--(\ref{eq-57})
that the representation $\hat{\rho}$ is irreducible, i.e., there is no
proper nontrivial subspace of
$\mathbb{C}\hat{y}_1+\mathbb{C}\hat{y}_2+\mathbb{C}\hat{y}_3$ that is
invariant under $\hat\rho(\SL(2,\mathbb{Z}))$.

(b) To prove (b), we note that (\ref{eq-51}) and (\ref{eq-52}) become
\begin{gather}\label{eq-51--}
\hat{y}_1(z)=z^2\hat{y}_3(z)+z\hat{m}_1^*(z)+\hat{m}_2(z),
\\
\label{eq-52--}
\hat{y}_2=\hat{y}_1+2z\hat{y}_3+\hat{m}_1^*+ \hat{y}_3,
\end{gather}
which implies that
\[y^*(z):=2zy_3(z)+m_1^*(z)\]
is also a solution of (\ref{eq-10}) and
\[\hat{y}_2=\hat{y}_1+\hat{y}^*+\hat{y}_3,\qquad\text{where}\quad \hat{y}^*=Fy^*=2z\hat{y}_3+\hat{m}_1^*.\]
By (\ref{eq-55}) and $\lambda=2$, we have
\[
2\big(\tfrac{-1}{z}\big)\hat{y}_1+\hat{m}_1^*\big|_{\ell}S=\hat{y}^*\big|_{\ell}S=(\hat{y}_2-\hat{y}_1-\hat{y}_3)\big|_{\ell}S
=\hat{y}_1-\hat{y}_2+\hat{y}_3=-\hat{y}^*=-2z\hat{y}_3-\hat{m}_1^*.\]
From here and (\ref{eq-51--}), we obtain
\begin{equation}\label{eq-58}
z (\hat{m}_1^*\big|_\ell S)(z)=z\hat{m}_1^*(z)+2\hat{m}_2(z).
\end{equation}

On the other hand, by (\ref{eq-51--}),
\[\hat{y}_3=\hat{y}_1\big|_{\ell}S=\big(\tfrac{-1}{z}\big)^2\hat{y}_1
+\big(\tfrac{-1}{z}\big)(\hat{m}_1^*\big|_{\ell}S)+\hat{m}_2\big|_{\ell}S,\]
which implies
\[\hat{y}_1=z^2\hat{y}_3+z(\hat{m}_1^*\big|_\ell S)-z^2(\hat{m}_2\big|_{\ell}S).\]
Again by (\ref{eq-51--}), we have
\begin{equation}\label{eq-59}
z(\hat{m}_1^*\big|_{\ell}S)(z)=z\hat{m}_1^*(z)+\hat{m}_2(z)+z^2(\hat{m}_2\big|_{\ell}S)(z).
\end{equation}
Thus (\ref{eq-58}) and (\ref{eq-59}) imply $\hat{m}_2\big|_{\ell-2}S =z^2(\hat{m}_2\big|_{\ell}S)=\hat{m}_2$. This proves that $\hat{m}_2(z)$ is a modular form of weight $\ell-2$.

Recalling $E_2\big|_2 S=E_2+6/\pi {\rm i} z$,
it follows from \eqref{eq-53} and \eqref{eq-58} that
\begin{align*}
z\left(\hat{m}_1+\frac{\pi {\rm i} }{3}\hat{m}_2E_2\right)+2\hat{m}_2
& =z (\hat{m}_1^*\big|_\ell S)
=z\left\{\hat{m}_1\big|_{\ell}S+\frac{\pi {\rm i}}{3}(\hat{m}_2\big|_{\ell-2}S) \left(E_2+\frac{6}{\pi {\rm i}z}\right)\right\}\\
&=z(\hat{m}_1\big|_{\ell}S)+\frac{\pi {\rm i} z}{3}\hat{m}_2E_2+2\hat{m}_2,
\end{align*}
which yields $\hat{m}_1\big|_{\ell}S=\hat{m}_1$. This proves that $\hat{m}_1(z)$ is a modular form of weight $\ell$.

Finally, to prove the modularity of $\hat{m}_0(z)$, we use (\ref{eq-51--}) and (\ref{eq-54}) to obtain
\begin{gather*}
z^2\hat{y}_3+z\left(\hat{m}_1+\frac{\pi {\rm i} }{3}\hat{m}_2E_2\right)+\hat{m}_2
=\hat{y}_1=\hat{y}_3\big|_{\ell}S\\
\qquad {} = \left(\frac{\pi {\rm i} z}{6}\right)^2(\hat{m}_2\big|_{\ell-2}S) \Big(E_2+\frac{6}{\pi {\rm i} z}\Big)^2+\frac{\pi {\rm i} z^2}{6}(\hat{m}_1\big|_{\ell}S) \Big(E_2+\frac{6}{\pi {\rm i} z}\Big)+\hat{m}_0\big|_{\ell}S\\
\qquad{} = z^2\hat{y}_3
+z\left(\hat{m}_1+\frac{\pi {\rm i} }{3}\hat{m}_2E_2\right)+\hat{m}_2+\hat{m}_0\big|_{\ell}S-z^2\hat{m}_0,
\end{gather*}
which implies $\hat{m}_0\big|_{\ell+2}S=z^{-2}(\hat{m}_0\big|_{\ell}S)=\hat{m}_0$. This proves that $\hat{m}_0(z)$ is a modular form of weight $\ell+2$. The proof is complete.
\end{proof}

Clearly the above arguments imply Theorem \ref{thm-01+}(1).

\subsection{Proofs of Theorems \ref{thm-01} and \ref{thm-01+}, and Corollary~\ref{coro-1-3}} In this section,
we complete the proof of Theorems \ref{thm-01} and \ref{thm-01+}, and Corollary \ref{coro-1-3}. First we need the following general observation.

\begin{Lemma}\label{lem-3-9} Let $(S1)$--$(S2)$ hold. Then the eigenvalues of $\hat\rho(R)$ are precisely
\begin{equation}\label{e2ee}
{\rm e}^{-\frac{\pi {\rm i}}{3}(\ell+2\kappa)},\qquad\kappa\in\big\{0, \kappa_{\rho}^{(2)}-\kappa_{\rho}^{(1)}, \kappa_{\rho}^{(3)}-\kappa_{\rho}^{(1)}\big\},
\end{equation}
and
the eigenvalues of $\hat{\rho}(S)$ are precisely
\begin{equation}\label{e2ee--}
i^{-\ell-2\kappa}, \qquad\kappa\in\big\{0, \kappa_{\rm i}^{(2)}-\kappa_{\rm i}^{(1)}, \kappa_{\rm i}^{(3)}-\kappa_{\rm i}^{(1)}\big\}.
\end{equation}
\end{Lemma}

\begin{proof} Under our assumption (S2) that $\rho$ is apparent, it follows from Remark~\ref{remark: about coefficients} and Lem\-ma~\ref{lemma: local solutions} below that any solution $y(z)$ of (\ref{eq-10})
 has an expansion of the form
\[
 \frac1{(1-w)^2}\sum_{n=0}^\infty a_nw^{n+\kappa}, \qquad
 w=w(z)=\frac{z-\rho}{z-\overline\rho}
 \]
 at $\rho$, where $a_0\neq0$ and
 $\kappa\in\big\{\kappa_{\rho}^{(j)}\colon j=1,2,3\big\}$. Hence, $\hat y(z)$ has
 an expansion of the form
 \[
 (1-w)^\ell\sum_{n=0}^\infty b_nw^{n+\kappa}, \qquad
 \kappa\in\big\{0,\kappa_{\rho}^{(2)}-\kappa_{\rho}^{(1)},
 \kappa_{\rho}^{(3)}-\kappa_{\rho}^{(1)}\big\}
\]
 with $b_0\neq 0$. Recalling that $\rho=\big({-}1+\sqrt{3}{\rm i}\big)/2$ is a fixed point of $R=\bigl(\begin{smallmatrix}0 & -1\\
1 & 1\end{smallmatrix}\bigr)$, we denote
\[
\zeta:=c\rho+d=\rho+1={\rm e}^{\pi {\rm i}/3}.
\] Then a direct computation gives
\[w(R z)=\zeta^{-2}w(z),
\qquad
 1-w(R z)
 =\zeta^{-1}(z+1)(1-w(z)),
 \]
so
\[
 \left(\hat y\big|_\ell R\right)(z)
 =\zeta^{-\ell}(1-w)^\ell\sum_{n=0}^\infty b_n\big(\zeta^{-2}w\big)^{n+k}.
 \] Therefore, $\hat y(z)$ is an eigenfunction of $\hat\rho(R)$ if and
 only if the series expansion of $\hat y(z)$ is of the form
\[
 (1-w)^\ell\sum_{n=0}^\infty b_{3n}w^{3n+\kappa}, \qquad
 \kappa\in\big\{0,\kappa_{\rho}^{(2)}-\kappa_{\rho}^{(1)},
 \kappa_{\rho}^{(3)}-\kappa_{\rho}^{(1)}\big\}
\]
 and the corresponding eigenvalue is
 $\zeta^{-\ell-2\kappa}={\rm e}^{-\frac{\pi {\rm i}}{3}(\ell+2\kappa)}$. Note
 from $\hat{\rho}(R)^{3}=(-1)^\ell I_3=I_3$ that~$\hat{\rho}(R)$ can
 be diagonizalied. We claim that the eigenvalues of~$\hat\rho(R)$
 are precisely those in~(\ref{e2ee}).

Indeed, for any $\kappa\in\big\{0,\kappa_{\rho}^{(2)}-\kappa_{\rho}^{(1)},
 \kappa_{\rho}^{(3)}-\kappa_{\rho}^{(1)}\big\}$, we define
 \[N_{\kappa}:=\#\big\{\tilde{\kappa}\in \big\{0,\kappa_{\rho}^{(2)}-\kappa_{\rho}^{(1)},
 \kappa_{\rho}^{(3)}-\kappa_{\rho}^{(1)}\big\} \,\big|\, \zeta^{-\ell-2\tilde{\kappa}}=\zeta^{-\ell-2\kappa}\big\}.\]
 Clearly (\ref{e2ee}) holds if $N_{\kappa}=3$ for some $\kappa$ (and so for all $\kappa$). So we only consider the case $N_{\kappa}\in \{1,2\}$ for all~$\kappa$.
 Assume by contradiction that there are $N_{\kappa}+1\in \{2,3\}$ linearly independent eigenfunctions \[\hat{y}_j=(1-w)^\ell\sum_{n=0}^\infty b_{j,3n}w^{3n+\kappa},\qquad b_{1,0}=b_{2,0}=b_{N_{\kappa}+1,0}\neq 0, \qquad 1\leq j\leq N_{\kappa}+1,\] corresponding to the same eigenvalue $\zeta^{-\ell-2\kappa}$. Then
 \[\hat{y}_1-\hat{y}_2=(1-w)^\ell\sum_{n=n_0}^\infty (b_{1,3n}-b_{2,3n})w^{3n+\kappa}\]
 is also an eigenfunction of $\zeta^{-\ell-2\kappa}$, where $n_0\geq 1$ is the smallest integer such that $b_{1,3n_0}-b_{2,3n_0}\allowbreak \neq 0$. This implies $\kappa,\kappa+3n_0 \in\big\{0,\kappa_{\rho}^{(2)}-\kappa_{\rho}^{(1)},
 \kappa_{\rho}^{(3)}-\kappa_{\rho}^{(1)}\big\}$, already a contradiction if $N_{\kappa}=1$. If $N_{\kappa}=2$, then by using the linear combination of $\hat{y}_1$, $\hat{y}_2$, $\hat{y}_3$, there is another $n_1\geq 1$ satisfying $n_1\neq n_0$ such that $\kappa,\kappa+3n_0, \kappa+3n_1 \in\big\{0,\kappa_{\rho}^{(2)}-\kappa_{\rho}^{(1)},
 \kappa_{\rho}^{(3)}-\kappa_{\rho}^{(1)}\big\}$, again a contradiction with $N_{\kappa}=2$.
 Thus, for any $\kappa$, the dimension of eigenfunctions of $\zeta^{-\ell-2\kappa}$ is at most $N_{\kappa}$. This implies the assertion~(\ref{e2ee}).

The proof of (\ref{e2ee--}) is similar and is omitted here.
\end{proof}

Note that if (S3) does not hold, it follows from $\kappa_{\rm i}^{(1)}+\kappa_{\rm i}^{(2)}+\kappa_{\rm i}^{(3)}=3$ that $\big\{3\kappa_{\rm i}^{(1)},3\kappa_{\rm i}^{(2)},
 3\kappa_{\rm i}^{(3)}\big\}\equiv\{1,1,1\} \mod 2$.
We have

\begin{Theorem}\label{prop-110}
Let $(S1)$--$(S2)$ hold and suppose $\big\{3\kappa_{\rm i}^{(1)},3\kappa_{\rm i}^{(2)},
 3\kappa_{\rm i}^{(3)}\big\}\equiv\{1,1,1\} \mod 2$. Then $12 | \ell$ and for any solution $y(z)$ of~\eqref{eq-10}, $\hat{y}(z)$ is a modular form of weight~$\ell$. In particular, the representation $\hat\rho$ is trivial.
\end{Theorem}

\begin{proof}
By (\ref{e2ee--}) and $\big\{3\kappa_{\rm i}^{(1)},3\kappa_{\rm i}^{(2)},
 3\kappa_{\rm i}^{(3)}\big\}\equiv\{1,1,1\}$ $\mod 2$, we see that the eigenvalues of $\hat{S}$ are all the same, so we see from $\hat S^2=I_3$ that $\hat S=I_3$. Consequently, $\hat{T}=\hat{R}$ and then $\hat{T}^3=\hat{R}^3=I_3$. Since the eigenvalues of $\hat{T}$ are $\{1,1,1\}$, we obtain $\hat R=\hat T=I_3$, i.e., the representation $\hat\rho$ is trivial and $\hat{y}(z)$ is a modular form of weight $\ell$ for any solution~$y(z)$. Furthermore, it follows from Lemma~\ref{lem-3-9} that ${\rm e}^{-\frac{\pi {\rm i}}{3}\ell}={\rm i}^{-\ell}=1$, so $\ell\equiv 0 \mod 12$.
\end{proof}

\begin{proof}[Proof of Theorems \ref{thm-01} and~\ref{thm-01+}]
%Clearly
Theorem~\ref{thm-01} follows from Theorems~\ref{thm-sl-2} and~\ref{prop-110}.
\end{proof}

\begin{proof}[Proof of Corollary \ref{coro-1-3}]
Under the assumptions (H1)--(H3) and $\kappa_{\rho}^{(1)}\in\mathbb{Z}$, we have $\kappa_{\rho}^{(j)}\in\mathbb{Z}$ for all $j$. Together with $\kappa_\rho^{(1)}+\kappa_\rho^{(2)}+\kappa_\rho^{(3)}=3$, we have
either
$\kappa_\rho^{(1)}\equiv\kappa_\rho^{(2)}\equiv\kappa_\rho^{(3)}\mod
3$ or
$\big\{\kappa_\rho^{(1)},\kappa_\rho^{(2)},\kappa_\rho^{(3)}\big\}\equiv\{0,1,2\}\mod
3$.

First suppose $\big\{3\kappa_{\rm i}^{(1)},3\kappa_{\rm i}^{(2)},
 3\kappa_{\rm i}^{(3)}\big\}\equiv\{0,0,1\} \mod 2$. Then
Theorem~\ref{thm-sl-2} holds, in particular, $\hat{R}\neq I_3$ and the eigenvalues can not be all the same. This together with~(\ref{e2ee}) imply that $\kappa_\rho^{(1)}\equiv\kappa_\rho^{(2)}\equiv\kappa_\rho^{(3)}\mod
3$ is impossible, so
\begin{equation}\label{eq-130}
\big\{\kappa_{\rho}^{(1)},\kappa_{\rho}^{(2)},
 \kappa_{\rho}^{(3)}\big\}\equiv\{0,1,2\}\quad \mod 3.
\end{equation}

Conversely, suppose (\ref{eq-130}) holds. If $\big\{3\kappa_{\rm i}^{(1)},3\kappa_{\rm i}^{(2)},
 3\kappa_{\rm i}^{(3)}\big\}\equiv\{1,1,1\} \mod 2$, then Theorem~\ref{prop-110} implies $\hat{R}=I_3$, which together with (\ref{e2ee}) imply
$\kappa_\rho^{(1)}\equiv\kappa_\rho^{(2)}\equiv\kappa_\rho^{(3)}\mod
3$, a contradiction with~(\ref{eq-130}). Thus $\big\{3\kappa_{\rm i}^{(1)},3\kappa_{\rm i}^{(2)},
 3\kappa_{\rm i}^{(3)}\big\}\equiv\{0,0,1\} \mod 2$.
\end{proof}

\begin{Remark} \label{remark: Tuba}
 We note the under the assumption that the eigenvalues
 of $\hat\rho(T)$ are all~$1$, Proposition~2.5 and Corollary to
 Theorem~2.9 of~\cite{Tuba-Wenzl} and results of~\cite{Westbury}
 imply that $\hat\rho$ is irreducible if and only if
 the eigenvalues of $\hat\rho(S)$ and $\hat\rho(R)$ are $1$, $-1$, $-1$ and
 $1$, ${\rm e}^{2\pi {\rm i}/3}$, ${\rm e}^{-2\pi {\rm i}/3}$, respectively. Our Theorem~\ref{thm-01} shows that the irreducibility property of $\hat\rho$ is
 solely determined by the local exponents at~${\rm i}$. This link between
 the results of \cite{Tuba-Wenzl,Westbury} and Theorem~\ref{thm-01}
 is provided by Lemma~\ref{lem-3-9}. In other words, one may also use
 results of \cite{Tuba-Wenzl,Westbury} and Lemma~\ref{lem-3-9} to
 give an alternative proof of Theorem~\ref{thm-01}. Our approach has
 the advantage that it directly shows that~$\hat y_+(z)$ is a
 quasimodular form of depth~$2$.
 (Note that Westbury's paper~\cite{Westbury} does not seem to
 be easily available. We refer the reader to the introduction
 section of~\cite{LeBruyn} for a quick review of Westbury's results.)
\end{Remark}

\section[Reducibility and SU(3) Toda systems on SL(2,Z)]{Reducibility and $\boldsymbol{\SU(3)}$ Toda systems on $\boldsymbol{\SL(2,\Z)}$}\label{section-6}

In view of Theorem \ref{prop-110} or equivalently Theorem~\ref{thm-01+}(2), it is natural to ask whether the converse statement holds or not.
The purpose of this section is to establish such a converse statement and apply it to the~$\SU(3)$ Toda system.
Let $\Gamma$ be a discrete subgroup of $\SL(2,\R)$ commensurable with~$\SL(2,\Z)$.
In general, there are at least three sources of modular forms and
quasimodular forms that will give rise to third-order MODEs on~$\Gamma$:
\begin{enumerate}
 \item[(i)] If $f(z)\in\wt\sM_k^{\le 2}(\Gamma,\chi)$, then
 $f(z)/\sqrt[3]{W_f(z)}$ satisfies a third-order MODE on $\Gamma$. This case has been studied in Section~\ref{section-2}.
 \item[(ii)] If $f(z)=f_1(z)\phi(z)+f_0(z)\in\wt\sM_k^{\le 1}(\Gamma,\chi_1)$ and
 $g(z)\in\sM_{k-1}(\Gamma,\chi_2)$ with $\chi_1(-I_2)\chi_2(-I_2)\allowbreak =-1$,
 then a similar argument as Theorem~\ref{thm-1} shows that
 \[
 f(z)/\sqrt[3]{W_{f,g}(z)} , \qquad (zf+\alpha f_1)/\sqrt[3]{W_{f,g}(z)}\qquad \text{and}\qquad g(z)/\sqrt[3]{W_{f,g}(z)}
 \]
 are solutions of some third-order MODE on~$\Gamma$. Here
 \[W_{f,g}(z)=\det\begin{pmatrix}f&f'&f''\\zf+\alpha f_1&(zf+\alpha f_1)'&(zf+\alpha f_1)''\\g&g'&g''\end{pmatrix}.\]

 \item[(iii)] If $f(z)\in\sM_k(\Gamma,\chi_1)$,
 $g(z)\in\sM_k(\Gamma,\chi_2)$, and $h(z)\in\sM_k(\Gamma,\chi_3)$
 for some characters $\chi_j$ of~$\Gamma$, then
 $f(z)/\sqrt[3]{W_{f,g,h}(z)}$, $g(z)/\sqrt[3]{W_{f,g,h}(z)}$,
 $h(z)/\sqrt[3]{W_{f,g,h}(z)}$ are solutions of some third-order
 MODE on $\Gamma$. Here
 \begin{equation} \label{equation: Wfgh}
 W_{f,g,h}(z)=\det\begin{pmatrix}f&f'&f''\\g&g'&g''\\h&h'&h''\end{pmatrix}.
 \end{equation}
\end{enumerate}
To simplify the situation, we impose the condition that the values of
$\chi_j(T)$ in the second and the third cases are all the same (so
that $\rho(T)$ has only one eigenvalue with multiplicity $3$). In the
case $\Gamma=\SL(2,\Z)$, this condition implies that $\chi_j$ are all
the same, say, $\chi_j=\chi$ for all $j$, so Case (ii) will not
occur. Moreover, in Case (iii), we can divide $f$, $g$, $h$ by an
eta-power $\eta(z)^m$ satisfying ${\rm e}^{2\pi {\rm i}m/24}=\chi(T)$. The
differential equation corresponding to $f(z)/\eta(z)^m$ is the same as
that corresponding to $f(z)$. Thus, in the case $\Gamma=\SL(2,\Z)$, we
may assume that $\chi$ is trivial.

\begin{Lemma} Let $f,g,h\in\sM_k(\SL(2,\Z))$ be three linearly
 independent modular forms of weight~$k$ on $\SL(2,\Z)$. Define
 $W_{f,g,h}(z)$ by~\eqref{equation: Wfgh}. Then $W_{f,g,h}$ is a
 modular form of weight $3(k+2)$ on $\SL(2,\Z)$.
\end{Lemma}

\begin{proof} Let $F(z)=(f(z),g(z),h(z))^t$, which satisfies $F(\gamma
 z)=(cz+d)^kF(z)$ for all $\gamma=\SM abcd\in\SL(2,\Z)$. Then the
 assertion follows from the basic properties of the determinant function.
\end{proof}

\begin{Theorem}\label{thm-66} Let $f,g,h\in\sM_k(\SL(2,\Z))$ be linearly
 independent modular forms and $\mathcal Ly=0$ be the differential
 equation satisfied by $f/\sqrt[3]{W_{f,g,h}}$,
 $g/\sqrt[3]{W_{f,g,h}}$, and $h/\sqrt[3]{W_{f,g,h}}$. Then
$Q_2(z)$ and $Q_3(z)-\frac12 Q_2'(z)$ are meromorphic modular forms of weight~$4$ and~$6$ respectively. Furthermore, $(H1)$--$(H3)$ hold and~$\infty$ is also apparent.
\end{Theorem}

\begin{proof} The proof is similar as that of Theorem~\ref{thm-1}. The only difference is that $\infty$ is also apparent because $f/\sqrt[3]{W_{f,g,h}}$, $g/\sqrt[3]{W_{f,g,h}}$ and $h/\sqrt[3]{W_{f,g,h}}$ are linearly
 independent solutions.
\end{proof}

\begin{Proposition}\label{prop-6-3} Let $f,g,h\in\sM_k(\SL(2,\Z))$ be linearly
 independent modular forms and $\mathcal Ly=0$ be the differential
 equation satisfied by $f/\sqrt[3]{W_{f,g,h}}$,
 $g/\sqrt[3]{W_{f,g,h}}$, and $h/\sqrt[3]{W_{f,g,h}}$. Then the local
 exponents of $\mathcal Ly=0$ at the elliptic points ${\rm i}$ and
 $\rho=\big({-}1+\sqrt{3}{\rm i}\big)/2$ satisfy
 \begin{enumerate}
 \item[$1)$]
 $\kappa_{\rm i}^{(2)}-\kappa_{\rm i}^{(1)},\kappa_{\rm i}^{(3)}-\kappa_{\rm i}^{(1)}\equiv
 0\mod 2$, and
 \item[$2)$] $\kappa_\rho^{(j)}\in\Z$ for all $j$, and
 $\kappa_\rho^{(1)}\equiv\kappa_\rho^{(2)}\equiv\kappa_\rho^{(3)}\mod 3$.
 \end{enumerate}
\end{Proposition}

\begin{proof} Note that every solution $y(z)$ of $\mathcal Ly=0$ can
 be written as
 \[
 (af(z)+bg(z)+ch(z))/\sqrt[3]{W_{f,g,h}(z)}
 \]
 for some
 $a,b,c\in\C$. As $a$, $b$, and $c$ vary, the order of $y(z)$ at~${\rm i}$
 (respectively, $\rho$) will go through all possible local exponents
 of $\mathcal Ly=0$ at~${\rm i}$ (respectively,~$\rho$). Since
\[\frac{\ord_{\rm i} (af+bg+cz)}{2}+\frac{\ord_\rho (af+bg+cz)}{3}\equiv \frac{\ord_{\rm i} (f)}{2}+\frac{\ord_\rho (f)}{3}\quad \mod \Z,\]
so
\begin{gather*}
\ord_{\rm i} (af+bg+cz)-\ord_{\rm i} (f)\equiv 0\quad \mod 2, \\
 \ord_{\rho} (af+bg+cz)-\ord_{\rho} (f)\equiv 0\quad \mod 3
 \end{gather*}
hold for any $(a,b,c)\neq (0,0,0)$. From here and
\begin{gather*}
\kappa_{\rm i}^{(j)}+\tfrac{1}{3}\ord_{\rm i}W_{f,g,h}(z)\in\{\ord_{\rm i} (af+bg+cz) \,|\, (a,b,c)\neq (0,0,0)\},\\
 \kappa_{\rho}^{(j)}+\tfrac{1}{3}\ord_{\rho}W_{f,g,h}(z)\in\{\ord_{\rho} (af+bg+cz) \,|\, (a,b,c)\neq (0,0,0)\}
 \end{gather*}
for all $j$,
we obtain the assertion (1) and $\kappa_{\rho}^{(2)}-\kappa_{\rho}^{(1)},\kappa_{\rho}^{(3)}-\kappa_{\rho}^{(1)}\equiv
 0\mod 3$. This together with $\kappa_{\rho}^{(1)}+\kappa_{\rho}^{(2)}+\kappa_{\rho}^{(3)}=3$ imply $\kappa_\rho^{(j)}\in\Z$ for all $j$ and so the assertion~(2) holds.
\end{proof}

The above result is precisely the converse statement of Theorem \ref{thm-01+}(2).

\begin{Example} Recall that the smallest weight $k$ such that
 $\dim\sM_k(\SL(2,\Z))=3$ is $24$. Let $f(z)=E_4(z)^6$,
 $g(z)=E_4(z)^3\Delta(z)$, and $h(z)=\Delta(z)^2$, which form a basis
 for $\sM_{24}(\SL(2,\Z))$. To determine the differential equation
\[
 \mathcal Ly:=D_q^3y(z)+Q(z)D_qy(z)+\left(\frac12D_qQ(z)+R(z)\right)y(z)=0
\]
 satisfied by $f/\sqrt[3]{W_{f,g,h}}$,
 $g/\sqrt[3]{W_{f,g,h}}$, and $h/\sqrt[3]{W_{f,g,h}}$, we
 use Ramanujan's identities
\[
 D_qE_2=\frac{E_2^2-E_4}{12}, \qquad
 D_qE_4=\frac{E_2E_4-E_6}3, \qquad
 D_qE_6=\frac{E_2E_6-E_4^2}2,
\]
 and compute that $W_{f,g,h}(z)=cE_4(z)^6E_6(z)^3\Delta(z)^3$ for
 some nonzero number $c$. Noticing that $W_{f,g,h}(z)$ has a zero of
 order $3$ at~${\rm i}$ and a zero of order~$6$ at~$\rho$, we know that
 $\kappa_{\rm i}^{(1)}=-1$, $\kappa_\rho^{(1)}=-2$, which, by
 Proposition~\ref{prop-6-3}, implies that $\kappa_{\rm i}^{(2)}=1$,
 $\kappa_{\rm i}^{(3)}=3$, $\kappa_\rho^{(2)}=1$, and
 $\kappa_\rho^{(3)}=4$. In other words, the indicial equations at ${\rm i}$
 and at $\rho$ are $(x+1)(x-1)(x-3)=0$ and $(x+2)(x-1)(x-4)=0$,
 respectively. Also, we have $\ord_\infty f-\frac13\ord_\infty
 W_{f,g,h}=-1$, $\ord_\infty g-\frac13\ord_\infty W_{f,g,h}=0$, and
 $\ord_\infty h-\frac13\ord_\infty W_{f,g,h}=1$, which implies that
 the indicial equation at $\infty$ is $(x+1)x(x-1)$. Therefore,
 according to Lemmas~\ref{lemma: Q R general form}, \ref{lemma:
 indicial at infinity}, and~\ref{lemma: indicial at z}, the
 meromorphic modular forms~$Q(z)$ and~$R(z)$ in~$\mathcal Ly$ are
 \[
 Q(z)=-E_4(z)-\frac34\frac{E_4(z)(E_4(z)^3-E_6(z)^2)}{E_6(z)^2}
 +\frac89\frac{E_4(z)^3-E_6(z)^2}{E_4(z)^2}
\]
 (note that this can also be computed directly using \eqref{eq-7})
 and
\[
 R(z)=s_{\rm i}^{(1)}\frac{E_4(z)^3-E_6(z)^2}{E_6(z)}
\]
 for some complex number $s_{\rm i}^{(1)}$.
 Using the apparentness condition at ${\rm i}$, we can show that
 $s_{\rm i}^{(1)}=0$. In other words, the differential equation is
 $D_q^3y(z)+Q(z)D_qy(z)+\frac 12D_qQ(z)y(z)=0$. We remark that the
 reason why the differential equation is of this special form is due
 to the fact that it is the symmetric square of some second order MODE.
\end{Example}

It is worth to point out that Theorem \ref{thm-66} can be applied to construct solutions of the $\SU(3)$ Toda system
\begin{equation*}
\begin{cases}
\displaystyle \Delta v_1+2{\rm e}^{v_1}-{\rm e}^{v_2}=4\pi\sum_{k=1}^N n_{1,k}\delta_{p_k},\\
\displaystyle \Delta v_2+2{\rm e}^{v_2}-{\rm e}^{v_1}=4\pi \sum_{k=1}^N n_{2,k}\delta_{p_k}
\end{cases}\qquad\text{in} \ \mathbb{R}^2,
\end{equation*}
where $\Delta=\frac{\partial^2}{\partial x_1^2}+\frac{\partial^2}{\partial x_2^2}$ is the Laplace operator and $\delta_{p}$ denotes the Dirac measure at $p$. We always use the complex variable $w=x_1+{\rm i}x_2$. Then the Laplace operator $\Delta=4\partial_{w\bar{w}}$.

As in Theorem \ref{thm-66}, we let $f,g,h\in\sM_k(\SL(2,\Z))$ be linearly
 independent modular forms and \begin{gather}\label{eq-10--}
y'''(z)+Q_2(z)y'(z)+Q_3(z)y(z)=0,\qquad z\in\mathbb{H}\end{gather} be the MODE satisfied by $y_1(z):=f(z)/\sqrt[3]{W_{f,g,h}}$,
 $y_2(z):=g(z)/\sqrt[3]{W_{f,g,h}}$, and $y_3(z):=h(z)/\sqrt[3]{W_{f,g,h}}$. Denote the set of regular singular points of~(\ref{eq-10--}) modulo $\SL(2,\mathbb{Z})$ on $\mathbb{H}$ by
$\mathcal{S}=\{z_1,\dots, z_m, {\rm i}, \rho\}$. For each $z\in\mathcal{S}$, it follows from Theorem~\ref{thm-66} that there are $m_{z}^{(1)}, m_{z}^{(2)}\in\mathbb{Z}_{\geq 0}$ such that the local exponents of~(\ref{eq-10--}) at $\gamma z$ are the same as those at $z$ and are given by
\[\kappa_{z}^{(1)}=-\frac{2m_{z}^{(1)}+m_{z}^{(2)}}{3},\qquad \kappa_{z}^{(2)}=\kappa_{z}^{(1)}+m_{z}^{(1)}+1,\qquad \kappa_{z}^{(3)}=\kappa_{z}^{(2)}+m_{z}^{(2)}+1\]
for any $\gamma\in \SL(2,\mathbb{Z})$. Similarly, there are $m_{\infty}^{(1)}, m_{\infty}^{(2)}\in\mathbb{N}$ such that the local exponents of (\ref{eq-10--}) at the cusp $\infty$ are given by
\[\kappa_{\infty}^{(1)}=-\frac{2m_{\infty}^{(1)}+m_{\infty}^{(2)}}{3},\qquad \kappa_{\infty}^{(2)}=\kappa_{\infty}^{(1)}+m_{\infty}^{(1)},\qquad \kappa_{\infty}^{(3)}=\kappa_{\infty}^{(2)}+m_{\infty}^{(2)}.\]

Given any $\lambda,\mu>0$, we define
 \begin{gather*}%\label{n-U1}
 {\rm e}^{-U_{1;\lambda,\mu}(z)}:= \frac{1}{4}
\big(\lambda^2\mu^{-1}|y_1|^2+\mu^2\lambda^{-1}|y_2|^2+\lambda^{-1}\mu^{-1}|y_3|^2\big),
\\
{\rm e}^{-U_{2;\lambda,\mu}(z)}
:=\frac{1}{4}
\big[\lambda\mu|W(y_1,y_2)|^2+\lambda^{-2}\mu|W(y_2,y_3)|^2+\lambda\mu^{-2}|W(y_3,y_1)|^2\big],
\end{gather*}
where $W(y_i, y_j):=y_i'y_j-y_j'y_i$. Note that ${\rm e}^{-U_{k;\lambda,\mu}(z)}$ is single-valued for any $z\in\mathbb{H}$ and $0<{\rm e}^{-U_{k;\lambda,\mu}(z)}<\infty$ as long as $z\notin \SL(2,\mathbb{Z})\mathcal{S}$. We have

\begin{Lemma} Given any $\lambda,\mu>0$, there holds
\begin{equation}\label{2toda0-eq}\begin{cases}
\Delta U_{1;\lambda,\mu}+{\rm e}^{2U_{1;\lambda,\mu}-U_{2;\lambda,\mu}}=0,\\
\Delta U_{2;\lambda,\mu}+{\rm e}^{2U_{2;\lambda,\mu}-U_{1;\lambda,\mu}}=0
\end{cases}\qquad\text{in} \ \mathbb{H}\setminus(\SL(2,\mathbb{Z})\mathcal{S}).\end{equation}
\end{Lemma}

\begin{proof} The proof can be easily adopted from \cite{CL-JDG,LNW,LWY}; we sketch the proof here for the reader's convenience. Given any $\lambda,\mu>0$,
we define
\[\mathcal{W}_{\lambda,\mu}:=\begin{pmatrix}\lambda^{\frac{3}{2}}y_1 & \mu^{\frac{3}{2}}y_2 & y_3\\\lambda^{\frac{3}{2}}y_1' & \mu^{\frac{3}{2}}y_2' & y_3'\\ \lambda^{\frac{3}{2}}y_1'' & \mu^{\frac{3}{2}}y_2'' & y_3''\end{pmatrix}.\]
Since the Wroksian $W(y_1,y_2,y_3)=1$, we have $\det \mathcal{W}_{\lambda,\mu}=(\lambda\mu)^{\frac{3}{2}}$. Define a positive definite matrix
\[R_{\lambda,\mu}:=(\lambda\mu)^{-1}\mathcal{W}_{\lambda,\mu}
\overline{\mathcal{W}_{\lambda,\mu}}^T.\]
then $\det R_{\lambda,\mu}=1$.
For $1\leq m\leq 3$, we let $R_{\lambda,\mu;m}$ denote the leading principal minor of $R_{\lambda,\mu}$ of dimension $m$. Since $y_j(z)$ is holomorphic in $\mathbb{H}\setminus(\SL(2,\mathbb{Z})\mathcal{S})$, a direct computation leads to (see, e.g.,~\cite{LWY})
\begin{equation}\label{n-Rm}
R_{\lambda,\mu;m}(\partial_{z\bar{z}}R_{\lambda,\mu;m})-(\partial_zR_{\lambda,\mu;m})
(\partial_{\bar{z}}R_{\lambda,\mu;m})=R_{{\lambda,\mu;m-1}}R_{\lambda,\mu;m+1},\qquad m=1,2,
\end{equation}
for $z\in \mathbb{H}\setminus(\SL(2,\mathbb{Z})\mathcal{S})$, where $R_{\lambda,\mu;0}:=1$.

On the other hand,
\begin{gather*}\frac{1}{4}R_{\lambda,\mu;1}
=\frac{1}{4}(\lambda\mu)^{-1}\big(\lambda^3|y_1|^2+\mu^3|y_2|^2+|y_3|^2\big)
={\rm e}^{-U_{1;\lambda,\mu}(z)}.%\label{n-eq-52}
\end{gather*}
Define ${\rm e}^{-V_{\lambda,\mu}(z)}:=\frac{1}{4}R_{\lambda,\mu;2}$. we will prove that ${\rm e}^{-V_{\lambda,\mu}(z)}={\rm e}^{-U_{2;\lambda,\mu}(z)}$.

Note that $R_{\lambda,\mu;3}=\det R_{\lambda,\mu}= 1$. Letting $m=1$ in (\ref{n-Rm}) leads to (note $0<{\rm e}^{-U_{1;\lambda,\mu}}, {\rm e}^{-V_{\lambda,\mu}}<+\infty$ in $\mathbb{H}\setminus(\SL(2,\mathbb{Z})\mathcal{S})$)
\begin{align}
4{\rm e}^{-V_{\lambda,\mu}}&=R_{\lambda,\mu;2}=16\big[{\rm e}^{-U_{1;\lambda,\mu}}
\big(\partial_{z\bar{z}}{\rm e}^{-U_{1;\lambda,\mu}}\big)
-\big(\partial_z{\rm e}^{-U_{1;\lambda,\mu}}\big)\big(\partial_{\bar{z}}{\rm e}^{-U_{1;\lambda,\mu}}\big)\big]\nonumber\\
&=-16{\rm e}^{-2U_{1;\lambda,\mu}}\partial_{z\bar{z}}U_{1;\lambda,\mu}=-4{\rm e}^{-2U_{1;\lambda,\mu}}\Delta U_{1;\lambda,\mu}\quad \text{in }\mathbb{H}\setminus(\SL(2,\mathbb{Z})\mathcal{S}),\label{n-V1-eq}
\end{align}
and letting $m=2$ in (\ref{n-Rm}) leads to
\begin{align*}
4{\rm e}^{-U_{1;\lambda,\mu}}&=R_{\lambda,\mu;1}=16\big[{\rm e}^{-V_{\lambda,\mu}}
(\partial_{z\bar{z}}{\rm e}^{-V_{\lambda,,\mu}})
-\big(\partial_z{\rm e}^{-V_{\lambda,\mu}}\big)\big(\partial_{\bar{z}}{\rm e}^{-V_{\lambda,\mu}}\big)\big]\\
&=-16{\rm e}^{-2V_{\lambda,\mu}}\partial_{z\bar{z}}V_{\lambda,\mu}=-4{\rm e}^{-2V_{\lambda,\mu}}\Delta V_{\lambda,\mu}\quad \text{in }\mathbb{H}\setminus(\SL(2,\mathbb{Z})\mathcal{S}).
\end{align*}
Furthermore,
we insert ${\rm e}^{-U_{1;\lambda,\mu}}=\frac{1}{4}\sum|a_jy_j|^2$ (where $a_1=\lambda\mu^{-1/2}$, $a_2=\mu\lambda^{-1/2}$ and $a_3=(\lambda\mu)^{-1/2}$) into (\ref{n-V1-eq}), which leads to
\begin{align*}
\frac{1}{4}{\rm e}^{-V_{\lambda,\mu}}&
={\rm e}^{-U_{1;\lambda,\mu}}\big(\partial_{z\bar{z}}{\rm e}^{-U_{1;\lambda,\mu}}\big)
-\big(\partial_z{\rm e}^{-U_{1;\lambda,\mu}}\big)\big(\partial_{\bar{z}}{\rm e}^{-U_{1;\lambda,\mu}}\big)\\
&=\frac{1}{16}\left[\left(\sum |a_jy_j|^2\right)\left(\sum |a_jy_j'|^2\right)-\left(\sum a_j^{2}y_j'\overline{y_j}\right)\left(\sum a_j^{2}y_j\overline{y_j'}\right)\right]\\
&=\frac{1}{16}\big[|W(a_1 y_1,a_2y_2)|^2+|W(a_2y_2,a_3y_3)|^2
+|W(a_3y_3,a_1 y_1)|^2\big],
\end{align*}
so ${\rm e}^{-V_{\lambda,\mu}(z)}={\rm e}^{-U_{2;\lambda,\mu}(z)}$. This proves that $(U_{1;\lambda,\mu}(z),U_{2;\lambda,\mu}(z))$ solves the Toda system~\eqref{2toda0-eq}.
\end{proof}

Now any $\tilde{z}\in\mathcal{S}$, it follows from the local behavior of~$y_j$'s that near $\gamma \tilde{z}$,
\begin{gather*}
U_{1;\lambda,\mu}(z)=-2\kappa_{\tilde{z}}^{(1)}\ln|z-\gamma \tilde{z}|+O(1),\\
 U_{2;\lambda,\mu}(z)=-2\big(2-\kappa_{\tilde{z}}^{(3)}\big)\ln|z-\gamma \tilde{z}|+O(1).
 \end{gather*}
 Similarly, at the cusp $\infty$, we have
\begin{gather*}
U_{1;\lambda,\mu}(z)=-2\kappa_{\infty}^{(1)}\ln|q|+O(1),\\
 U_{2;\lambda,\mu}(z)=2\kappa_{\infty}^{(3)}\ln|q|+O(1),
\end{gather*}
 where $q={\rm e}^{2\pi {\rm i}z}$. Since $f$, $g$, $h$, $W_{f,g,h}$ are modular forms of weights $k$, $k$, $k$ and $3(k+2)$ respectively, we easily obtain
 \[|y_j(\gamma z)|^2=\frac{|y_j(z)|^2}{|cz+d|^4}, \qquad |W(y_i,y_j)(\gamma z)|^2=\frac{|W(y_i,y_j)(z)|^2}{|cz+d|^4}\]
 for any $\gamma=\SM{a}{b}{c}{d}\in \SL(2,\mathbb{Z})$, so
 \[U_{j;\lambda,\mu}(\gamma z)=U_{j;\lambda,\mu}(z)+4\ln|cz+d|,\qquad j=1,2.\]
Now we define
\[(u_{1;\lambda,\mu}, u_{2;\lambda,\mu}):=(2U_{1;\lambda,\mu}-U_{2;\lambda,\mu},
2U_{2;\lambda,\mu}-U_{1;\lambda,\mu}).\]
Then we have
\[\begin{cases}
\Delta u_{1;\lambda,\mu}+2{\rm e}^{u_{1;\lambda,\mu}}-{\rm e}^{u_{2;\lambda,\mu}}=0,\\
\Delta u_{2;\lambda,\mu}+2{\rm e}^{u_{2;\lambda,\mu}}-{\rm e}^{u_{1;\lambda,\mu}}=0
\end{cases}\qquad\text{in} \ \mathbb{H}\setminus(\SL(2,\mathbb{Z})\mathcal{S}),\]
and near $\gamma \tilde{z}$,
\[u_{1;\lambda,\mu}(z)=2m_{\tilde{z}}^{(1)}\ln|z-\gamma \tilde{z}|+O(1),\qquad
u_{2;\lambda,\mu}(z)=2m_{\tilde{z}}^{(2)}\ln|z-\gamma \tilde{z}|+O(1),\]
while at the cusp $\infty$,
\[u_{1;\lambda,\mu}(z)=2m_{\infty}^{(1)}\ln|q|+O(1),\qquad u_{2;\lambda,\mu}(z)=2m_{\infty}^{(2)}\ln|q|+O(1).\]
Therefore, $(u_{1;\lambda,\mu}, u_{2;\lambda,\mu})$ is a solution of the following $\SU(3)$ Toda system
\begin{equation}\label{Toda-H}
\begin{cases}
\displaystyle \Delta u_1+2{\rm e}^{u_1}-{\rm e}^{u_2}=4\pi\sum_{\gamma}\bigg(m_{\rm i}^{(1)}\delta_{\gamma i}+m_{\rho}^{(1)}\delta_{\gamma \rho}+\sum_{j=1}^m m_{z_j}^{(1)}\delta_{\gamma z_j}\bigg)\quad \text{on} \ \mathbb{H},\\
\displaystyle \Delta u_2+2{\rm e}^{u_2}-{\rm e}^{u_1}=4\pi\sum_{\gamma}\bigg(m_{\rm i}^{(2)}\delta_{\gamma i}+m_{\rho}^{(2)}\delta_{\gamma \rho}+\sum_{j=1}^m m_{z_j}^{(2)}\delta_{\gamma z_j}\bigg)\quad \text{on} \ \mathbb{H},\\
\displaystyle u_k(z)=2m_{\infty}^{(k)}\ln|q|+O(1)\quad\text{as} \ \operatorname{Im} z\to\infty,\\
u_{j}(\gamma z)=u_{j}(z)+4\ln|cz+d|,\quad\forall \gamma\in \SL(2,\mathbb{Z}).
\end{cases}
\end{equation}

Now consider the modular function $w\colon \mathbb{H}\to \mathbb{C}$ defined by
\[w=w(z):=\frac{E_4(z)^3}{E_4(z)^3-E_6(z)^2}.\]
It is well known that $w(z)$ is holomorphic, surjective and
\[w({\rm i})=1,\qquad w(\rho)=0,\qquad w(\infty)=\infty.\]
 A direct computation gives
\[w'(z)=-2\pi {\rm i}\frac{E_4(z)^2 E_6(z)}{E_4(z)^3-E_6(z)^2}.\]
Denote $p_j:=w(\gamma z_j)$. Then all points of $\{p_1,\dots,p_m, 1,0\}\subset\mathbb{C}$ are distinct.
Now we define $(v_1(w),v_2(w))$ for $w\in\mathbb{C}$ by
\[u_k(z)=v_k(w(z))+2\ln |w'(z)|,\qquad z\in\mathbb{H}.\]
Since $w(\gamma z)=w(z)$ gives $w'(\gamma z)=(cz+d)^2w'(z)$, it follows from $
u_k(\gamma z)=u_k(z)+4\ln |cz+d|$ that $v_k(w)$ is well-defined for $w\in\mathbb{C}$.
Now outside $\{p_1,\dots, p_m, 1, 0\}$, we have
\begin{align*}
\Delta u_k(z)&=4\partial_{z\bar{z}}u_k(z)=4\partial_{w\bar{w}}v_k(w)|w'(z)|^2=|w'(z)|^2\Delta v_k(w)\\
&={\rm e}^{u_{k'}(z)}-2{\rm e}^{u_k(z)}=|w'(z)|^2\big({\rm e}^{v_{k'}(w)}-2{\rm e}^{v_k(w)}\big),
\end{align*}
so
\[\Delta v_k(w)+2{\rm e}^{v_k(w)}-{\rm e}^{v_{k'}(w)}=0,\]
where $\{k,k'\}=\{1,2\}$. Furthermore, since $w'(z_j)\neq 0$, we have that at $w=p_j=w(z_j)$,
\begin{align*}v_k(w) =u_k(z)-2\ln |w'(z)|=2m_{z_j}^{(k)}\ln |z-z_j|+O(1)
 =2m_{z_j}^{(k)}\ln |w-p_j|+O(1).\end{align*}
At $w=w({\rm i})=1$, since $\ord_{\rm i}(w-1)=2$ and $\operatorname{ord}_{\rm i}w'=1$, we have
\begin{align*}v_k(w)&=u_k(z)-2\ln |w'(z)|
=2\big(m_{{\rm i}}^{(k)}-1\big)\ln |z-{\rm i}|+O(1)\\
&=\big(m_{{\rm i}}^{(k)}-1\big)\ln |w-1|+O(1).
\end{align*}
At $w=w(\rho)=0$, since $\operatorname{ord}_{\rho}w=3$ and $\operatorname{ord}_{ \rho}w'=2$, we have
\begin{align*}v_k(w)& =u_k(z)-2\ln |w'(z)|
=2\big(m_{\rho}^{(k)}-2\big)\ln |z-\rho|+O(1) \\
 &=\frac{2\big(m_{\rho}^{(k)}-2\big)}{3}\ln |w|+O(1).
 \end{align*}
At $w=w(\infty)=\infty$, since
\[w(z)=Cq^{-1}(1+O(q)),\qquad w'(z)=-2\pi {\rm i} C q^{-1}(1+O(q)),\]
where $q={\rm e}^{2\pi {\rm i} z}$ and $C\neq 0$ is a constant, so
\begin{align*}v_k(w)&=u_k(z)-2\ln |w'(z)|
 =2\big(m_{\infty}^{(k)}+1\big)\ln |q|+O(1)\\
&=-2\big(m_{\infty}^{(k)}+1\big)\ln |w|+O(1).\end{align*}
Therefore, (\ref{Toda-H}) is equivalent to
\begin{equation}\label{Toda-more-general}
\begin{cases}
\displaystyle \Delta v_1+2{\rm e}^{v_1}-{\rm e}^{v_2}=4\pi\frac{m_{{\rm i}}^{(1)}-1}{2}\delta_1+4\pi\frac{m_{\rho}^{(1)}-2}{3}\delta_0+ 4\pi\sum_{j=1}^m m_{z_j}^{(1)}\delta_{p_j}\ \text{in} \ \mathbb{R}^2,\\
\displaystyle \Delta v_2+2{\rm e}^{v_2}-{\rm e}^{v_1}=4\pi\frac{m_{{\rm i}}^{(2)}-1}{2}\delta_1+4\pi\frac{m_{\rho}^{(2)}-2}{3}\delta_0+4\pi \sum_{j=1}^m m_{z_j}^{(2)}\delta_{p_j}\ \text{in} \ \mathbb{R}^2,\\
v_k(w)=-2\big(m_{\infty}^{(k)}+1\big)\ln|w|+O(1)\quad\text{as} \ |w|\to\infty.
\end{cases}
\end{equation}
Note from Proposition \ref{prop-6-3} that
\[
m_{{\rm i}}^{(k)}+1=\kappa_{{\rm i}}^{(k+1)}-\kappa_{{\rm i}}^{(k)} \equiv 0 \mod 2 \qquad \text{and} \qquad
m_{\rho}^{(k)}+1=\kappa_{\rho}^{(k+1)}-\kappa_{\rho}^{(k)}\equiv 0 \mod 3,
\] so
$\frac{m_{{\rm i}}^{(k)}-1}{2}\in\mathbb{Z}_{\geq 0}$ and $\frac{m_{\rho}^{(2)}-2}{3}\in\mathbb{Z}_{\geq 0}$ for $k=1,2$.
In conclusion, starting from any given linearly independent modular
forms $f,g,h\in\sM_k(\SL(2,\Z))$, we can construct a two-parametric
family of solutions to certain Toda system~(\ref{Toda-more-general}):

\begin{Theorem}$(v_{1;\lambda,\mu}(w), v_{2;\lambda,\mu}(w))$ defined by
\[v_{k;\lambda,\mu}(w(z))=u_{k;\lambda,\mu}(z)-2\ln |w'(z)|,\qquad z\in\mathbb{H}\] are a two-parametric family of solutions of the $\SU(3)$ Toda system~\eqref{Toda-more-general}, where $\lambda,\mu>0$ can be arbitrary.
\end{Theorem}

\section{Polynomial systems derived from the conditions (H1)--(H3)}\label{section-4}

In view of Theorem \ref{thm-01} proved in Section~\ref{section-5}, a
natural question is whether given a prescribed set of singular points
and the local exponents at singularities and at the cusps, there exist
MODEs (\ref{eq-01}) satisfying the conditions (H1)--(H3). We will see
in this section that this problem of existence is equivalent to that
of solving a certain system of polynomial equations. Note that in
view of Theorems~\ref{thm-02} and~\ref{thm-66}
such a MODE~\eqref{eq-01} exists for certain sets of data.

\subsection{Solution expansions for MODEs}
Let $\Gamma$ be a discrete subgroup of $\SL(2,\mathbb{R})$ commensurable with $\SL(2,\mathbb{Z})$ and (\ref{eq-01}) be a MODE on $\Gamma$. To verify the apparentness of a singular point $z_0$ of (\ref{eq-01}), we use the classical Frobenius method. However, since (\ref{eq-01}) is modular, it will be more convenient that all functions are expanded in terms of $\wt w$ as introduced in~\cite{LY} rather than~$z-z_0$.

Fix $z_0\in \H$, we let $w=(z-z_0)/(z-\overline z_0)$ and
\begin{equation} \label{eq: wt w 0}
 \wt w=\frac{w}{1+Aw},\qquad\text{with}\quad
 A=\frac{4\pi\phi^\ast(z_0)\Im z_0}{\alpha},
\end{equation}
where $\phi(z)$ is the quasimodular form of weight $2$ and depth $1$ on $\Gamma$, i.e.,
\begin{align*}
\big(\phi\big|_{2}\gamma\big)(z)=\phi(z)+\frac{\alpha_0 c}{2\pi {\rm i}(cz+d)},\qquad \gamma=\begin{pmatrix}
a & b\\
c & d
\end{pmatrix}
\in \Gamma,
\end{align*}
for some nonzero complex number $\alpha_0$ (note $\alpha_0=2\pi {\rm i} \alpha$ for $\alpha$ given in (\ref{eqe-5})), and
 $\phi^*(z):=\phi(z)+\frac{\alpha_0}{2\pi {\rm i} (z-\bar{z})}$. Clearly $\phi^*(z)$ satisfies
\begin{equation*}
\big(\phi^\ast\big|_{2}\gamma\big)(z)=\phi^\ast(z),\qquad \gamma\in\Gamma.
\end{equation*}

The advantage of the expansion in terms of $\wt w$ is the following result.

\begin{Proposition}[{\cite[Propositions A.4 and A.7]{LY}}]\label{prop-2}
Let $f(z)$ be a meromorphic modular form of weight $k$ on
$\Gamma$. Then~$f(z)$ admits an expansion of the form
\[
f(z)=(1-(1+A)\wt w)^k\sum_{n=n_0}^\infty
a_n(-4\pi(\Im z_0)\widetilde{w})^n.
\]
Furthermore, if $z_0$ is an elliptic point with the stabilizer subgroup $\Gamma_{z_0}$ of order $N$, then $a_n=0$ whenever $k+2n\not\equiv 0 \mod N$.
\end{Proposition}

\begin{Remark} \label{remark: about coefficients}
 Note that when $f(z)$ is a holomorphic modular form,
 the coefficients $a_n$ in the series can be expressed in terms of
 the Serre derivatives of $f(z)$. See \cite[Proposition A.4]{LY} for
 the precise statement.

 Also note that when $z_0$ is an elliptic point, we have $A=0$ and
 hence $\wt w=w$. This is because~$\phi^\ast$ transforms like a
 modular form of weight $2$, but any modular form of weight $2$ will
 vanish at every elliptic point.
\end{Remark}

By Proposition \ref{prop-2}, we can write
\begin{gather*}%\label{eq-12}
Q(z):=\frac{Q_2(z)}{(2\pi {\rm i})^2}=(1-(1+A)\wt w)^4\sum_{n=-2}^\infty
a_n(-4\pi(\Im z_0)\widetilde{w})^n,
\\
%\label{eq-13}
R(z):=\frac{Q_3(z)-\frac{1}{2}Q_2'(z)}{(2\pi {\rm i})^3}=(1-(1+A)\wt w)^6\sum_{n=-3}^\infty
b_n(-4\pi(\Im z_0)\widetilde{w})^n.
\end{gather*}
Then (\ref{eq-01}) is equivalent to
\begin{equation} \label{equation: DE in y}
 D_q^3y(z)+Q(z)D_qy(z)+\left(\frac12D_qQ(z)+R(z)\right)y(z)=0,
 \end{equation}
 where $D_q:=q\frac{{\rm d}}{{\rm d}q}=\frac1{2\pi {\rm i}}\frac{{\rm d}}{{\rm d}z}$.
In particular, Proposition~\ref{prop-2} yields
\begin{Corollary}\label{coro}
Suppose $-I_2\in\Gamma$ and let $z_0$ be the elliptic point of order $e$ of $\Gamma$.
Then $N=2e$ and so $a_n=0$ if $n\not\equiv e-2\mod e$ and $b_n=0$ if
$n\not\equiv e-3 \mod e$.
\end{Corollary}

For later usage, we recall Bol's identity in the following
form.

\begin{Lemma} \label{lemma: Bol} Let $\gamma=\SM abcd\in\GL(2,\C)$ and $r$ be a positive integer.
 Set \[ w=\gamma z=(az+b)/(cz+d).\] Then for a $(r+1)$-th differentiable
 function~$g(z)$, we have
\[
 \frac{{\rm d}^{r+1}}{{\rm d}z^{r+1}}\left(\frac{(\det\gamma)^{r/2}}
 {(a-cw)^r}g(w)\right)
 =\frac{(a-cw)^{r+2}}{(\det\gamma)^{r/2+1}}
 \frac{{\rm d}^{r+1}}{{\rm d}w^{r+1}}g(w).
\]
\end{Lemma}

\begin{proof} Bol's identity states that
 \[
 \big(y\big|_{-r}\gamma\big)^{(r+1)}(z)
 =\big(y^{(r+1)}\big|_{r+2}\gamma\big)(z).
\]
 Noticing that
\[
 a-cw=a-c\frac{az+b}{cz+d}=\frac{\det\gamma}{cz+d},
\]
 we find that the factor $(\det\gamma)^{1/2}/(cz+d)$ appearing in the
 slash operator can be written as
 \begin{equation} \label{equation: a-cw}
 \frac{(\det\gamma)^{1/2}}{cz+d}=\frac{a-cw}{(\det\gamma)^{1/2}},
 \end{equation}
 which yields the version of Bol's identity stated in the lemma.
\end{proof}

To apply Frobenius' method by using the expansion in Proposition~\ref{prop-2}, we need the following result.

\begin{Lemma} \label{lemma: local solutions}
 Let $Q(z)$ and $R(z)$ be meromorphic modular forms
 of weight $4$ and $6$, respectively, on $\Gamma$. Assume that
 $\wt Q(x)=\sum_{n\ge n_0}a_nx^n$ and $\wt R(x)=\sum_{n\ge n_0}b_nx^n$
 are the power series such that
 \begin{gather}
 Q(z)=(1-(1+A)\wt w)^4\wt Q(-4\pi(\Im z_0)\wt w),\nonumber\\
 R(z)=(1-(1+A)\wt w)^6\wt R(-4\pi(\Im z_0)\wt w).\label{eq: QR in wt w 0}
 \end{gather}
 Then
 \begin{equation}\label{eq-11}
 y(z)=\frac1{(1-(1+A)\wt w)^2}\sum_{n=0}^\infty c_n
 (-4\pi(\Im z_0)\wt w)^{n+\alpha}
\end{equation}
 is a solution of \eqref{equation: DE in y}
 if and only if the series $\wt y(x)=\sum_{n=0}^\infty c_nx^{n+\alpha}$
 satisfies
 \begin{equation} \label{equation: DE in wt y}
 \frac{{\rm d}^3}{{\rm d}x^3}\wt y(x)+\wt Q(x)\frac{{\rm d}}{{\rm d}x}\wt y(x)
 +\left(\frac12\frac{{\rm d}}{{\rm d}x}\wt Q(x)+\wt R(x)\right)\wt y(x)=0.
 \end{equation}
\end{Lemma}

\begin{proof} Let
 \[
 \gamma=\M{-4\pi\Im z_0}001\M10A1\M1{-z_0}1{-\overline z_0}
 =\M{-4\pi\Im z_0}{(4\pi\Im z_0)z_0}{1+A}{-Az_0-\overline z_0}
 \]
 with $\det\gamma=-4\pi(\Im z_0)(z_0-\overline z_0)$. Let $x=\gamma
 z=-4\pi(\Im z_0)\wt w$. Note that if we write $\gamma$ as
 $\gamma=\SM abcd$, then
\[
 \frac{(a-cx)^2}{\det\gamma}
 =2\pi {\rm i}(1-(1+A)\wt w)^2.
\]
 Thus, by \eqref{equation: a-cw}, $y(z)$ and $\wt y(x)$ are related by
 \[ y(z)=2\pi {\rm i}\big(\wt y\big|_{-2}\gamma\big)(z).
\]
 Hence, applying Lemma \ref{lemma: Bol} with $r=2$, we obtain
\[
 D_q^3y(z)=\frac1{(2\pi {\rm i})^2}\frac{{\rm d}^3}{{\rm d}z^3}
 \big(\wt y\big|_{-2}\gamma\big)(z)
 =(1-(1+A)\wt w)^4\frac{{\rm d}^3}{{\rm d}x^3}\wt y(x).
 \]
 Also, a direct computation yields, by \eqref{equation: a-cw},
\[
 \frac{{\rm d}x}{{\rm d}z}=\frac{{\rm d}\gamma z}{{\rm d}z}=\frac{\det\gamma}{(cz+d)^2}
 =\frac{(a-cx)^2}{\det\gamma}=2\pi {\rm i}(1-(1+A)\wt w)^2,
\]
 and
\[
 \frac{{\rm d}\wt w}{{\rm d}z}=-\frac1{4\pi\Im z_0}\frac{{\rm d}x}{{\rm d}z}
 =\frac{(1-(1+A)\wt w)^2}{2{\rm i}\Im z_0}.
 \]
 Hence,
 \begin{align*}
 D_qy(z)&=\frac1{2\pi {\rm i}}\frac{{\rm d}}{{\rm d}z}\left(\frac1{(1-(1+A)\wt w)^2}
 \wt y(x)\right) \\
 &=-\frac{1+A}{2\pi(\Im z_0)(1-(1+A)\wt w)}\wt y(x)
 +\frac{{\rm d}}{{\rm d}x}\wt y(x),
 \end{align*}
 and
 \begin{align*}
 D_qQ(z)&=\frac1{2\pi {\rm i}}\frac{{\rm d}}{{\rm d}z}\left((1-(1+A)\wt w)^4
 \wt Q(x)\right) \\
 &=(1-(1+A)\wt w)^6\left(\frac{1+A}{\pi(\Im z_0)(1-(1+A)\wt w)}\wt Q(x)
 +\frac{{\rm d}}{{\rm d}x}\wt Q(x)\right).
 \end{align*}
 Putting everything together, we find that
 \begin{gather*}
 D_q^3y(z)+Q(z)D_qy(z)+\left(\frac12D_q(z)+R(z)\right)y(z) \\
\qquad{} = (1-(1+A)\wt w)^4\left(
 \frac{{\rm d}^3}{{\rm d}x^3}\wt y(x)+\wt Q(x)\frac{{\rm d}}{{\rm d}x}\wt y(x)
 +\left(\frac12\frac{{\rm d}}{{\rm d}x}\wt Q(x)+\wt R(x)\right)\wt y(x)
 \right).
 \end{gather*}
 Thus, $y(z)$ is a solution of~\eqref{equation: DE in y} if and
 only if $\wt y(x)$ is a solution of~\eqref{equation: DE in wt y}.
\end{proof}

We will see from Lemma~\ref{lemma-4-7} below that Corollary~\ref{coro} and Lemma~\ref{lemma: local solutions} can be applied to prove that (\ref{eq-01}) or equivalently~(\ref{equation: DE in y}) is apparent at elliptic points of order $e\geq 3$. This is a great advantage of using expansions in terms of $\wt w$ rather than~$z-z_0$.

\begin{Remark} \label{remark: Q(x) SL(2,Z)}
 In practice, the power series $\wt Q(x)$ and $\wt R(x)$ can be
 computed using Proposition~A.4 of~\cite{LY}. For example, for the
 Eisenstein series $E_4(z)$ and $E_6(z)$ on $\SL(2,\Z)$, we find that
 \begin{gather}
 E_4(z) =(1-(1+A)\wt w)^4\left(B-\frac 13Cu+\frac5{72}B^2u^2
 -\frac5{432}BCu^3+\cdots\right), \nonumber\\
 E_6(z) =(1-(1+A)\wt w)^6\left(C-\frac12B^2u+\frac7{48}BCu^2+\cdots\right),\label{equation: E4 E6}
 \end{gather}
 where $u=-4\pi(\Im z_0)\wt w$, $B=E_4(z_0)$, and $C=E_6(z_0)$.
\end{Remark}

\subsection[Existence of Q(z) and R(z)]{Existence of $\boldsymbol{Q(z)}$ and $\boldsymbol{R(z)}$}

 In this section we shall discuss the criterion of the existence of meromorphic modular
 forms $Q(z)$ and $R(z)$ of weight $4$ and~$6$, respectively, on
 $\SL(2,\Z)$, such that the differential equation~(\ref{equation: DE in y}) is Fuchsian and apparent throughout $\H$ with prescribed
 local exponents at singularities and at cusps.

 Throughout the section, we let $z_j$, $j=1,\dots,m$, be
 $\SL(2,\Z)$-inequivalent point on $\H$, none of which is an elliptic
 point. We let ${\rm i}=\sqrt{-1}$ and $\rho=\big({-}1+\sqrt3{\rm i}\big)/2$
 be the unique elliptic point of order $2$ and $3$ of $\SL(2,\Z)$,
 respectively. For $z=z_j$, ${\rm i}$, or $\rho$, we assume that
 $\kappa_{z}^{(1)}<\kappa_{z}^{(2)}<\kappa_{z}^{(3)}$,
 are rational numbers in $\frac13\Z$ such that
 $\kappa_{z}^{(1)}+\kappa_{z}^{(2)}+\kappa_{z}^{(3)}=3$ and
 $\kappa_z^{(2)}-\kappa_z^{(1)}\in\Z$ (and hence
 $\kappa_z^{(3)}-\kappa_z^{(1)}\in\Z)$. When $z={\rm i}$, we further assume
 that
 \begin{equation} \label{equation: assumption rho1}
 \big\{3\kappa_{\rm i}^{(1)},3\kappa_{\rm i}^{(2)},3\kappa_{\rm i}^{(3)}\big\}\equiv
 \{0,0,1\}\mod 2.
 \end{equation}
 Also, when $z=\rho$, we note that the assumptions on
 $\kappa_\rho^{(j)}$ above imply that $\kappa_\rho^{(j)}\in\Z$ for
 all~$j$. We further assume that
 \begin{equation} \label{equation: difference assumption}
 \big\{\kappa_\rho^{(1)},\kappa_\rho^{(2)},\kappa_\rho^{(3)}\big\}\equiv\{0,1,2\} \quad \mod
 3.
 \end{equation}
 Finally, for the cusp $\infty$ of $\SL(2,\Z)$, we let
 $\kappa_\infty^{(1)}\le\kappa_\infty^{(2)}\le\kappa_\infty^{(3)}$ be
 three rational numbers in $\frac13\Z$ such that
 $\kappa_\infty^{(1)}+\kappa_\infty^{(2)}+\kappa_\infty^{(3)}=0$ and
 $\kappa_\infty^{(2)}-\kappa_\infty^{(1)},
 \kappa_\infty^{(3)}-\kappa_\infty^{(2)}\in\Z$. We shall consider the
 problem whether there exist meromorphic modular forms $Q(z)$ and
 $R(z)$ of weight $4$ and $6$, respectively, on $\SL(2,\Z)$, such
 that~(\ref{eq-10}) is Fuchsian and apparent throughout~$\H$
 and the local exponents $\kappa$'s are given as above.

 \begin{Lemma} \label{lemma: Q R general form}
 Let notations ${\rm i}$ and $\rho$ be
 as above. Then meromorphic modular forms~$Q(z)$ and~$R(z)$ of
 weight $4$ and $6$, respectively, on $\SL(2,\Z)$ that have poles of order
 at most $2$ and $3$, respectively, at points $\SL(2,\Z)$-equivalent
 to $z_j$, ${\rm i}$, or $\rho$ and are holomorphic at other points
 and cusps are of the form
 \begin{gather}
 Q(z) =r_\infty E_4(z)
 +r_{{\rm i}}^{(2)}\frac{E_4(z)\Delta_0(z)}{E_6(z)^2}
 +r_{\rho}^{(2)}\frac{\Delta_0(z)}{E_4(z)^2} \nonumber\\
\hphantom{Q(z) =}{}+\sum_{j=1}^m\left(r_{z_j}^{(2)}
 \frac{E_4(z)\Delta_0(z)^2}{F_j(z)^2}
 +r_{z_j}^{(1)}\frac{E_4(z)\Delta_0(z)}{F_j(z)}\right), \nonumber\\
 R(z) =s_\infty E_6(z)+s_{{\rm i}}^{(3)}
 \frac{\Delta_0(z)^2}{E_6(z)^3}
 +s_{{\rm i}}^{(1)}\frac{\Delta_0(z)}{E_6(z)} \nonumber\\
\hphantom{R(z) =}{}
+s_{\rho}^{(3)}\frac{E_6(z)\Delta_0(z)}{E_4(z)^3}
 +\sum_{j=1}^m\sum_{k=1}^3s_{z_j}^{(k)}
 \frac{E_6(z)\Delta_0(z)^k}{F_j(z)^k},\label{equation: Q R SL(2,Z)}
 \end{gather}
 where $\Delta_0(z)=1728\Delta(z)=(E_4(z)^3-E_6(z)^2)$ and
 $F_j(z)=E_4(z)^3-t_jE_6(z)^2$ with $t_j=E_4(z_j)^3/E_6(z_j)^2$.
 \end{Lemma}

 \begin{proof} By Corollary \ref{coro}, there are no
 meromorphic modular forms of weight $4$ having a pole at ${\rm i}$
 or $\rho$ with a nonzero residue. Also, the order of a meromorphic modular form of weight
 $6$ on $\SL(2,\Z)$ at ${\rm i}$ is necessarily odd, while that at $\rho$
 is congruent to $0$ modulo~$3$. Thus, we can take $r_{\rm i}^{(2)}$, $r_{\rho}^{(2)}$, $r_{z_j}^{(2)}$, $r_{z_j}^{(1)}$ such that
\begin{gather*}
Q(z)-r_{{\rm i}}^{(2)}\frac{E_4(z)\Delta_0(z)}{E_6(z)^2}
 -r_{\rho}^{(2)}\frac{\Delta_0(z)}{E_4(z)^2}
 -\sum_{j=1}^m\left(r_{z_j}^{(2)}
 \frac{E_4(z)\Delta_0(z)^2}{F_j(z)^2}
 +r_{z_j}^{(1)}\frac{E_4(z)\Delta_0(z)}{F_j(z)}\right)
\end{gather*}
is a holomorphic modular form of weight $4$ on $\SL(2,\mathbb{Z})$, so it must be a multiple of $E_4(z)$. The proof for $R(z)$ is similar.
\end{proof}

We first determine the indicial equations of \eqref{equation: DE in y}
at $\infty$, ${\rm i}$, $\rho$, and $z_j$, $j=1,\dots,m$.

\begin{Lemma} \label{lemma: indicial at infinity}
 Suppose that $Q(z)$ and $R(z)$ are meromorphic modular
 forms given by \eqref{equation: Q R SL(2,Z)}. Then the indicial
 equation of \eqref{equation: DE in y} at the cusp $\infty$ is
 \[
 x^3+r_\infty x+s_\infty=0.
\]
\end{Lemma}

\begin{proof} It is clear that $Q(z)=r_\infty+O(q)$ and
 $R(z)=s_\infty+O(q)$. Assume that there is a~solution of~\eqref{equation:
 DE in y} of the form $y(z)=q^\alpha(1+O(q))$, $\alpha\in\R$. We
 compute that
 \begin{gather*}
 D_q^3y(z)+Q(z)D_qy(z)+\left(\frac12D_qQ(z)+R(z)\right)y(z)\\
 \qquad{}=\alpha^3q^\alpha+r_\infty\alpha q^\alpha
 +s_\infty q^\alpha+O\big(q^{\alpha+1}\big),
 \end{gather*}
 from which we see that the indicial equation at $\infty$ is
 $x^3+r_\infty x+s_\infty=0$.
\end{proof}

 We now consider the indicial equation of \eqref{equation: DE in y}
 at a point in~$\H$. Let $z_0$ be one of $z_j$, ${\rm i}$, or~$\rho$ and~$\wt w$ be defined by~\eqref{eq: wt w 0}
 with $\phi^\ast(z)=E_2^\ast(z):=E_2(z)+6/(\pi {\rm i}(z-\overline z))$.
 Recall that in Lemma~\ref{lemma: local solutions} we have proved
 that if $\wt Q(x)$ and $\wt R(x)$ are the Laurent series in~$x$
 such that~\eqref{eq: QR in wt w 0} holds.
 Then
\[
 y(z;\alpha)=\frac1{(1-(1+A)\wt w)^2}\sum_{n=0}^\infty c_n(\alpha)
 (-4\pi(\Im z_0)\wt w)^{n+\alpha}, \qquad c_0(\alpha)=1,
\]
 is a solution of (\ref{equation: DE in y}) if and only if the series
 $\wt y(x;\alpha)=\sum_{n=0}^\infty c_n(\alpha)x^{n+\alpha}$ satisfies
 \begin{equation} \label{equation: existence temp}
 \frac{{\rm d}^3}{{\rm d}x^3}\wt y(x;\alpha)+\wt Q(x)\frac{{\rm d}}{{\rm d}x}\wt y(x;\alpha)
 +\left(\frac12\frac{{\rm d}}{{\rm d}x}\wt Q(x)+\wt R(x)\right)\wt y(x;\alpha)=0.
 \end{equation}
 Let
 \begin{equation} \label{equation: Q(x) R(x)}
 \wt Q(x)=\sum_{n=-2}^\infty a_nx^n, \qquad
 \frac12\frac{{\rm d}}{{\rm d}x}\wt Q(x)+\wt R(x)=\sum_{n=-3}^\infty b_nx^n,
 \end{equation}
 where each $a_n$ and $b_n$ is linear in the parameters $r$'s and
 $s$'s. Then following the computation in Appendix \ref{section-3},
 we see that \eqref{equation: existence temp} is equivalent to
 \begin{equation} \label{equation: Rn}
 R_n(\alpha):=f(\alpha+n)c_n(\alpha)+\sum_{k=0}^{n-1}
 [(\alpha+k)a_{n-k-2}+b_{n-k-3}]c_k(\alpha)=0
 \end{equation}
 for all $n\ge 0$, where $f(t):=t(t-1)(t-2)+a_{-2}t+b_{-3}$.
 In particular, $R_0(\alpha)\colon f(\alpha)=0$ is the indicial equation at
 $z_0$, i.e.,
\[
 f(t)=\big(t-\kappa_{z_0}^{(1)}\big)\big(t-\kappa_{z_0}^{(2)}\big)\big(t-\kappa_{z_0}^{(3)}\big).
\]
 Using \eqref{equation: E4 E6}, we can work out the coefficients $a_{-2}$ and $b_{-3}$.

 \begin{Lemma} \label{lemma: indicial at z}
 Suppose that $Q(z)$ and $R(z)$ are meromorphic modular
 forms given by~\eqref{equation: Q R SL(2,Z)}. Then the indicial
 equations of \eqref{equation: DE in y} at ${\rm i}$,
 $\rho$, and~$z_j$, $j=1,\dots,m$, are
 \begin{gather*}
 x(x-1)(x-2)+4r_{{\rm i}}^{(2)}x-4r_{{\rm i}}^{(2)}-8s_{{\rm i}}^{(3)}=0,
\\
 x(x-1)(x-2)-9r_\rho^{(2)}x+9r_\rho^{(2)}+27s_\rho^{(3)}=0,
\\
 x(x-1)(x-2)+\frac{r_{z_j}^{(2)}}{t_j}x-\frac{r_{z_j}^{(2)}}{t_j}
 +\frac{s_{z_j}^{(3)}}{t_j^2}=0,
\end{gather*}
 respectively.
 \end{Lemma}

 \begin{proof}
 For the elliptic point ${\rm i}$, we know that $E_6({\rm i})=0$. Hence, using
 \eqref{equation: E4 E6}, we find that
 \begin{gather*}
 \wt Q(x)=\frac{4r_{{\rm i}}^{(2)}}{x^2}+\dots, \qquad
 \wt R(x)=-\frac{8s_{{\rm i}}^{(2)}}{x^3}+\cdots.
 \end{gather*}
 Therefore, we have $a_{-2}=4r_{{\rm i}}^{(2)}$ and $b_{-3}=-4r_{{\rm i}}^{(2)}-8s_{{\rm i}}^{(2)}$.
 Similarly, for the elliptic point $\rho$, using $E_4(\rho)=0$ and~\eqref{equation: E4 E6} again, we find that
 $a_{-2}=-9r_\rho^{(2)}$ and $b_{-3}=9r_\rho^{(2)}+27s_\rho^{(3)}$.
 For the point~$z_j$, letting $B=E_4(z_j)$ and $C=E_6(z_j)$, we compute that
 \begin{gather*}
 \frac{\Delta_0(z)}{(1-(1+A)\wt w)^{12}} =\big(B^3-C^2\big)+O(\wt w)
 =C^2(t_j-1)+O(\wt w), \\
 \frac{F_j(z)}{(1-(1+A)\wt w)^{12}} =\left(
 B-\frac13Cu+O\big(\wt w^2\big)\right)^3
 -t_j\left(C-\frac12B^2u+O\big(\wt w^2\big)\right)^2 \\
 \hphantom{\frac{F_j(z)}{(1-(1+A)\wt w)^{12}}}{}
 =(t_j-1)B^2Cu+O\big(\wt w^2\big),
 \end{gather*}
 where $u=-4\pi(\Im z_j)\wt w$. It follows that
 \[
 \wt
 Q(x)=r_{z_j}^{(2)}\frac{(t_j-1)^2BC^4}{(t_j-1)^2B^4C^2x^2}+\cdots
 =\frac{r_{z_j}^{(2)}}{t_jx^2}+\dots,
 \] and
 \[
 \wt R(x)=s_{z_j}^{(3)}\frac{(t_j-1)^3C^7}{(t_j-1)^3B^6C^3}+\cdots
 =\frac{s_{z_j}^{(3)}}{t_j^2x^3}+\cdots.
 \]
 Therefore, $a_{-2}=r_{z_j}^{(2)}/t_j$ and
 $b_{-3}=-r_{z_j}^{(2)}/t_j+s_{z_j}^{(3)}/t_j^2$. This proves the lemma.
 \end{proof}

 The two lemmas show that the parameters $r_\infty$, $s_\infty$,
 $r_{{\rm i}}^{(2)}$, $s_{{\rm i}}^{(3)}$, $r_\rho^{(2)}$, $s_\rho^{(3)}$,
 $r_{z_j}^{(2)}$, and $s_{z_j}^{(3)}$, $j=1,\dots,m$, solely depend
 on the local exponents $\kappa$'s. The remaining parameters are
 \begin{equation*} %\label{equation: free parameters}
 \begin{cases}
 r_{z_j}^{(1)},s_{z_j}^{(2)},s_{z_j}^{(1)}, \quad \text{for} \ j=1,\dots,m,\\
 s_{{\rm i}}^{(1)}.
 \end{cases}
 \end{equation*}
 That is, the number of remaining parameters is $3m+1$. We now show that the
 apparentness condition will impose $3m+1$ polynomial constraints on
 the remaining parameters. For convenience, we let $\mathbf r$ and
 $\mathbf s$ denote $r_{z_1}^{(1)},\dots,r_{z_m}^{(1)}$ and
 $s_{{\rm i}}^{(1)},s_{z_1}^{(2)},s_{z_1}^{(1)},\dots,s_{z_m}^{(2)},s_{z_m}^{(1)}$,
 respectively.

 Let $z_0\in\{{\rm i},\rho,z_1,\dots,z_m\}$. Assume that $\alpha$ is one
 of the local exponents $\kappa_{z_0}^{(k)}$. We observe that one can
 recursively determine $c_n(\alpha)$ using~\eqref{equation: Rn} as
 long as $f(\alpha+n)\neq 0$. Thus, there always
 exists a solution $\wt y\big(x;\kappa_{z_0}^{(3)}\big)$ with local exponent
 $\kappa_{z_0}^{(3)}$; see, e.g., Lemma~\ref{lemma-3-2}. For the
 exponent $\alpha=\kappa_{z_0}^{(2)}$, because
 \[
 f\big(\alpha+\kappa_{z_0}^{(3)}-\kappa_{z_0}^{(2)}\big)=f\big(\kappa_{z_0}^{(3)}\big)=0,
\]
 we can only solve for $c_n(\alpha)$ up to
 $n=\kappa_{z_0}^{(3)}-\kappa_{z_0}^{(2)}-1$. At
 $n'=\kappa_{z_0}^{(3)}-\kappa_{z_0}^{(2)}$, the relation
 $R_{n'}(\alpha)$ becomes
 \begin{equation} \label{equation: existence condition 1}
 \sum_{k=0}^{n'-1}[(\alpha+k)a_{n'-k-2}
 +b_{n'-k-3}]c_k(\alpha)=0.
 \end{equation}
 If $c_k(\alpha)$, $k=0,\dots,n'-1$, do not satisfy this relation,
 then there is no solution with exponent~$\kappa_{z_0}^{(2)}$.
 If $c_k(\alpha)$, $k=0,\dots,n'-1$, satisfy this relation, we then
 can choose $c_{n'}(\alpha)$ to be any number and solve
 for $c_n(\alpha)$ recursively and obtain a solution $\wt
 y\big(x;\kappa_{z_0}^{(2)}\big)$ with local exponent $\kappa_{z_0}^{(2)}$.
 Likewise, for the local exponent $\alpha=\kappa_{z_0}^{(1)}$, there will be
 two conditions~\eqref{equation: existence condition 1} corresponding
 to $n'=\kappa_{z_0}^{(2)}-\kappa_{z_0}^{(1)}$ and
 $n'=\kappa_{z_0}^{(3)}-\kappa_{z_0}^{(1)}$ that $c_k(\alpha)$ must
 satisfy. Thus, there are three polynomial equations
 \begin{equation*} %\label{equation: three conditions}
 P_{z_0,k_1,k_2}(\mathbf r,\mathbf s)=0, \qquad
 (k_1,k_2)\in\{(1,2),(1,3),(2,3)\}
 \end{equation*}
 that $\mathbf r$ and $\mathbf s$ need to satisfy. However, we will
 see in a moment that when $z_0$ is an elliptic point, some or all of
 the three polynomial equations hold trivially.

 \begin{Lemma}\label{lemma-4-7}
 Assume that $z_0$ is an elliptic point of order $e$ of
 $\SL(2,\Z)$. Let $(k_1,k_2)\in\{(1,2),\allowbreak (1,3),(2,3)\}$.
 If $\kappa_{z_0}^{(k_2)}-\kappa_{z_0}^{(k_1)}\not\equiv 0\mod e$,
 then the polynomial $P_{z_0,k_1,k_2}(\mathbf r,\mathbf s)$ is
 identically zero.
 In particular, under our assumptions \eqref{equation: assumption
 rho1} and \eqref{equation: difference assumption}, the
 differential equation~\eqref{equation: DE in y} is apparent
 at $\rho$ for any $\mathbf r$ and $\mathbf s$, and also there is
 at most one pair $(k_1,k_2)$ such that $P_{{\rm i},k_1,k_2}(\mathbf
 r,\mathbf s)\neq 0$.
 \end{Lemma}

 \begin{proof} By Corollary \ref{coro}, the
 coefficients $a_n$ in the expansion of $\wt Q(x)$ vanish whenever
 $n\not\equiv-2\allowbreak \mod e$. Likewise, the coefficients $b_n$ in $\wt
 R(x)$ vanish whenever $n\not\equiv-3\mod e$. Using these facts, we
 find that the condition \eqref{equation: Rn} reduces to
 \begin{equation} \label{equation: Rn 2}
 f(\alpha+n)c_n(\alpha)+\sum_{k\equiv n\mod e,\, k\le n-1}
 [(\alpha+k)a_{n-k-2}+b_{n-k-3}]c_k(\alpha)=0.
 \end{equation}
 Then we can easily prove by induction up to $n=n'-1$ that
 $c_n(\alpha)=0$ whenever $n\not\equiv 0\mod e$. Now if $n'\not\equiv
 0\mod e$, then \eqref{equation: existence condition 1} automatically
 holds because every summand is $0$ due to the facts that
 $a_{n'-k-2}$ and $b_{n'-k-3}$ are nonzero only when $k\equiv n'\mod
 e$ and $c_k(\alpha)$ is nonzero only when $k\equiv 0\mod e$, but $k$
 cannot be congruent to $n'$ and $0$ at the same time. This proves
 the lemma.
 \end{proof}

 In the following, we let $P_{{\rm i}}(\mathbf r,\mathbf s)$ denote the only
 nonzero polynomial $P_{{\rm i},k_1,k_2}(\mathbf r,\mathbf s)$ in
 the lemma. The discussion above shows that there are $3m+1$
 polynomial equations $P_{{\rm i}}(\mathbf r,\mathbf s)=0$,
 $P_{z_j,k_1,k_2}(\mathbf r,\mathbf s)=0$, $j=1,\dots,m$,
 $(k_1,k_2)\in\{(1,2),(1,3),(2,3)\}$ in $3m+1$ variables such that~\eqref{equation: DE in y} is Fuchsian and apparent throughout~$\H$
 and all $\SL(2,\Z)$-inequivalent singularities belong to
 $\{{\rm i},\rho,z_1,\dots,z_m\}$ with the given local exponents if and only if
 the parameters~$\mathbf r$ and~$\mathbf s$ are common roots of the
 polynomials. We now consider the degree of these polynomials.

 \begin{Proposition} \label{proposition: existence}
 We have
\[
 \deg P_{z_j,k_1,k_2}(\mathbf r,\mathbf s)=\begin{cases}
 \kappa_{z_j}^{(k_2)}-\kappa_{z_j}^{(k_1)},
 &\text{if }(k_1,k_2)=(1,2)\text{ or }(2,3), \\
 \kappa_{z_j}^{(3)}-\kappa_{z_j}^{(1)}-1,
 &\text{if }(k_1,k_2)=(1,3), \end{cases}
\]
 and
\[
 \deg P_{{\rm i}}(\mathbf r,\mathbf s)=\big(\kappa_{\rm i}^{(k_2)}-\kappa_{\rm i}^{(k_1)}\big)/2,
\]
 where $(k_1,k_2)$ is the unique pair such that
 $\kappa_{\rm i}^{(k_2)}-\kappa_{\rm i}^{(k_1)}\equiv 0\mod 2$. To be more precise,
 for a~polynomial $P(\mathbf r,\mathbf s)$ in $\mathbf r$
 and $\mathbf s$, we let $\LT(P)$ denote the sum of the terms of
 highest degree in~$P$. Then, up to nonzero scalars, we have
 \[
 \LT(P_{z_j,k_1,k_2}(\mathbf r,\mathbf s))
 =\prod_{k=1}^{\kappa_{z_j}^{(k_2)}-\kappa_{z_j}^{(k_1)}}
 \big(\big(\kappa_{z_j}^{(k_1)}+k-3/2\big)d_1r_{z_j}^{(1)}+d_2s_{z_j}^{(2)}\big)
 \]
 for $(k_1,k_2)=(1,2)$ or $(2,3)$, and
 \begin{gather*}
 \LT(P_{z_j,1,3}(\mathbf r,\mathbf s))
 =\big(d_3s_{z_j}^{(1)}+d_1'r_{z_1}^{(1)}+\cdots+d_m'r_{z_m}^{(1)}\big) \\
\hphantom{\LT(P_{z_j,1,3}(\mathbf r,\mathbf s))=}{}
\times\prod_{k=1,\,k\neq\kappa_{z_j}^{(2)}-\kappa_{z_j}^{(1)},
 \kappa_{z_j}^{(2)}-\kappa_{z_j}^{(1)}+1}
 \big(\big(\kappa_{z_j}^{(k_1)}+k-3/2\big)d_1r_{z_j}^{(1)}+d_2s_{z_j}^{(2)}\big),
 \end{gather*}
 where $d_1,d_2,d_3,d_1',\dots,d_m'$ are complex numbers with
 $d_1,d_2,d_3\neq 0$.

Also, $\LT(P_{\rm i}(\mathbf r,\mathbf s))$ is a
 product of $\big(\kappa_{\rm i}^{(k_2)}-\kappa_{\rm i}^{(k_1)}\big)/2$ linear sums in
 $s_{{\rm i}}^{(1)}$ and $\mathbf r$ with the coefficients of $s_{{\rm i}}^{(1)}$
 being nonzero.
 \end{Proposition}

 \begin{proof}
 By Lemma \ref{lemma: Q R general form}, each of $a_n$ and $b_n$ in
 \eqref{equation: Q(x) R(x)} is a
 linear combination of $r$'s and $s$'s. The key
 observation here is that only $r_{z_j}^{(2)}$ (which has been
 determined by the local exponents and is regarded as a constant)
 appears in $a_{-2}$, only $r_{z_j}^{(2)}$ and $r_{z_j}^{(1)}$ appear
 in $a_{-1}$, only $r_{z_j}^{(2)}$ and $s_{z_j}^{(3)}$ (which has
 been determined by the local exponents and is regarded as a
 constant) appear in~$b_{-3}$, only
 $r_{z_j}^{(2)}$, $r_{z_j}^{(1)}$, $s_{z_j}^{(3)}$, and
 $s_{z_j}^{(2)}$ appear in $b_{-2}$, and only $r_{z_j}^{(2)}$,
 $r_{z_j}^{(1)}$, $s_{z_j}^{(3)}$, $s_{z_j}^{(2)}$, and
 $s_{z_j}^{(1)}$ appear in~$b_{-1}$. In particular, we have
 \begin{equation} \label{equation: LT a-1, b-2}
 \LT(a_{-1})=d_1r_{z_j}^{(1)}, \qquad
 \LT(b_{-2})=-\frac12d_1r_{z_j}^{(1)}+d_2s_{z_j}^{(2)},
 \end{equation}
 where $d_1$ and $d_2$ are the leading coefficients in the series
 expansions of $E_4(z)\Delta(z)/F_j(z)$ and
 $E_6(z)\Delta(z)^2/F_j(z)^2$ in $\wt w$ (and hence are nonzero).
 Consider \eqref{equation: existence condition 1} for the cases
 $(\alpha,n')=\big(\kappa_{z_j}^{(2)},\kappa_{z_j}^{(3)}-\kappa_{z_j}^{(2)}\big)$
 and
 $(\alpha,n')=\big(\kappa_{z_j}^{(1)},\kappa_{z_j}^{(2)}-\kappa_{z_j}^{(1)}\big)$
 first. From~\eqref{equation: Rn} and~\eqref{equation: LT a-1, b-2}, we
 can easily show inductively that, up to $n=n'-1$,
 \begin{align}
 \LT(c_n(\alpha))&=\prod_{k=1}^n
 \left(-\frac{\LT((\alpha+k-1)a_{-1}+b_{-2})}{f(\alpha+k)}\right) \nonumber\\
 &=\prod_{k=1}^n
 \left(-\frac{(\alpha+k-3/2)d_1r_{z_j}^{(1)}
 +d_2s_{z_j}^{(2)}}{f(\alpha+k)}\right),\label{equation: leading term}
 \end{align}
 where $d_1$ and $d_2$ are the two nonzero complex numbers in
 \eqref{equation: LT a-1, b-2}, and hence
\[
 \LT(P_{z_j,k_1,k_2}(\mathbf r,\mathbf s))=
 \prod_{k=1}^{\kappa_{z_0}^{(k_2)}-\kappa_{z_0}^{(k_1)}}
 \left(-\frac{(\alpha+k-3/2)d_1r_{z_j}^{(1)}
 +d_2s_{z_j}^{(2)}}{f(\alpha+k)}\right)
\]
 for $(k_1,k_2)=(1,2)$ or $(2,3)$.

 We now consider \eqref{equation: existence condition 1} for the
 remaining case $(\alpha,n')=\big(\kappa_{z_j}^{(1)},
 \kappa_{z_j}^{(3)}-\kappa_{z_j}^{(1)}\big)$. Let
 $n''=\kappa_{z_j}^{(2)}-\kappa_{z_j}^{(1)}$. Up to $n=n''-1$, the
 terms of highest degree in $c_n(\alpha)$ is given by
 \eqref{equation: leading term}. Since $P_{z_j,1,2}(\mathbf r,
 \mathbf s)=0$ is assumed to hold, \eqref{equation: Rn} holds for
 $n=n''$ for arbitrary~$c_{n''}(\alpha)$. Here, we simply choose $c_{n''}(\alpha)$ to be
 $0$. Then we have, by \eqref{equation: Rn}
\[
 c_{n''+1}(\alpha)=-\frac1{f(\alpha+n''+1)}
 \sum_{k=0}^{n''-1}[(\alpha+k)a_{n''-k-1}+b_{n''-k-2}]c_k(\alpha),
\]
 and hence
\[
 \LT(c_{n''+1}(\alpha))=-\frac1{f(\alpha+n''+1)}
 \LT((\alpha+n''-1)a_0+b_{-1})\LT(c_{n''-1}(\alpha)),
\]
 which, by \eqref{equation: leading term}, is equal to
\[
 -\frac{\LT((\alpha+n''-1)a_0+b_{-1})}{f(\alpha+n''+1)}
 \prod_{k=1}^{n''-1}\left(-\frac{(\alpha+k-3/2)d_1r_{z_j}^{(1)}
 +d_2s_{z_j}^{(2)}}{f(\alpha+k)}\right).
\]
 Then using \eqref{equation: Rn} we can inductively show that for $n$
 with $n''+1\le n\le n'-1$,
 \begin{align*}
 \LT(c_n(\alpha))
 &=\LT(c_{n''+1}(\alpha))\prod_{k=n''+2}^n
 \left(-\frac{\LT((\alpha+k-1)a_{-1}+b_{-2})}{f(\alpha+k)}\right)\\
 &=\LT(c_{n''+1}(\alpha))\prod_{k=n''+2}^n
 \left(-\frac{(\alpha+k-3/2)d_1r_{z_j}^{(1)}+d_2s_{z_j}^{(2)}}
 {f(\alpha+i)}\right).
 \end{align*}
 Thus, for $(k_1,k_2)=(1,3)$, we have
 \begin{gather*}
 \LT(P_{z_j,1,3}(\mathbf r,\mathbf s))
 =\LT(c_{n''+1}(\alpha))\prod_{k=n''+2}^{n'}
 \left(-\frac{(\alpha+k-3/2)d_1r_{z_j}^{(1)}+d_2s_{z_j}^{(2)}}
 {f(\alpha+k)}\right) \\
\hphantom{\LT(P_{z_j,1,3}(\mathbf r,\mathbf s))}{}
=-\frac{\LT((\alpha+n''-1)a_0+b_{-1})}{f(\alpha+n''+1)} \\
\hphantom{\LT(P_{z_j,1,3}(\mathbf r,\mathbf s))=}{}\times
 \prod_{k=1,\,k\neq n'',n''+1}^{n'}
 \left(-\frac{(\alpha+k-3/2)d_1r_{z_j}^{(1)}
 +d_2s_{z_j}^{(2)}}{f(\alpha+k)}\right).
 \end{gather*}
 Note that $\LT((\alpha+n''-1)a_0+b_{-1})$ is of the form
\[
 d_3s_{z_j}^{(1)}+d_1'r_{z_1}^{(1)}+\cdots+d_m'r_{z_m}^{(1)},
\]
 where $d_3$, $d_1',\dots,d_m'$ are complex numbers with $d_3\neq
 0$.

 We next consider the case where $z_0={\rm i}$ is the
 elliptic point of order $2$. There exists a unique pair of $k_1$ and
 $k_2$ with $k_1<k_2$ such that $\kappa_{{\rm i}}^{(k_2)}-\kappa_{{\rm i}}^{(k_1)}\equiv0\mod
 2$. We let $\alpha=\kappa_{{\rm i}}^{(k_1)}$ and
 $n'=\kappa_{{\rm i}}^{(k_2)}-\kappa_{z_0}^{(k_1)}$. We have seen earlier
 that $c_n(\alpha)\neq 0$ only when $2|n$. Also, from
 \eqref{equation: Rn 2}, we can inductively show that
\[
 \LT(c_n(\alpha))=\prod_{k=1}^{n/2}\left(
 -\frac{\LT((\alpha+2k-2)a_0+b_{-1})}{f(\alpha+2k)}\right)
\]
 and
\[
 \LT(P_{{\rm i}}(\mathbf r,\mathbf s))=\prod_{k=1}^{n'/2}\left(
 -\frac{\LT((\alpha+2k-2)a_0+b_{-1})}{f(\alpha+2k)}\right).
 \]
 Noting that $\LT((\alpha+2k-2)a_0+b_{-1}))
 =ds_{{\rm i}}^{(1)}+(\text{a linear sum in }\mathbf r)$ for some nonzero
 complex number $d$ (which is the leading coefficient of the
 expansion of $\Delta(z)/E_6(z)$ at ${\rm i}$ in $\wt w$), we conclude that
 $\LT(P_{{\rm i}}(\mathbf r,\mathbf s)$ is a product of
 $\big(\kappa_{{\rm i}}^{(k_2)}-\kappa_{{\rm i}}^{(k_1)}\big)/2$ linear sums in $s_{{\rm i}}^{(1)}$
 and $\mathbf r$ with all coefficients of~$s_{{\rm i}}^{(1)}$ nonzero. This
 completes the proof.
\end{proof}

\begin{Remark} The proposition suggests that when the local exponents
 $\kappa$'s are fixed, for generic points $z_j$, the number of pairs
 $(Q,R)$ of modular forms such that \eqref{equation: DE in y}
 satisfies the conditions (S1)--(S3) is
\[
 \frac{\kappa_{\rm i}^{(k_2)}-\kappa_{\rm i}^{(k_1)}}2\prod_{j=1}^m
 \big(\kappa_{z_j}^{(2)}-\kappa_{z_j}^{(1)}\big)
 \big(\kappa_{z_j}^{(3)}-\kappa_{z_j}^{(2)}\big)
 \big(\kappa_{z_j}^{(3)}-\kappa_{z_j}^{(1)}-1\big),
\]
 where $k_1$ and $k_2$ are the integers such that
 $\kappa_{\rm i}^{(k_2)}-\kappa_{\rm i}^{(k_1)}\in2\mathbb N$. (Notice that if
 $\big(\kappa_z^{(1)},\kappa_z^{(2)},\kappa_z^{(3)}\big)=(0,1,2)$ for all
 $z\in\H$, this number is~$1$, as expected.)
 However, because the polynomials have intersection of
 positive dimension at infinity in general, we are not able to use
 the B\'ezout theorem to obtain this conclusion. (Even the existence
 of $(Q,R)$ with an arbitrary set of given data is not established yet.) We
 leave this problem for future study.
\end{Remark}

\section{Extremal quasimodular forms}
\label{section-extremal}

We note that by using some results in Section \ref{section-4}, Theorem \ref{thm-1} can be improved in some special case. More precisely,
the main result of this section is Theorem \ref{thm-2} below, which
states that $Q_2(z)$, $Q_3(z)$ in Theorem~\ref{thm-1} can be explicitly written
down in the case $f(z)$ is an extremal quasimodular form on
$\Gamma=\SL(2,\mathbb{Z})$.

\begin{Definition}[\cite{KK-2006}]
A quasimodular form $f\in \widetilde{\sM}_{k}^{\leq r}(\SL(2,\mathbb{Z}))$ is said to be extremal if its vanishing order at $\infty$ is equal to $\dim \widetilde{\sM}_{k}^{\leq r}(\SL(2,\mathbb{Z}))-1$. We say $f$ is normalized if its leading Fourier coefficient is~$1$.
\end{Definition}

Pellarin \cite{PN} proved that if $r\leq 4$, then a normalized
extremal quasimodular form in \linebreak $\widetilde{\sM}_{k}^{\leq
 r}(\SL(2,\mathbb{Z}))$ exists and is unique.

\begin{Theorem}\label{thm-2} Let $f(z)$ be an extremal quasimodular form of weight
 $k$ and depth $2$ on $\SL(2,\Z)$ and
 \begin{equation} \label{equation: DE in theorem}
 D_q^3y(z)+Q(z)D_qy(z)+\left(\frac12D_qQ(z)+R(z)\right)y(z)=0
 \end{equation}
 be the differential equation satisfied by $f(z)/\sqrt[3]{W_f(z)}$ as derived in Section~{\rm \ref{section-2}}.
 \begin{enumerate}
 \item[$(i)$] If $k\equiv 0\mod 4$, then
\[
 Q(z)=-\frac{k^2}{48}E_4(z),
 \qquad
 R(z)=-\frac{k^3}{864}E_6(z).
 \]
 \item[$(ii)$] If $k\equiv 2\mod 4$, then
\[
 Q(z)=-\frac{(k-2)^2}{48}E_4(z)
 -\frac13\frac{E_4(z)\big(E_4(z)^3-E_6(z)^2\big)}{E_6(z)^2},
 \]
 and
 \begin{gather*}
 R(z)=-\frac{(k-2)^3}{864}E_6(z)
 +\frac5{54}\frac{\big(E_4(z)^3-E_6(z)^2\big)^2}{E_6(z)^3} \\
\hphantom{R(z)=}{} +\frac{12-(k-2)^2}{144}
 \frac{E_4(z)^3-E_6(z)^2}{E_6(z)}.
 \end{gather*}
 \end{enumerate}
\end{Theorem}

\begin{Remark}We note that the differential equation in the case $k\equiv 0\mod 4$ is
 equivalent to the following
 equation studied by Kaneko and Koike \cite[Theorem~3.1]{KK-2006}:
\begin{equation*}%\label{eq-3}
D_q^3f-\frac{k}{4}E_2D_q^2f+\frac{k(k-1)}{4}D_qE_2D_qf-\frac{k(k-1)(k-2)}{24}D_q^2E_2f=0.
\end{equation*}
Indeed, by letting $f(z)=\Delta(z)^{\frac{k}{12}}y(z)$, a direct computation shows that $y(z)$ solves
\begin{equation*}%\label{eq-05}
D_q^3y-\frac{k^2}{48}E_4(z)D_qy-\frac{2k^2}{12^3}\left(3E_2(z)E_4(z)+(k-3)E_6(z)\right)y=0.
\end{equation*}
\end{Remark}

To prove Theorem~\ref{thm-2}, we need the following general lemma,
in which the quasimodular form~$f(z)$ is not assumed to be extremal.

\begin{Lemma} \label{lemma: local exponents at infinity}
 Assume that
 \begin{gather*}
 f(z)=f_0(z)+f_1(z)E_2(z)+f_2(z)E_2(z)^2\in \wt\sM^{\le2}_k (\SL(2,\Z)),\\
 f_j(z)\in\sM_{k-2j}(\SL(2,\Z)).
 \end{gather*}
 Let
 \begin{gather*}
 g(z)=f_1(z)+2f_2(z)E_2(z), \qquad h(z)=f_2(z), \qquad m=\min(\ord_\infty f, \ord_\infty g,\ord_\infty h).
 \end{gather*} Let
 $\kappa_\infty^{(1)}\le\kappa_\infty^{(2)}\le\kappa_\infty^{(3)}$ be
 the local exponents of \eqref{eq-1} at $\infty$.
 \begin{enumerate}
 \item[$(i)$] If $\ord_\infty f=m$, then $\ord_\infty W_f=3m$ and
 $\kappa_\infty^{(j)}=0$ for all $j$.
 \item[$(ii)$] If $\ord_\infty g=m$ and $\ord_\infty f\neq m$, then
 $\ord_\infty W_f=\ord_\infty f+2\ord_\infty g$ and
 $\kappa_\infty^{(1)}=\kappa_\infty^{(2)}=(\ord_\infty
 g-\ord_\infty f)/3$ and $\kappa_\infty^{(3)}=2(\ord_\infty
 f-\ord_\infty g)/3$.
 \item[$(iii)$] If $\ord_\infty h<\ord_\infty f\le\ord_\infty g$, then
 $\ord_\infty W_f=2\ord_\infty f+\ord_\infty h$ and
 $\kappa_\infty^{(1)}=2(\ord_\infty h-\ord_\infty f)/3$ and
 $\kappa_\infty^{(2)}=\kappa_\infty^{(3)}=(\ord_\infty
 f-\ord_\infty h)/3$.
 \item[(iv)] If $\ord_\infty h<\ord_\infty g<\ord_\infty f$, then
 $\ord_\infty W_f=\ord_\infty f+\ord_\infty g+\ord_\infty h$ and
 $\kappa_\infty^{(1)}=\ord_\infty h-\frac13\ord_\infty W_f$,
 $\kappa_\infty^{(2)}=\ord_\infty g-\frac13\ord_\infty W_f$, and
 $\kappa_\infty^{(3)}=\ord_\infty f-\frac13\ord_\infty W_f$.
 \end{enumerate}
\end{Lemma}

\begin{proof}
 Let $r=\frac13\ord_\infty W_f$. Since up to scalars,
 $g_3(z)=f(z)/\sqrt[3]{W_f(z)}$ is the unique solution of~\eqref{eq-1} without logarithmic singularity near $\infty$, according to
 Frobenius' method for complex ordinary differential equations (see, e.g., Appendix~\ref{section-3}), we
 must have $\kappa_\infty^{(3)}=\ord_\infty f-r$. Likewise, because
 $g_2(z)=(2zf(z)+\alpha g(z))/\sqrt[3]{W_f(z)}$ and
 $g_1(z)=\big(z^2f(z)+\alpha zg(z)+\alpha^2h(z)\big)/\sqrt[3]{W_f(z)}$ are the other
 two linearly independent solutions of \eqref{eq-1}, we have
 \[\kappa_\infty^{(2)}=\min(\ord_\infty f,\ord_\infty g)-r,\qquad\kappa_\infty^{(1)}=\min(\ord_\infty f,\ord_\infty g,\ord_\infty
 h)-r.\] Analyzing case by case, we obtain the claimed conclusions.
\end{proof}

\begin{proof}[Proof of Theorem~\ref{thm-2}] First of all, recall that
 \begin{equation} \label{equation: dim}
 \dim\wt\sM_k^{\le 2}(\SL(2,\Z))=1+\gauss{\frac k4}.
 \end{equation}
 Let $f(z)=f_0(z)+f_1(z)E_2(z)+f_2(z)E_2(z)^2$,
 $f_j\in\sM_{k-2j}(\SL(2,\Z))$, be an extremal quasimodular form in
 $\wt\sM_k^{\le 2}(\SL(2,\Z))$. Note that $f_j(z)$ cannot have a
 common zero on $\H$. To see this, say, assume that $f_j(z)$ has a
 common zero at $z_0$. Let $F(z)$ be a modular form of weight $k'$
 with a simple zero at $z_0$ and nonvanishing elsewhere. Then
 $f(z)/F(z)\in \widetilde{\sM}_{k-k'}^{\le 2}(\SL(2,\Z))$ has order $\gauss{k/4}$ at $\infty$, which is impossible
 by \eqref{equation: dim} and the facts that $k'\ge 4$ and that
 extremal quasimodular forms of depth $2$ exist for any weight and
 are unique up to scalars. Therefore, $f_j(z)$ have no common zeros
 on $\H$.

 Now according to Pellarin's argument \cite{PN}, one has $\ord_\infty
 W_f=\ord_\infty f=\gauss{k/4}$. Hence, we have
\[
 W_f(z)=\begin{cases}
 c\Delta(z)^{k/4}, &\text{if }k\equiv 0\mod 4, \\
 c\Delta(z)^{(k-2)/4}E_6(z),
 &\text{if }k\equiv 2\mod 4, \end{cases}
\]
 for some nonzero complex number $c$. Also,
 by Lemma \ref{lemma: local exponents at
 infinity}, we must have $\ord_\infty(f_1+2f_2E_2)=0$
 and the local exponents at $\infty$ must be $-r/3$, $-r/3$, and
 $2r/3$, where $r=\gauss{k/4}$. In other words, the indicial equation
 of \eqref{eq-1} at $\infty$ is
 \begin{equation} \label{equation: indicial in proof}
 x^3-\frac{r^2}3x-\frac{2r^3}{27}=0.
 \end{equation}

 Consider first the case $k\equiv 0\mod 4$. In this case, since
 $\Delta(z)$ has no zeros on $\H$,
 \eqref{eq-1} has no singularities on $\H$. Hence, $Q(z)$
 is a multiple of $E_4(z)$, while $R(z)$ is
 a multiple of $E_6(z)$. In view of \eqref{equation: indicial in
 proof} and Lemma \ref{lemma: indicial at infinity}, we see that
\[
 Q(z)=-\frac{k^2}{48}E_4(z), \qquad R(z)=-\frac{k^3}{864}E_6(z).
\]

 We now consider the case $k\equiv 2\mod 4$. In this case,
 $W_f(z)=c\Delta(z)^{(k-2)/4}E_6(z)$ has a simple zero at ${\rm i}$. Thus,
 the local exponents of \eqref{eq-1} at ${\rm i}$ are $-1/3$,
 $2/3$, and $8/3$ since the differences must be positive integers
 and the sum must be equal to $3$, and the indicial equation at ${\rm i}$
 is
 \begin{equation} \label{equation: indicial at i}
 x^3-3x^2+\frac23x+\frac{16}{27}=0.
 \end{equation}
 We will use this information, together with the apparentness
 property, to determine~$Q(z)$ and~$R(z)$.

 First of all, according to Lemma \ref{lemma: Q R general form}, $Q(z)$ is of the form
\[
 Q(z)=r_\infty E_4(z)+r_{\rm i}^{(2)}\frac{E_4(z)\big(E_4(z)^3-E_6(z)^2\big)}{E_6(z)^2},
 \] while $R(z)$ is of the form
\[
 R(z)=s_\infty E_6(z)+s_{\rm i}^{(3)}\frac{\big(E_4(z)^3-E_6(z)^2\big)^2}{E_6(z)^3}
 +s_{\rm i}^{(1)}\frac{E_4(z)^3-E_6(z)^2}{E_6(z)}
\]
 for some complex numbers $r_\infty$, $r_{\rm i}^{(2)}$, $s_\infty$,
 $s_{\rm i}^{(3)}$, and $s_{\rm i}^{(1)}$.
 The parameters $r_\infty$ and $s_\infty$ are determined by the local exponents
 at $\infty$. As in the case $k\equiv 0\mod 4$, we find that
 $r_\infty=-\frac{(k-2)^2}{48}$ and $s_\infty=-\frac{(k-2)^3}{864}$.
 We now determine the other parameters.

 By \eqref{equation: indicial at i} and Lemma~\ref{lemma: indicial at
 z}, we have
 $r_{\rm i}^{(2)}=-\frac13$ and $ s_{\rm i}^{(3)}=\frac 5{54} $.
 To determine the remaining parameter $s_{\rm i}^{(1)}$, we let
 $w=(z-{\rm i})/(z+{\rm i})$ and recall that, by~\eqref{equation: E4 E6},
\[
 E_4(z)=(1-w)^4\left(B+\frac5{72}B^2u^2+\frac{5}{6912}B^3u^4
 +\cdots\right)
\]
 and
\[
 E_6(z)=(1-w)^6\left(-\frac12B^2u-\frac7{432}B^3u^3
 -\frac7{17280}B^4u^5+\cdots\right),
\]
 where $u=-4\pi w$ and $B=E_4({\rm i})$. (Note that $E_6(z_0)$ and the
 constant~$A$ in Lemma~\ref{lemma: local solutions} are both~$0$ when
 $z_0={\rm i}$.) Then the power series~$\wt Q(x)$ and~$\wt R(x)$ such
 that $Q(z)=(1-w)^4\wt Q(-4\pi w)$ and
 $R(z)=(1-w)^6\wt R(-4\pi w)$
 are
 \begin{equation*}
 \wt Q(x)=\frac{4r_{\rm i}^{(2)}}{x^2}
 +\left(r_\infty-\frac{4r_{\rm i}^{(2)}}{27}\right)B+\cdots
 =-\frac{4}{3x^2}+\left(-\frac{(k-2)^2}{48}
 +\frac4{81}\right)B+\cdots
 \end{equation*}
 and
 \begin{equation*}
 \wt R(x)=-\frac{8s_{\rm i}^{(3)}}{x^3}
 +\left(-2s_{\rm i}^{(1)}+\frac{13}9s_{\rm i}^{(3)}\right) \frac Bx+\cdots
 =-\frac{20}{27x^3}+\left(-2s_{\rm i}^{(1)}+\frac{65}{486}\right)\frac
 Bx+\cdots,
 \end{equation*}
 respectively. By Lemma \ref{lemma: local solutions}, the series~\eqref{eq-11} with $c_0=1$
 is a solution of~\eqref{equation: DE in theorem} if and only if
 the power series $\wt y(x)=\sum_{n=0}^\infty c_nx^{n+\alpha}$
 satisfies~\eqref{equation: DE in wt y}.
 Consider $\wt y(x)$ with $\alpha=2/3$. The coefficients~$c_n$ need
 to satisfy
 \begin{gather*}
 \sum_{n=0}^\infty c_n(n+2/3)(n-1/3)(n-4/3) x^{n-7/3} \\
 \qquad{}+\left(-\frac4{3x^2}
 +\left(\frac4{81}-\frac{(k-2)^2}{48}\right)B+\cdots\right)
 \sum_{n=0}^\infty c_n(n+2/3)x^{n-1/3} \\
 \qquad{}+\left(\frac{16}{27x^3}
 +\left(\frac{65}{486}-2s_{\rm i}^{(1)}\right)\frac Bx+\cdots\right)
 \sum_{n=0}^\infty c_nx^{n+2/3}=0.
 \end{gather*}
 Considering the coefficients of
 $x^{-1/3}$, we find that
 $s_{\rm i}^{(1)}=\frac{12-(k-2)^2}{144}.$
 This completes the proof of the theorem.
\end{proof}

\appendix

\section{The solution structure of third order ODE}\label{section-3}

In this appendix, we apply Frobenius' method to study the solution structure for
 \begin{equation} \label{eq-23}
 \mathcal{L}y:=\frac{{\rm d}^3}{{\rm d}x^3} y(x)+ \mathcal{Q}(x)\frac{{\rm d}}{{\rm d}x} y(x)
 +\mathcal{R}(x)y(x)=0.
 \end{equation}
See, e.g., \cite{Henrici,Hille} for detailed expositions of Frobenius' method.
The following arguments are known to experts in this field. However, since we can not find a suitable reference, we would like to provide all necessary details for later usage.

Suppose $0$ is a regular singular point of (\ref{eq-23}) with three local exponents
\[\kappa_1,\qquad \kappa_2=\kappa_1+m_1,\qquad \kappa_3=\kappa_2+m_2,\qquad \text{where}\quad m_1,m_2\in\mathbb{Z}_{\geq 0},\]
Since the exponent differences are all integers, there might be logarithmic singularities for solutions of~(\ref{eq-23}), or more precisely, the local expansion of some solutions at $x=0$ might contains $\ln x$ terms or even $(\ln x)^2$ terms. This leads us to give the following definition.

\begin{Definition}\quad
\begin{itemize}\itemsep=0pt
\item[(1)] We say (\ref{eq-23}) is apparent at $x=0$ if all solutions have no logarithmic singularities at $x=0$. Otherwise~(\ref{eq-23}) is called not apparent at $x=0$.
\item[(2)] If (\ref{eq-23}) is not apparent at $x=0$ and the local expansion of some solutions contains $(\ln x)^2$ terms, we say~(\ref{eq-23}) is completely not apparent at $x=0$.
\end{itemize}
\end{Definition}

Since $0$ is a regular singular point of (\ref{eq-23}), both $x^2\mathcal{Q}(x)$ and $x^3\mathcal{R}(x)$ are holomorphic at $x=0$, so we may write
\[\mathcal{Q}(x)=\sum_{n=-2}^{\infty}a_n x^{n},\qquad \mathcal{R}(x)=\sum_{n=-3}^{\infty}b_n x^{n}.\]
Let
\[y(x;\alpha)=x^{\alpha}\sum_{n=0}^{\infty}c_n(\alpha)x^n,\qquad\text{where} \quad c_0(\alpha)=1.\]
Then
\begin{align}
\mathcal{L}y(x;\alpha)&=\sum_{n=0}^{\infty}
\bigg(f(\alpha+n)c_n(\alpha)
+\sum_{k=0}^{n-1}[(\alpha+k)a_{n-k-2}+b_{n-k-3}]c_k(\alpha)\bigg)x^{n+\alpha-3}\nonumber\\
&=:x^{\alpha-3}\sum_{n=0}^{\infty}R_n(\alpha)x^n,\label{eq-33}
\end{align}
where
\begin{equation*}
f(s):=s(s-1)(s-2)+sa_{-2}+b_{-3}=\prod_{j=1}^{3}(s-\kappa_j),
\end{equation*}
i.e., $f(s)=0$ is the indicial equation of (\ref{eq-23}) at $x=0$.
Note $R_0(\alpha)=f(\alpha)$.

For any $\alpha$ satisfying $|\alpha-\kappa_3|< 1/2$, we have
\[f(\alpha+n)\neq 0\qquad \text{for any} \ n\geq 1,\]
so by letting
\[
R_{n}(\alpha)=f(\alpha+n)c_n(\alpha)
+\sum_{k=0}^{n-1}[(\alpha+k)a_{n-k-2}+b_{n-k-3}]c_k(\alpha)=0,\qquad n\geq 1,
\]
we see that $c_n(\alpha)$ can be uniquely solved for any $n\geq 1$ such that
\begin{equation}\label{eq-24}
\mathcal{L}y(x;\alpha)=x^{\alpha-3}f(\alpha),\qquad \text{for any} \ |\alpha-\kappa_3|< 1/2.
\end{equation}
Note that $c_n(\alpha)\in \mathbb{C}(\alpha)$ is a rational function of $\alpha$ for any $n\geq 1$.

In particular, letting $\alpha=\kappa_3$ in (\ref{eq-24}) leads to~$\mathcal{L}y(x;\kappa_3)=0$, so

\begin{Lemma}\label{lemma-3-2}
\begin{equation*}%\label{eq-26}
y(x;\kappa_3)=x^{\kappa_3}\sum_{n=0}^{\infty}c_n(\kappa_3)x^n
\end{equation*}
is always a solution of \eqref{eq-23} with the local exponent~$\kappa_3$.
\end{Lemma}

By (\ref{eq-24}), we have
\[\mathcal{L}\frac{\partial y(x;\alpha)}{\partial \alpha}=x^{\alpha-3}f(\alpha)\ln x +x^{\alpha-3}f'(\alpha),\]
so
\begin{equation}\label{eq-27}
\mathcal{L}\frac{\partial y(x;\alpha)}{\partial \alpha}\Big|_{\alpha=\kappa_3}=x^{\kappa_3-3}f'(\kappa_3).
\end{equation}
Similarly,
\begin{equation}\label{eq-28}
\mathcal{L}\frac{\partial^2 y(x;\alpha)}{\partial \alpha^2}\Big|_{\alpha=\kappa_3}=2x^{\kappa_3-3}f'(\kappa_3)\ln x +x^{\kappa_3-3}f''(\kappa_3).
\end{equation}
Note that
\begin{gather}\label{eq-30}
\frac{\partial y(x;\alpha)}{\partial \alpha}\Big|_{\alpha=\kappa_3}=(\ln x) y(x;k_3) +x^{\kappa_3}\sum_{n=1}^{\infty}c_n'(\kappa_3)x^n,
\\
\label{eq-31}
\frac{\partial^2 y(x;\alpha)}{\partial \alpha^2}\Big|_{\alpha=\kappa_3}=(\ln x)^2 y(x;k_3)+2(\ln x)x^{\kappa_3}\sum_{n=1}^{\infty}c_n'(\kappa_3)x^n
+x^{\kappa_3}\sum_{n=1}^{\infty}c_n''(\kappa_3)x^n.
\end{gather}

\begin{Theorem}\label{thm-3-3}
 If $\kappa_1=\kappa_2=\kappa_3$, then $\frac{\partial
 y(x;\alpha)}{\partial \alpha}|_{\alpha=\kappa_3}$ and
 $\frac{\partial^2 y(x;\alpha)}{\partial
 \alpha^2}|_{\alpha=\kappa_3}$ given in \eqref{eq-30} and
 \eqref{eq-31} are the other two linearly independent solutions of
 \eqref{eq-23}, namely \eqref{eq-23} is completely not apparent at
 $x=0$.
\end{Theorem}
\begin{proof}
Since $f(s)=(s-\kappa_3)^3$,
we have $f'(\kappa_3)=f''(\kappa_3)=0$, so this theorem follows from (\ref{eq-27}) and~(\ref{eq-28}).
\end{proof}

Next we consider the case $\kappa_1<\kappa_2=\kappa_3$, i.e., $m_1>0$, $m_2=0$ and $f(s)=(s-\kappa_1)(s-\kappa_3)^2$.
Then $f'(\kappa_3)=0$ and $f''(\kappa_3)\neq 0$, so
$\frac{\partial y(x;\alpha)}{\partial \alpha}|_{\alpha=\kappa_3}$ given in~(\ref{eq-30}) is the second solution of~(\ref{eq-23}), and (\ref{eq-28}) becomes
\begin{equation}\label{eq-38}
\mathcal{L}\frac{\partial^2 y(x;\alpha)}{\partial \alpha^2}\Big|_{\alpha=\kappa_3}=x^{\kappa_3-3}f''(\kappa_3)\neq 0.
\end{equation}

On the other hand, $f(\kappa_1+n)\neq 0$ for $n\in\mathbb{N}\setminus\{m_1\}$. Thus by letting $c_{m_1}(\kappa_1)=0$ and $R_{n}(\kappa_1)=0$ for any $n\in\mathbb{N}\setminus\{m_1\}$ in $\mathcal{L}y(x;\kappa_1)$ (see (\ref{eq-33}) with $\alpha=\kappa_1$), we see that $c_n(\kappa_1)$ can be uniquely solved for any $n\in\mathbb{N}\setminus\{m_1\}$ such that
\begin{equation}\label{eq-34}
\mathcal{L}y(x;\kappa_1)=x^{\kappa_1-3}R_{m_1}(\kappa_1)x^{m_1}=R_{m_1}(\kappa_1)x^{\kappa_3-3},
\end{equation}
where
\begin{equation*}%\label{eq-34-1}
R_{m_1}(\kappa_1)
=\sum_{k=0}^{m_1-1}[(\kappa_1+k)a_{m_1-k-2}+b_{m_1-k-3}]c_k(\kappa_1)
\end{equation*}
is a constant. Thus we obtain

\begin{Theorem}\label{thm-3-4} Suppose
 $\kappa_1<\kappa_2=\kappa_3$. Then $\frac{\partial
 y(x;\alpha)}{\partial \alpha}|_{\alpha=\kappa_3}$ given in~\eqref{eq-30} is the second solution of~\eqref{eq-23}. Furthermore,
\begin{itemize}\itemsep=0pt
\item[$(1)$]
If $R_{m_1}(\kappa_1)=0$, then
\begin{equation*}%\label{eq-36}
y(x;\kappa_1)=x^{\kappa_1}\sum_{n=0}^{\infty}c_n(\kappa_1)x^n
\end{equation*}
is the third solution of \eqref{eq-23} that has the local exponent~$\kappa_1$.

\item[$(2)$] If $R_{m_1}(\kappa_1)\neq 0$, then
 \eqref{eq-38} and \eqref{eq-34} imply that
\begin{align*}
y_3(x):={} &\frac{\partial^2 y(x;\alpha)}{\partial \alpha^2}\Big|_{\alpha=\kappa_3}-\frac{f''(\kappa_3)}{R_{m_1}(\kappa_1)}y(x;\kappa_1)\\
={}&(\ln x)^2 y(x;k_3)+2(\ln x)x^{\kappa_3}\sum_{n=1}^{\infty}c_n'(\kappa_3)x^n\\
&{} +x^{\kappa_3}\sum_{n=1}^{\infty}c_n''(\kappa_3)x^n
-\frac{f''(\kappa_3)}{R_{m_1}(\kappa_1)}y(x;\kappa_1)
\end{align*}
is the third solution of~\eqref{eq-23} that corresponds to the local
exponent $\kappa_1$, namely~\eqref{eq-23} is completely not apparent
at $x=0$.
\end{itemize}
\end{Theorem}

The remaining case is $\kappa_1\leq \kappa_2<\kappa_3$, i.e., $m_2=\kappa_3-\kappa_2>0$.
Then for any $\alpha$ satisfying $|\alpha-\kappa_2|< 1/2$, we have
\[f(\alpha+n)\neq 0\qquad \text{for any} \ n\in\mathbb{N}\setminus\{m_2\},\]
and
\[f(\alpha+n)=0\qquad \text{for $n\geq 1$ if and only if $\alpha=\kappa_2$ and $n=m_2$},\]
so by letting $c_{m_2}(\alpha)=0$ and $R_{n}(\alpha)=0$ for any $n\in\mathbb{N}\setminus\{m_2\}$ in $\mathcal{L}y(x;\alpha)$ (see (\ref{eq-33})), we see that $c_n(\alpha)$ can be uniquely solved for any $n\in\mathbb{N}\setminus\{m_2\}$ such that
\begin{equation}\label{eq-37}
\mathcal{L}y(x;\alpha)=x^{\alpha-3}f(\alpha)+x^{\alpha+m_2-3}R_{m_2}(\alpha),
\qquad \text{for} \ |\alpha-\kappa_2|< 1/2,
\end{equation}
where
\[R_{m_2}(\alpha)=\sum_{k=0}^{m_2-1}[(\alpha+k)a_{m_2-k-2}+b_{m_2-k-3}]c_k(\alpha)\in\mathbb{C}(\alpha),\]
because $c_k(\alpha)\in\mathbb{C}(\alpha)$ for any $n\in\mathbb{N}\setminus\{m_2\}$.

Similarly as before, it follows from (\ref{eq-37}) that
\begin{gather}\label{eq-39}
\mathcal{L}y(x;\kappa_2)=x^{\kappa_3-3}R_{m_2}(\kappa_2),
\\ \label{eq-40}
\mathcal{L}\frac{\partial y(x;\alpha)}{\partial \alpha}\Big|_{\alpha=\kappa_2}=x^{\kappa_2-3}f'(\kappa_2)+x^{\kappa_3-3}R_{m_2}'(\kappa_2)
+(\ln x)x^{\kappa_3-3}R_{m_2}(\kappa_2),
\end{gather}
where
\begin{equation*}%\label{eq-41}
\frac{\partial y(x;\alpha)}{\partial \alpha}\Big|_{\alpha=\kappa_2}=(\ln x) y(x;\kappa_2) +x^{\kappa_2}\sum_{n=1}^{\infty}c_n'(\kappa_2)x^n.
\end{equation*}

\begin{Theorem}\label{thm-3-5} Suppose $\kappa_1=\kappa_2<\kappa_3$.
\begin{itemize}\itemsep=0pt
\item[$(1)$] If $R_{m_2}(\kappa_2)=0$, then \eqref{eq-39} implies that
 $y(x;\kappa_2)$ is the second solution of \eqref{eq-23}, and
 \begin{align*}
y_3(x):={}&\frac{\partial y(x;\alpha)}{\partial \alpha}\Big|_{\alpha=\kappa_2}-\frac{R_{m_2}'(\kappa_2)}{f'(\kappa_3)}\frac{\partial y(x;\alpha)}{\partial \alpha}\Big|_{\alpha=\kappa_3}\\
={}&(\ln x) y(x;\kappa_2) +x^{\kappa_2}\sum_{n=1}^{\infty}c_n'(\kappa_2)x^n\\
&{}-\frac{R_{m_2}'(\kappa_2)}{f'(\kappa_3)}\bigg((\ln x) y(x;\kappa_3) +x^{\kappa_3}\sum_{n=1}^{\infty}c_n'(\kappa_3)x^n\bigg)
\end{align*}
is the third solution of \eqref{eq-23}.
\item[$(2)$] If $R_{m_2}(\kappa_2)\neq 0$, then \eqref{eq-27} and
 \eqref{eq-39} imply that
\begin{gather*}
\frac{\partial y(x;\alpha)}{\partial \alpha}\Big|_{\alpha=\kappa_3}-\frac{f'(\kappa_3)}{R_{m_2}(\kappa_2)}y(x;\kappa_2)\\
\qquad{} =(\ln x) y(x;\kappa_3) +x^{\kappa_3}\sum_{n=1}^{\infty}c_n'(\kappa_3)x^n-\frac{f'(\kappa_3)}{R_{m_2}(\kappa_2)}y(x;\kappa_2)
\end{gather*}
is the second solution of \eqref{eq-23}, and
\begin{align*}
y_3(x):={}&\frac{\partial^2 y(x;\alpha)}{\partial \alpha^2}\Big|_{\alpha=\kappa_3}-
\frac{2f'(\kappa_3)}{R_{m_2}(k_2)}\frac{\partial y(x;\alpha)}{\partial \alpha}\Big|_{\alpha=\kappa_2}\\
&{}-\frac{f''(\kappa_3)R_{m_2}(\kappa_2)-2f'(\kappa_3)R_{m_2}'(\kappa_2)}{R_{m_2}(\kappa_2)^2}y(x;\kappa_2)\\
={}&(\ln x)^2 y(x;k_3)+2(\ln x)x^{\kappa_3}\sum_{n=1}^{\infty}c_n'(\kappa_3)x^n
+x^{\kappa_3}\sum_{n=1}^{\infty}c_n''(\kappa_3)x^n\\
&{}-\frac{2f'(\kappa_3)}{R_{m_2}(k_2)}\bigg((\ln x) y(x;\kappa_2) +x^{\kappa_2}\sum_{n=1}^{\infty}c_n'(\kappa_2)x^n\bigg)\\
&{}-\frac{f''(\kappa_3)R_{m_2}(\kappa_2)-2f'(\kappa_3)R_{m_2}'(\kappa_2)}{R_{m_2}(\kappa_2)^2}y(x;\kappa_2).
\end{align*}
is the third solution of \eqref{eq-23}, namely \eqref{eq-23} is completely not apparent at $x=0$.
\end{itemize}
\end{Theorem}

\begin{proof} Since $\kappa_1=\kappa_2<\kappa_3$, i.e., $m_1=0$, $m_2>0$ and $f(s)=(s-\kappa_2)^2(s-\kappa_3)$,
so $f'(\kappa_2)=0$ and $f'(\kappa_3)\neq 0$.

(1) Note that (\ref{eq-40}) becomes
\[
\mathcal{L}\frac{\partial y(x;\alpha)}{\partial \alpha}\Big|_{\alpha=\kappa_2}=x^{\kappa_3-3}R_{m_2}'(\kappa_2),
\]
we see from (\ref{eq-27}) that $\mathcal{L}y_3=0$.

(2) Note that (\ref{eq-40}) becomes
\begin{equation}\label{eq-42}
\mathcal{L}\frac{\partial y(x;\alpha)}{\partial \alpha}\Big|_{\alpha=\kappa_2}=x^{\kappa_3-3}R_{m_2}'(\kappa_2)
+(\ln x)x^{\kappa_3-3}R_{m_2}(\kappa_2).
\end{equation}
Then (\ref{eq-28}), (\ref{eq-39}) and (\ref{eq-42}) together imply $\mathcal{L}y_3=0$.
\end{proof}

Finally, we consider the last case $\kappa_1<\kappa_2<\kappa_3$, i.e., $m_1>0$ and $m_2>0$.
Then $f'(\kappa_j)\neq 0$ for all $j$. Since $f(\kappa_1+n)=0$ for $n\geq 1$ if and only if $n\in \{m_1, m_1+m_2\}$, by letting $c_{m_1}(\kappa_1)=c_{m_1+m_2}=0$ and $R_{n}(\kappa_1)=0$ for any $n\in\mathbb{N}\setminus\{m_1,m_1+m_2\}$ in $\mathcal{L}y(x;\kappa_1)$ (see (\ref{eq-33}) with $\alpha=\kappa_1$), we see that $c_n(\kappa_1)$ can be uniquely solved for any $n\in\mathbb{N}\setminus\{m_1,m_1+m_2\}$ such that
\begin{equation}\label{eq-34-0}
\mathcal{L}y(x;\kappa_1)=R_{m_1}(\kappa_1)x^{\kappa_2-3}+R_{m_1+m_2}(\kappa_1)x^{\kappa_3-3},
\end{equation}
where
\begin{gather*}
R_{m_1}(\kappa_1)=\sum_{k=0}^{m_1-1}[(\kappa_1+k)a_{m_1-k-2}+b_{m_1-k-3}]c_k(\kappa_1),\\
R_{m_1+m_2}(\kappa_1)=\sum_{k=0}^{m_1+m_2-1}[(\kappa_1+k)a_{m_1+m_2-k-2}+b_{m_1+m_2-k-3}]c_k(\kappa_1),
\end{gather*}
are constants.

\begin{Theorem}\label{thm-3-6} Suppose $\kappa_1<\kappa_2<\kappa_3$ and $R_{m_2}(\kappa_2)=0$. Then \eqref{eq-39} implies that $y(x;\kappa_2)$ is the second solution of \eqref{eq-23}. Furthermore,
\begin{itemize}\itemsep=0pt
\item[$(1)$] If $R_{m_1}(\kappa_1)=R_{m_1+m_2}(\kappa_1)=0$, then~\eqref{eq-34-0} implies that $y(x;\kappa_3)$ is the third solution
 of~\eqref{eq-23}, namely $0$ is an apparent singularity of~\eqref{eq-23}.
\item[$(2)$] If $R_{m_1}(\kappa_1)=0$ and $R_{m_1+m_2}(\kappa_1)\neq 0$,
 then
\begin{align*}
y_3(x):={}&\frac{\partial y(x;\alpha)}{\partial \alpha}\Big|_{\alpha=\kappa_3}-\frac{f'(\kappa_3)}{R_{m_1+m_2}(\kappa_1)}y(x;\kappa_1)\\
={} &(\ln x) y(x;\kappa_3) +x^{\kappa_3}\sum_{n=1}^{\infty}c_n'(\kappa_3)x^n-\frac{f'(\kappa_3)}{R_{m_1+m_2}(\kappa_1)}y(x;\kappa_1)
\end{align*}
is the third solution of \eqref{eq-23}.
\item[$(3)$] If $R_{m_1}(\kappa_1)\neq 0$, then
\begin{align*}
y_3(x)
:={} &\frac{\partial y(x;\alpha)}{\partial \alpha}\Big|_{\alpha=\kappa_2}-
\frac{f'(\kappa_2)}{R_{m_1}(k_1)}y(x;\kappa_1)\\
&{} -\frac{R_{m_1}(\kappa_1)R_{m_2}'(\kappa_2)-f'(\kappa_2)R_{m_1+m_2}(\kappa_1)}{f'(\kappa_3)R_{m_1}(\kappa_1)}\frac{\partial y(x;\alpha)}{\partial \alpha}\Big|_{\alpha=\kappa_3}\\
={} &(\ln x) y(x;\kappa_2) +x^{\kappa_2}\sum_{n=1}^{\infty}c_n'(\kappa_2)x^n-
\frac{f'(\kappa_2)}{R_{m_1}(k_1)}y(x;\kappa_1)-\\
&\frac{R_{m_1}(\kappa_1)R_{m_2}'(\kappa_2)-f'(\kappa_2)R_{m_1+m_2}(\kappa_1)}
{f'(\kappa_3)R_{m_1}(\kappa_1)}\bigg((\ln x) y(x;\kappa_3) +x^{\kappa_3}\sum_{n=1}^{\infty}c_n'(\kappa_3)x^n\bigg)
\end{align*}
is the third solution of \eqref{eq-23}.
\end{itemize}
\end{Theorem}

\begin{proof}Note that (\ref{eq-40}) becomes
\begin{equation}\label{eq-44}
\mathcal{L}\frac{\partial y(x;\alpha)}{\partial \alpha}\Big|_{\alpha=\kappa_2}=x^{\kappa_2-3}f'(\kappa_2)+x^{\kappa_3-3}R_{m_2}'(\kappa_2).
\end{equation}

(2) Note that (\ref{eq-34-0}) becomes
\begin{equation}\label{eq-45}\mathcal{L}y(x;\kappa_1)=R_{m_1+m_2}(\kappa_1)x^{\kappa_3-3}\neq 0.\end{equation}
From here and (\ref{eq-27}), we easily obtain $\mathcal{L}y_3=0$.

(3) Similarly, it is easy see from (\ref{eq-27}), (\ref{eq-34-0}) and (\ref{eq-44}) that $\mathcal{L}y_3=0$.
\end{proof}

Similarly, we can obtain

\begin{Theorem}\label{thm-3-7} Suppose $\kappa_1<\kappa_2<\kappa_3$ and $R_{m_2}(\kappa_2)\neq 0$. Then
 \eqref{eq-27} and \eqref{eq-39} imply that
\begin{align*}
\frac{\partial y(x;\alpha)}{\partial \alpha}\Big|_{\alpha=\kappa_3}-\frac{f'(\kappa_3)}{R_{m_2}(\kappa_2)}y(x;\kappa_2)
 = (\ln x) y(x;\kappa_3) +x^{\kappa_3}\sum_{n=1}^{\infty}c_n'(\kappa_3)x^n-\frac{f'(\kappa_3)}{R_{m_2}(\kappa_2)}y(x;\kappa_2)
\end{align*}
is the second solution of \eqref{eq-23}. Furthermore,
\begin{itemize}\itemsep=0pt
\item[$(1)$] If $R_{m_1}(\kappa_1)=R_{m_1+m_2}(\kappa_1)=0$, then
\eqref{eq-34-0} implies that $y(x;\kappa_3)$ is the third solution of~\eqref{eq-23}.
\item[$(2)$] If $R_{m_1}(\kappa_1)=0$ and $R_{m_1+m_2}(\kappa_1)\neq 0$, then \eqref{eq-39} and \eqref{eq-45} imply that
\[y(x;\kappa_1)-\frac{R_{m_1+m_2}(\kappa_1)}{R_{m_2}(\kappa_2)}y(x;\kappa_2)\]
is the third solution of \eqref{eq-23}.
\item[$(3)$] If $R_{m_1}(\kappa_1)\neq 0$, then \eqref{eq-28}, \eqref{eq-39}, \eqref{eq-40} and \eqref{eq-34-0} imply that
\begin{align*}
y_3(x):={}
&\frac{\partial^2 y(x;\alpha)}{\partial \alpha^2}\Big|_{\alpha=\kappa_3}-\frac{2f'(\kappa_3)}{R_{m_2}(\kappa_2)}\frac{\partial y(x;\alpha)}{\partial \alpha}\Big|_{\alpha=\kappa_2}+C_1y(x;\kappa_1)
-C_2y(x;\kappa_2)\\
={} &(\ln x)^2 y(x;\kappa_3)+2(\ln x)x^{\kappa_3}\sum_{n=1}^{\infty}c_n'(\kappa_3)x^n
+x^{\kappa_3}\sum_{n=1}^{\infty}c_n''(\kappa_3)x^n\\
&{} -\frac{2f'(\kappa_3)}{R_{m_2}(k_2)}\bigg((\ln x) y(x;\kappa_2) +x^{\kappa_2}\sum_{n=1}^{\infty}c_n'(\kappa_2)x^n\bigg)+C_1y(x;\kappa_1)
-C_2y(x;\kappa_2)
\end{align*}
is the third solution of \eqref{eq-23}, where
\begin{gather*}
C_1:=\frac{2f'(\kappa_3)f'(\kappa_2)}{R_{m_2}(\kappa_2)R_{m_1}(\kappa_1)},\\
C_2:=\frac{1}{R_{m_2}(\kappa_2)}\left[f''(\kappa_3)-
\frac{2f'(\kappa_3)R_{m_2}'(\kappa_2)}{R_{m_2}(\kappa_2)}
+\frac{2f'(\kappa_3)f'(\kappa_2)R_{m_1+m_2}(\kappa_1)}{R_{m_2}(\kappa_2)R_{m_1}(\kappa_1)}\right].
\end{gather*}
In particular, \eqref{eq-23} is completely not apparent at $x=0$.
\end{itemize}
\end{Theorem}

\begin{Remark}\label{rmkk} It follows from Theorem \ref{thm-3-6}(1) that $0$ can be apparent only for the case $\kappa_1<\kappa_2<\kappa_3$.
\end{Remark}

\begin{Remark}\label{rmk-apparent}
Clearly all the above arguments work when we study whether the regular singularity $\infty$ is apparent or not for
\begin{equation}\label{eq-60-0}y'''(z)+Q_2(z)y'(z)+Q_3(z)y(z)=0,\qquad z\in\mathbb{H},\end{equation}
when the local exponents $\kappa_\infty^{(1)}\leq\kappa_{\infty}^{(2)}\leq\kappa_{\infty}^{(3)}$ satisfy $\kappa_{\infty}^{(j)}-\kappa_{\infty}^{(j-1)}\in\mathbb{Z}$. Since $Q_j(z)$'s have Fourier expansions in terms of $q_N={\rm e}^{\frac{2\pi {\rm i}z}{N}}$ (where $N$ is the width of the cusp $\infty$ on $\Gamma$ and $N=1$ for $\Gamma=\SL(2,\mathbb{Z})$),
this is equivalent to whether the regular singularity $q_N=0$ is apparent or not for
\begin{equation}\label{eq-60}\left(q_N\frac{{\rm d}}{{\rm d}q_N}\right)^3y+\frac{N^2}{(2\pi {\rm i})^2}Q_2q_N\frac{{\rm d}}{{\rm d}q_N}y+\frac{N^3}{(2\pi {\rm i})^3}Q_3y=0.\end{equation}
All the above statements are true for (\ref{eq-60}) in terms of $q_N$. In particular, (\ref{eq-60-0}) or equivalently~(\ref{eq-60}) always has a solution of the form
\begin{equation}\label{eq-60-1}y_+(z):=q_N^{\kappa_{\infty}^{(3)}}\sum_{n=0}^{\infty}c_n(\kappa_{\infty}^{(3)})q_N^n,\quad c_0=1,\end{equation}
and (\ref{eq-60-0}) is completely not apparent at $z=\infty$ or equivalently (\ref{eq-60}) is completely not apparent at $q_N=0$ if and only if the local expansion of some solutions in terms of $q_N$ contains the term $z^2y_+(z)$ because of $\ln q_N=2\pi {\rm i} z/N$.
More precisely, if (\ref{eq-60-0}) is completely not apparent at
$z=\infty$, then it follows from Theorems~\ref{thm-3-3},
\ref{thm-3-4}(2), \ref{thm-3-5}(2) and Theorem~\ref{thm-3-7}(3)
that~(\ref{eq-60-0}) has two solutions of the following form
\begin{gather}
 y_-(z) :=z^2y_+(z)+z\eta_1(z)+\eta_2(z),\qquad
 y_\perp(z) :=zy_+(z)+\eta_3(z),\label{eq-60-2}
 \end{gather}
such that $(y_-,y_\perp,y_+)^t$ is a basis of solutions, where
\[\eta_1(z)=q_N^{\kappa_{\infty}^{(2)}}\sum_{n=0}^{\infty}c_{n,1}q_N^{n},\qquad \eta_2(z)=q_N^{\kappa_{\infty}^{(1)}}\sum_{n=0}^{\infty}c_{n,2}q_N^{n},
\qquad\eta_3(z)=q_N^{\kappa_{\infty}^{(2)}}\sum_{n=0}^{\infty}c_{n,3}q_N^{n}.\]
Note that $c_{0,j}=0$ may happen for any $j$; see Theorem~\ref{thm-3-3} for example.
\end{Remark}

\subsection*{Acknowledgements}

We would like to thank the referees for many valuable comments and pointing out some references. The research of Z.~Chen was supported by NSFC (No.~12071240).

\pdfbookmark[1]{References}{ref}
\LastPageEnding

\end{document}